\documentclass{article}

\usepackage[margin=1in]{geometry}
\usepackage[margin=1cm,font=footnotesize]{caption}
\usepackage{afterpage}
\usepackage{url}

\usepackage{amsmath,mathtools,xfrac}
\usepackage{algorithm,algpseudocode}
\usepackage{optidef}
\def\argmin{\mathop{\rm arg\,min}}

\usepackage{amsfonts,amssymb,bbm}

\usepackage{graphicx,overpic,float}

\usepackage{color,xcolor}
\definecolor{prostate}{rgb}{0,0.4470,0.7410}
\definecolor{nodes}{rgb}{0.4940,0.1840,0.5560}
\definecolor{rectum}{rgb}{0.8500,0.3250,0.0980}
\definecolor{bladder}{rgb}{0.9290,0.6940,0.1250}
\definecolor{femheads}{rgb}{0.6350,0.0780,0.1840}

\usepackage{amsthm}
\newtheorem{theorem}{Theorem}
\newtheorem{lemma}{Lemma}
\newtheorem{definition}{Definition}

\begin{document}

\title{A nonconvex optimization approach to IMRT planning with dose--volume constraints}
\author{Kelsey Maass \\ \small{Department of Applied Mathematics, University of Washington} \\[1ex]
Minsun Kim \\ \small{Department of Radiation Oncology, University of Washington} \\[1ex]
Aleksandr Aravkin \\ \small{Department of Applied Mathematics, University of Washington}}
\date{}

\maketitle

\abstract{Fluence map optimization for intensity-modulated radiation therapy planning can be formulated as a large-scale inverse problem with competing objectives and constraints associated with the tumors and organs-at-risk.
Unfortunately, clinically relevant dose--volume constraints are nonconvex, so standard algorithms for convex problems cannot be directly applied.
While prior work focused on convex approximations for these constraints, we propose a novel relaxation approach to handle nonconvex dose--volume constraints.
We develop efficient, provably convergent algorithms based on partial minimization, and show how to adapt them to handle maximum-dose constraints and infeasible problems.
We demonstrate our approach using the CORT dataset, and show that it is easily adaptable to radiation treatment planning with dose--volume constraints for multiple tumors and organs-at-risk.}

\section{Introduction}
\label{sec:intro}

Cancer is one of the leading causes of death worldwide, with approximately 38.4\% of the population expected to develop cancer at some point in their lives \cite{nci2018stats}.
Many cancer therapies target the weakened DNA damage response of cancer cells; in particular, radiotherapy uses high-energy ionizing radiation to damage cancer cell DNA \cite{o2015targeting, nci2019radiation}.
Unfortunately, this radiation also damages the healthy cells in its path, so the main goal of radiation treatment planning is to maximize the differential between damage to the tumor and to nearby healthy tissue.
The most common form of radiotherapy is external beam radiation therapy (EBRT), where photon beams generated by a linear accelerator are directed at a patient from a rotating gantry \cite{nci2018external}.
By surrounding a tumor with beams from multiple angles, physicians can design treatment plans that target tumors with a variety of geometries while avoiding a large dose to critical structures. 

One way to limit damage to healthy tissue is to require that treatment plans satisfy various constraints on the amount of radiation delivered to organs-at-risk (OARs).
For example, it is often clinically relevant to impose maximum-dose constraints on a fraction of an organ, known as a {\it dose--volume constraint}.
Unfortunately, these constraints are nonconvex, so convex formulations and algorithms cannot be directly applied to the treatment planning problem.
We propose a novel approach to handle dose--volume constraints for intensity-modulated radiation therapy (IMRT) while preserving their nonconvex structure, as opposed to previous efforts that focused on convex approximations.
The proposed method is amenable to efficient algorithms based on partial minimization and naturally adapts to handle maximum-dose constraints and cases of infeasiblity.
We demonstrate our approach using the CORT dataset \cite{craft2014shared, craft2014shared2}, and we show that it easily adapts to radiation treatment planning with dose--volume constraints for multiple tumors and OARs.
To start off, we give a brief overview of EBRT, including IMRT, and the specific challenge posed by dose--volume constraints.

\subsection{External beam radiation therapy}

Conventional EBRT began soon after the discovery of X-rays in 1895 and was significantly impacted by technological innovations such as advanced imaging and multileaf collimators (MLC) \cite{bortfeld2006imrt}.
Specifically, the invention of computed tomography (CT) in the 1970s enabled accurate three-dimensional mapping of a patient's internal anatomy, while the introduction of the MLC to radiotherapy in the 1990s provided more control over beam shape and fluence.
For example, three-dimensional conformal radiation therapy (3D-CRT) uses the sliding leaves of the MLC to shape beam outlines to match tumor geometries, reducing radiation exposure to healthy tissue, while IMRT uses MLCs to not only shape beam outlines, but also modulate beam fluence, giving additional control over the patient dose distribution \cite{webb2003physical}.

In the early history of radiotherapy, treatment plans were created using a trial-and-error approach called {\it forward planning}, where physicians would calculate dose distributions for different sets of beam parameters until an acceptable treatment plan was found.
However, modulating beam fluence for IMRT involves discretizing each beam into subunits, called beamlets or bixels, which dramatically increases the number of variables under consideration, rendering forward planning infeasible \cite{bortfeld1999optimized, ehrgott2010mathematical}.
In the 1980s, the idea of {\it inverse planning} for IMRT was introduced by Brahme \cite{brahme1982solution, brahme1988optimization}, where mathematical optimization was used to calculate beam parameters to deliver a specified dose distribution.

Within radiotherapy inverse planning, there are a variety of different paradigms regarding dose objectives, mathematical models, and optimization methods.
For instance, while our proposed model and many others seek to optimize dose distributions defined in terms of {\it physical criteria} such as dose and volume \cite{bortfeld1999optimized, chui2001inverse}, it is also possible to consider {\it biological criteria} such as tumor control probability (TCP) and normal-tissue complication probability (NTCP) \cite{alber1999objective}, equivalent uniform dose (EUD) \cite{thieke2003physical, olafsson2005optimization, craft2007approach}, and biologically effective dose (BED) \cite{saberian2016optimal, saberian2016theoretical}.
Additionally, many early approaches to inverse planning used analytical {\it direct methods} inspired by the similarities between IMRT planning and CT image reconstruction \cite{cormack1987problem, bortfeld1990methods, holmes1991unified}, but the trend has since moved towards approximate {\it iterative methods} \cite{xing1996iterative, xing1998fast, shepard2000iterative}.

Amongst iterative methods, the presence of nonconvexity within many IMRT models has fueled the debate between {\it stochastic methods} such as simulated annealing \cite{webb1989optimisation, webb1992optimization, mageras1993application} and genetic algorithms \cite{wu2000selection, ahmad2010genetic}, which can escape local minima but tend to have higher computational costs, versus {\it gradient-based methods} using techniques such as Newton's method \cite{bortfeld1990methods, wu2000algorithms}, L-BFGS \cite{lahanas2003multiobjective, pflugfelder2008comparison}, and conjugate gradient methods \cite{spirou1998gradient, alber2002degeneracy}, which tend to be faster but can get trapped in local minima.
Our method uses a projected gradient approach to handle issues related to nonconvexity, which we discuss in Section~\ref{sec:dvc}.
For a more detailed history of radiotherapy and IMRT, see \cite{bortfeld2006imrt, webb2003physical, bortfeld1999optimized, ehrgott2010mathematical, chui2001inverse, shepard1999optimizing}.

\subsection{Intensity-modulated radiation therapy}

In current practice, the patient undergoes imaging (e.g., CT) before IMRT treatment, producing a three-dimensional model of their internal anatomy.
This model is discretized into volume units (voxels) and labeled by structure.
Next, the physician specifies the prescription, often in terms of a uniform target to the tumor with various constraints on critical structures.
The goal of the treatment planner is then to determine the number of beams, beam angles, and beamlet intensities that will deliver the prescribed dose of radiation to the tumor while keeping doses to healthy tissue low.
Inverse planning for IMRT typically involves three steps \cite{ehrgott2010mathematical}:
\begin{enumerate}
    \item Beam angles: Determine the number of beams and their orientations.
    \item Fluence map: Calculate beamlet intensities to deliver the desired dose distribution to the patient.
    \item Segmentation: Design MLC sequences that achieve the optimized fluence map.
\end{enumerate}
These steps require solving optimization problems that can be approached separately or in combination \cite{lee2003integer, ahmed2010automated, zhang2010two, bertsimas2013hybrid}.
In this paper, we focus on the fluence map optimization (FMO) problem of assigning beamlet intensities for a given set of beam angles.

The FMO problem has been approached using a number of different mathematical frameworks including linear programming \cite{langer1987optimization, zaghian2014automatic, kishimoto2018successive}, piecewise linear--quadratic models \cite{cho1998optimization, romeijn2006new, fu2019convex}, quadratic programming \cite{wu2000algorithms, spirou1998gradient, cotrutz2002using}, and other nonlinear programming models \cite{alber1999objective, thieke2003physical, shepard2000iterative, llacer2001comparative}.
Using multi-objective optimization, many studies have generated solutions along the Pareto frontier to explore the trade-off between competing tumor and healthy tissue objectives and constraints \cite{craft2007approach, lahanas2003multiobjective, halabi2006dose}.
Other approaches include formulating the FMO problem as a dynamical system \cite{hou2003optimization} or eigenvalue problem \cite{crooks2001linear}, using a database of reference dose--volume histograms to guide optimization \cite{zarepisheh2014dvh}, and clustering beams and voxels to handle large-scale problems \cite{ungun2019real}.
We develop a new approach that is well adapted to handle nonconvex dose--volume constraints, in addition to convex maximum-dose and mean-dose constraints. We first explain all three constraint types. 

Radiation dose is measured in gray (Gy), where one Gy is defined as the absorption of one joule of energy per kilogram of matter.
Although lower doses are always better for healthy tissue, tolerable doses for many tissues are known empirically and depend on the risk level that a patient and physician are willing to accept.
In practice, tolerable doses for different tissues are formulated using the following three types of constraints:
\begin{itemize}
    \item Maximum dose: No voxel in the organ receives more than $d^\text{max}$ Gy.
    \item Mean dose: Average radiation per voxel does not exceed $d^\text{mean}$ Gy.
    \item Dose--volume: At most $p\%$ of the organ volume receives more than $d^\text{dv}$ Gy. 
\end{itemize}
Serial structures such as the spinal cord and brainstem lose functionality if any of their subvolumes are damaged, so maximum-dose constraints are a good indicator of tissue damage.
Maximum-dose and mean-dose constraints define a {\it convex} feasible region for the decision variables.
The associated optimization problems can be solved at scale; for example, inverse planning with linear inequality constraints can be formulated as a convex quadratic program, with state-of-the-art approaches and commercial implementations readily available (see e.g., \cite{gill2005snopt}).
On the other hand, parallel structures such as the liver and lungs are composed of semi-independent units.
In this case, a fraction of the parallel organ can be damaged without loss of functionality, so dose--volume constraints are appropriate \cite{spirou1998gradient}.
Assigning maximum-dose constraints on partial volumes to achieve the desired dose--volume criteria was first introduced by Langer and Leong in 1987 \cite{langer1987optimization} and brought to the attention of the mathematical community by Shepard et al. in 1999 \cite{shepard1999optimizing}.

\subsection{Challenge of dose--volume constraints}
\label{sec:dvc}

In contrast to maximum-dose and mean-dose constraints, dose--volume constraints are {\it combinatorial}, since the choice of the $p\%$ of organ voxels that may receive more than $d^\text{dv}$ Gy is left to the planner. 
In early IMRT research, using only a small number of beamlets and voxels, it was possible to solve the treatment planning problem for all possible voxel combinations satisfying a dose--volume constraint and then choose the best resulting plan \cite{langer1987optimization}.
This approach becomes infeasible as the problem size increases to clinically relevant situations: for an organ with $n$ voxels and a dose--volume constraint imposed on $m$ of these voxels, there are $\sfrac{n!}{(n-m)!m!}$ possible voxel combinations that would satisfy the constraint \cite{spirou1998gradient}.

To illustrate the nonconvex geometry of a feasible beamlet region induced by a dose--volume constraint, we recreate an example from Wu et al. \cite{wu2003method} using the CORT dataset described in Section~\ref{sec:results}.
We consider three body voxels of a patient, where voxel $u$ belongs to a prostate tumor and voxels $v_1$ and $v_2$ belong to the rectum (left panel of Figure~\ref{fig:simple}).
Using two beamlets $x_1$ and $x_2$, we aim to deliver a uniform dose of 81 Gy to the tumor while ensuring that {\it no more than 50\% of the rectum volume (one voxel in this case) exceeds 20 Gy.} 

We calculate the optimal beamlet intensities by minimizing the objective function 
\begin{equation}
    \label{eq:simpleObj}
    \frac{1}{2} \| A_1x - 81 \|_2^2 + \frac{\lambda}{2} \| x \|_2^2
\end{equation}
subject to the dose--volume constraint
\begin{equation}
    \label{eq:simpleConst}
    \| (A_2x - 20)_+ \|_0 \leq 1.
\end{equation}
The matrices $A_1 \in \mathbb{R}^{1 \times 2}$ and $A_2 \in \mathbb{R}^{2 \times 2}$ map beamlet intensities to prostate and rectum voxel doses, the vector $x \in \mathbb{R}^2$ contains the beamlet intensities, and the regularization term controlled by $\lambda = 5 \times 10^{-6}$ is added to stabilize beamlet solutions, amortizing the ill-conditioned beamlet-to-voxel map.
The dose--volume constraint~\eqref{eq:simpleConst} forces the number of rectum voxels receiving a dose greater than 20 Gy to be less than or equal to 1:
\begin{equation}
    \label{eq:dvcLong}
    \| (A_2 x - 20)_+ \|_0 = \|\max\{0,A_2x - 20\}\|_0 = \text{number of voxels s.t. } A_2x > 20,
\end{equation}
where the $\ell_0$-norm returns the number of nonzero elements in a vector (e.g., \cite{penfold2017sparsity}).
In Figure~\ref{fig:simple} (right), we plot the feasible beamlet region along with the contours of the objective function.
Since negative beamlet intensities are impossible, we restrict the beamlet values to the first quadrant. 

\begin{figure}
    \centering
    \begin{minipage}[c]{0.45\textwidth}
        \begin{overpic}[scale=0.45]{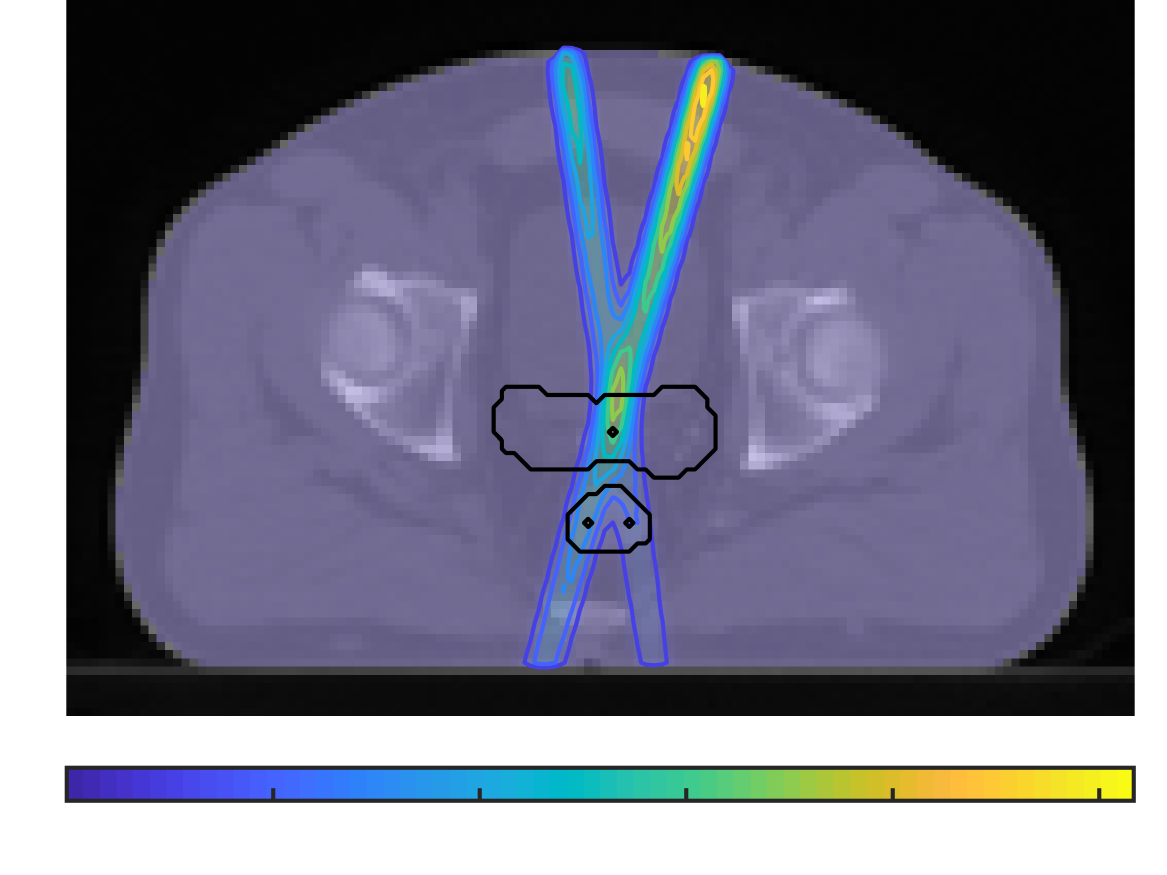}
            \put(43.3,30){\color{white}$v_1$}
            \put(56.4,30){\color{white}$v_2$}
            \put(46.5,37){\color{white}$u$}
            \put(46,72.15){\color{white}$x_2$}
            \put(60,71.5){\color{white}$x_1$}
            \put(5,3){\scriptsize0}
            \put(21.5,3){\scriptsize20}
            \put(39,3){\scriptsize40}
            \put(57,3){\scriptsize60}
            \put(74.5,3){\scriptsize80}
            \put(91.5,3){\scriptsize100}
            \put(43,-1){\scriptsize Dose (Gy)}
        \end{overpic}
    \end{minipage}
    \begin{minipage}[c]{0.45\textwidth} \hspace{5em}
        \begin{overpic}[scale=0.4]{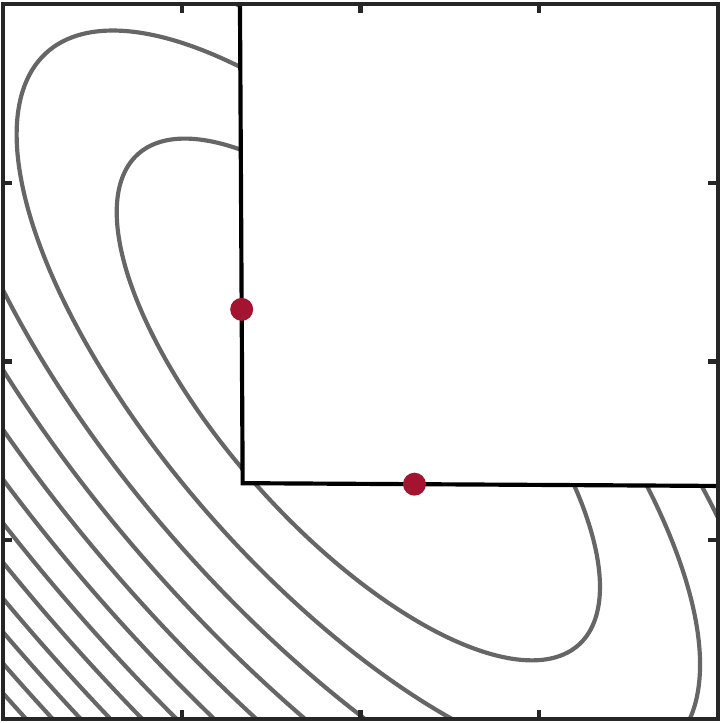}
            \put(-1,-5){\scriptsize0}
            \put(21,-5){\scriptsize0.5}
            \put(48.5,-5){\scriptsize1}
            \put(71,-5){\scriptsize1.5}
            \put(97,-5){\scriptsize2}
            \put(100,-2){\tiny$\times10^4$}
            \put(-4,-1){\scriptsize0}
            \put(-9,23){\scriptsize0.5}
            \put(-4,48.5){\scriptsize1}
            \put(-9,73){\scriptsize1.5}
            \put(-4,98){\scriptsize2}
            \put(-1,101){\tiny$\times10^4$}
            \put(36,55){\small A}
            \put(55,36){\small B}
            \put(35.5,77){\rotatebox{90}{\footnotesize$v_1\leq20$}}
            \put(77,36){\footnotesize$v_2\leq20$}
            \put(47,-11){\small$x_1$}
            \put(-13,46){\rotatebox{90}{\small$x_2$}}
        \end{overpic} \vspace{6ex}
    \end{minipage} \vspace{2ex}
    \caption{Left: Three body voxels, one belonging to a prostate tumor ($u$) and two belonging to the rectum ($v_1$ and $v_2$), are irradiated by two beamlets ($x_1$ and $x_2$).
    The beamlet intensities shown correspond to the global minimum at point B.
    Right: The nonconvex feasible region induced by the dose--volume constraint that no more than 50\% of the rectum volume exceeds 20 Gy is explicitly given by the union of regions $v_1 \leq 20$ and $v_2 \leq 20$.
    The contours indicate the isolines of objective function~\eqref{eq:simpleObj}, with local minimum at point A and global minimum at point B.\label{fig:simple}}
\end{figure}
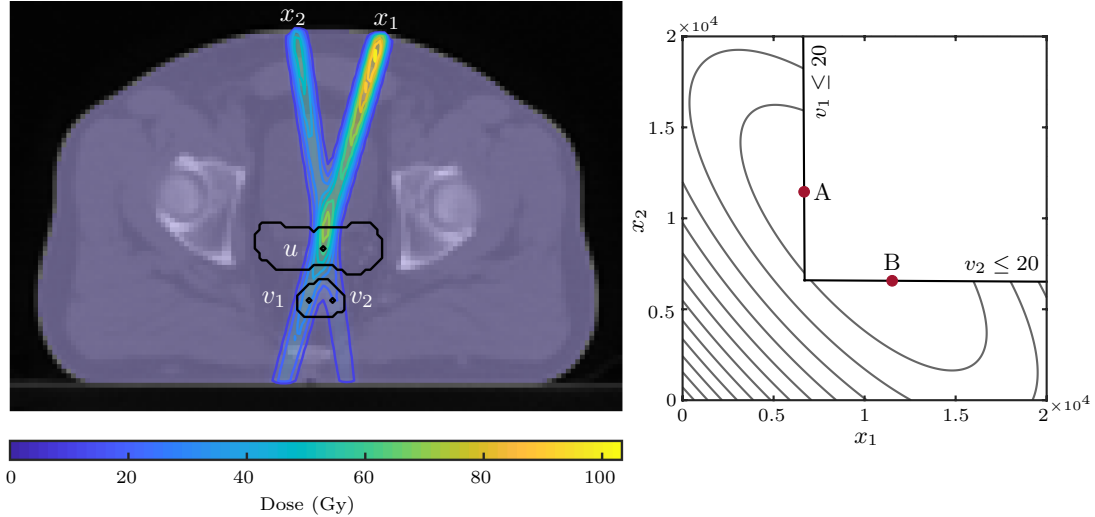

There are two possibilities for satisfying the dose--volume constraint: either the first rectum voxel is allowed to exceed 20 Gy (Figure~\ref{fig:simple}, $v_2 \leq 20$ region), or the second rectum voxel is allowed to exceed 20 Gy (Figure~\ref{fig:simple}, $v_1 \leq 20$ region).
The union of these two regions forms an L-shaped nonconvex region with two local minima at points A and B.
Depending on the beamlet initialization and algorithm used to minimize the objective function, it is possible to converge to either of these two local minima, and there is no guarantee that the global minimum at B will be reached.
In general, the feasible beamlet region imposed by a dose--volume constraint comprises the nonconvex union of regions satisfying maximum-dose constraints for different choices of voxels \cite{deasy1997multiple}. 

The FMO problem with nonconvex dose--volume constraints is related to the general convex-cardinality problem 
\begin{mini!}[2] 
    {x\in\mathcal{C}}{f(x)}{}{}
    \addConstraint{\|x\|_0}{\leq k}{\label{eq:cardConst}},
\end{mini!}
where $f(x)$ is a convex function and $\mathcal{C}$ is a convex set.
The $\ell_0$-norm constraint defines a highly nonconvex set: the \emph{union} of all subspaces of dimension at most $k$.
Problems with these cardinality (or sparsity) constraints have become widely used within the past three decades \cite{tropp2010computational}, with applications such as portfolio optimization \cite{lobo2007portfolio}, signal processing \cite{candes2006stable}, and compressive sensing \cite{candes2008enhancing}.
Due to the combinatorial nature of these constraints, the related feasibility problem is NP-complete \cite{bienstock1996computational}, and the problem itself is known to be NP-hard in general \cite{bienstock1996computational, welch1982algorithmic}, except for in special cases \cite{gao2013polynomial}.
More generally, it is an NP-hard problem just to  certify whether a point of a nonlinear function is a local minimum (let alone a global one)~\cite{murthy1987some}.
For a more detailed overview of cardinality problems, see \cite{tropp2010computational, bruglieri2006annotated, beck2013sparsity, sun2013recent}.

Many of the standard approaches for solving the convex-cardinality problem have been applied to the FMO problem with nonconvex dose--volume constraints, most relying upon convex relaxations or approximations.
One technique is to reformulate both problems as a mixed-integer programming (MIP) problem, which can be solved using branch-and-bound or branch-and-cut methods \cite{lee2003integer, bienstock1996computational, langer1990large}, simplex methods using big $M$ variables \cite{zhang2017mixed}, or by replacing the binary constraint $x \in \{0,1\}^n$ with the linear constraint $x \in [0,1]^n$ \cite{halabi2006dose, burdakov2016mathematical, mukherjee2019integrating}.
Another popular technique for solving the convex-cardinality problem is to relax the nonconvex $\ell_0$-norm constraint~\eqref{eq:cardConst} using the convex $\ell_1$-norm as either a constraint or regularizer \cite{tropp2010computational, boyd2004convex, tibshirani1996regression}.
Convex approximations to the nonconvex dose--volume constraint $\|(Ax - d^\text{dv})_+\|_0 \leq k$ include the convex constraint $\| (1 + \lambda(Ax - d^\text{dv}))_+ \|_1 \leq k$ \cite{fu2019convex} and conditional value-at-risk (CVaR) constraints \cite{rockafellar2000optimization} that limit the average dose to the tail of the dose distribution \cite{ahmed2010automated, kishimoto2018successive, romeijn2006new, romeijn2003novel}.
One common dose--volume penalty includes the convex objective term $\lambda \|(Ax - d^\text{dv})_+\|_2^2$ \cite{ehrgott2010mathematical, bortfeld1990methods, wu2000algorithms, spirou1998gradient}, though the $\ell_2$-norm is less effective at promoting sparsity than the $\ell_1$-norm \cite{tibshirani1996regression}.
When replacing constraints with objectives, weight factors can be applied to individual vector elements, and iterative re-weighting schemes can then be used to improve convergence \cite{cotrutz2002using, halabi2006dose, candes2008enhancing, abdi2013comparison}.

The complication of nonconvexity arises from the combinatorial nature of the cardinality constraint~\eqref{eq:cardConst}; once a sparsity pattern for $x$ has been chosen, the problem becomes convex.
This observation is the motivation behind the multi-stage approach of 1) solving an approximation of the original problem, 2) determining a sparsity pattern for $x$ based on the resulting solution, and 3) {\it polishing} the solution from Step 1 by solving a convex problem using the sparsity pattern determined in Step 2 \cite{sun2013recent, boyd2004convex, diamond2016general, banjac2017novel}.
In the FMO problem with nonconvex dose--volume constraints, a variety of heuristics have been used to choose the $p\%$ of organ voxels that may receive more than $d^\text{dv}$ Gy in Step 3.
For example, Morrill et al. \cite{morrill1990influence, morrill1991dose} divide organ volumes {\it a priori} into high-dose and low-dose regions based on tumor proximity, Saberian et al. \cite{saberian2016optimal, saberian2016theoretical} and Hou et al. \cite{hou2003optimization} solve the FMO problem without dose--volume constraints to determine sparsity patterns, while Fu et al. \cite{fu2019convex} and Mukherjee et al. \cite{mukherjee2019integrating} use convex relaxations of the original problem to determine which voxels may receive more than $d$ Gy.

The great watershed in optimization is between convex and nonconvex problems \cite{rockafellar1993lagrange}.
When minimizing a convex objective function over a convex set, all local minima are also global.
Conversely, optimizing any function over a nonconvex set leads to a nonconvex problem with potentially many local minima. Even in the small 2-D example in Figure~\ref{fig:simple}, the dose--volume constraint creates a nonconvex feasible set where each choice of voxel corresponds to a different local minimum.
As mentioned above, early IMRT treatment planning relied on stochastic methods such as simulated annealing and genetic algorithms to avoid converging to local minima, but these methods suffered from slow convergence rates \cite{webb2003physical, spirou1998gradient}.
Fortunately, many investigations into the properties of local minima induced by nonconvex dose--volume constraints have concluded that, given a good initialization, it is possible to find local minima whose dose distributions are close to the global minimum using fast gradient methods \cite{bortfeld1999optimized, wu2003method, wu2002multiple}, so the use of stochastic methods may not be warranted \cite{chui2001inverse, alber1999objective, llacer2003absence}.
In this case, the emphasis in IMRT is often on finding a good local minimum that satisfies clinical objectives rather than finding the global minimum \cite{chui2001inverse, zhang2006fluence}.

\subsection{Contributions} 

In this paper, we develop a novel approach to handle dose--volume constraints in the radiation treatment planning problem.
In contrast to prior efforts, which focused on convex approximations, our formulation is  nonconvex, yet can be efficiently solved using continuous optimization techniques.
We start with the idealized problem, formulated using objectives for the tumor and constraints for the OARs, and extend it to a flexible formulation that allows re-weighting strategies to balance infeasible solutions (with respect to the OARs) against delivering a sufficient amount of radiation to the tumor.
To solve this formulation, we develop a customized algorithm that combines quadratic programming with fast nonlinear operators, and show that it is guaranteed to converge to a stationary point of the proposed model.
The resulting approach is scalable, handles multiple objectives on tumors and OARs, and returns actionable treatment plans on patient-scale datasets. 

The paper proceeds as follows.
In Section~\ref{sec:formulation} we formulate radiation treatment planning as a nonconvex optimization problem over the fluence map.
In Section~\ref{sec:approach} we develop a custom optimization algorithm and provide a convergence analysis. In Section~\ref{sec:results} we present multiple radiation treatment use cases with data from an anonymized cancer patient dataset, demonstrating the feasibility of our approach.
Finally, in Section~\ref{sec:conclusion} we summarize the results and discuss future avenues for research. 

\section{Problem Formulation}
\label{sec:formulation}

Given a fixed set of beam angles, the dose delivered to the patient is approximated with a linear mapping from beamlet intensity to voxel dose.
The FMO step is to calculate beamlet intensities that achieve the desired dose distribution to planning target volumes (PTVs) indexed by $i \in \mathcal{I}$, while respecting constraints on OARs indexed by $j \in \mathcal{J}$.
Let $d_i$ denote target delivery doses (often in Gy) to PTV $i \in \mathcal{I}$, and $d_j^{\max}$, $d_j^{\text{mean}}$, and $d_j^{\text{dv}}$ denote the maximum, mean, and dose--volume tolerance values for OAR $j\in \mathcal{J}$,
where the size of dose parameters (i.e., scalar or vector) can be determined by context.
The sets of OARs that have maximum-dose, mean-dose, and dose--volume constraints are denoted by $\mathcal{J}^{\mbox{\scriptsize{max}}}, \mathcal{J}^{\mbox{\scriptsize{mean}}}$, and $\mathcal{J}^{\mbox{\scriptsize{dv}}}$.
Let $A_i$ and $A_j$ be the linear beamlet-to-voxel maps for the PTV $i$ and OAR $j$.
The nonnegative variable $x$ encodes the unknown vector of beamlet intensities.
The idealized FMO problem with a uniform dose to the PTVs and the three constraints described in Section~\ref{sec:intro} on the OARs is given by
\begin{mini!}[2] 
    {x\geq 0}{\sum_{i \in \mathcal{I}} \frac{\alpha_i}{2n_i} \| A_i x - d_i \|_2^2 + \frac{\lambda}{2} \| x \|_2^2,}{\label{eqn:origEq:tumor}}{(\mathcal{P}_0)}
    \addConstraint{A_j x}{\leq d_j^\text{max}, \quad \label{eqn:origEq:max}}{j \in \mathcal{J}^\text{max}}
    \addConstraint{\mathbf{1}^TA_jx}{\leq n_jd_j^\text{mean},\quad \label{eqn:origEq:mean}}{j \in \mathcal{J}^\text{mean}}
    \addConstraint{\| (A_j x - d_j^\text{dv})_+ \|_0}{\leq \frac{n_j p_j}{100},\quad \label{eqn:origEq:dv}}{j \in \mathcal{J}^\text{dv},}
\end{mini!}
where $n_i$ and $n_j$ indicate the number of voxels in the $i$th PTV and $j$th OAR, and the percentage of voxels in the $j$th OAR that may receive a dose exceeding $d_j^\text{dv}$ is specified by $p_j$.
The least-squares term used for the uniform target dose is widely used \cite{ehrgott2010mathematical, bortfeld1990methods, wu2000algorithms, spirou1998gradient, cotrutz2002using}.
In principle, Problem~$(\mathcal{P}_0)$ can be replaced by a feasibility problem with lower and upper dose--volume constraints also placed on the PTV \cite{cho1998optimization, penfold2017sparsity, michalski2004dose}.
Here we focus on $(\mathcal{P}_0)$, but the techniques we propose extend to the feasibility problem.

Problem~$(\mathcal{P}_0)$ is difficult for several reasons.
First, it is high-dimensional (beamlets can be on the order of $10^3-10^5$) \cite{wu2003method}.
Second, beamlet-to-voxel maps are generally ill-conditioned \cite{webb2003physical, shepard2000iterative, ahmad2010genetic, crooks2001linear}.
Third, dose--volume constraints are combinatorial in nature, as constraint~\eqref{eqn:origEq:dv} requires making a choice of voxels allowed to exceed $d_j^\text{dv}$, so finding the global solution is NP-hard \cite{tuncel2012strong}.
Finally, while Problem~$(\mathcal{P}_0)$ is always feasible (i.e., $x = 0$ is a solution), it may not yield a clinically useful radiation distribution. Feasibility-type reformulations of $(\mathcal{P}_0)$ often fail to have any meaningful solutions because the target doses almost always compete with OAR doses, as it is common for tumors to be close to or within OARs.
Practical radiation treatment planning converts constraints into objectives, and then looks for the best trade-off among competing objectives. 

In the next section, we formulate the idealized problem using constraints, and then develop a relaxation that always has a feasible solution.
This allows significant modeling and computational flexibility, yielding efficiently computable and clinically useful solutions that balance the need to deliver enough radiation to the tumor while approximately satisfying OAR constraints.  

\subsection{New formulation} 

To simplify the exposition, we explain the new formulation using only dose--volume constraints on the OARs:
\begin{mini!}[2] 
    {x\geq 0}{\sum_{i \in \mathcal{I}} \frac{\alpha_i}{2n_i} \| A_i x - d_i \|_2^2 + \frac{\lambda}{2} \| x \|_2^2}{}{(\mathcal{P}_1)}
    \label{eq:justDvcEq}
    \addConstraint{\| (A_j x - d_j^{\text{dv}})_+ \|_0}{\leq \frac{n_jp_j}{100},\quad}{j \in \mathcal{J}^{\text{dv}}\label{eq:p1dvc}}.
\end{mini!}
Solutions to~$(\mathcal{P}_1)$ may not be clinically useful if the dose--volume constraints prevent sufficient dose delivered to the tumor.
We develop an extended, nonconvex formulation that relaxes the dose--volume constraints.
The new formulation includes additional variables $w_j \approx A_j x - d_j^{\text{dv}}$, and penalizes deviations of the delivered OAR doses from $w_j$:
\begin{mini!}[2] 
    {x \geq 0,w}{f(x,w) := \sum_{i \in \mathcal{I}} \frac{\alpha_i}{2n_i} \| A_i x - d_i \|_2^2}{}{(\mathcal{P}_2)\nonumber}
    \breakObjective{\hspace{3em}+ \sum_{j \in \mathcal{J}^\text{dv}} \frac{\alpha_j}{2n_j} \| w_j - (A_j x - d_j^{\text{dv}}) \|_2^2 + \frac{\lambda}{2} \| x \|_2^2}\label{eq:fObj}
    \addConstraint{\| (w_j)_+ \|_0}{\leq \frac{n_jp_j}{100},\quad}{j \in \mathcal{J}^{\text{dv}}}.
    \label{newEq}
\end{mini!}
In formulation~$(\mathcal{P}_2)$, the auxiliary variables $w_j$ are forced to satisfy the dose--volume constraints, while the residuals $ A_j x - d_j^{\text{dv}}$ may not necessarily do so.
The weights $\alpha_j$ control how closely $w_j$ must approximate $A_j x - d_j^{\text{dv}}$, and as $\alpha_j \uparrow \infty$, problem~$(\mathcal{P}_2)$ converges to the idealized problem~$(\mathcal{P}_1)$ in an epigraphical sense. 

The epigraph of a function is defined by 
\begin{equation}
    \text{epi}(f) = \{(x,\beta) \,:\, f(x) \leq \beta\}.
\end{equation}
Epigraphs make it possible to go between sets and functions, and the notion of epigraphical convergence~\cite{rockafellar2009variational} 
is exactly that of set convergence for epigraphs, i.e.,
\begin{equation}
    f^\alpha  \quad \rightarrow_{\text{epi}}  \quad f \qquad \text{exactly when}  \qquad \text{epi}(f^\alpha)   \quad \rightarrow \quad \text{epi}(f).
\end{equation}
We know that our objective functions \eqref{eq:fObj} converge epigraphically to the original problem by~\cite[Theorem 7.4(d)]{rockafellar2009variational}
because of monotonicity,
\begin{equation}
    f^{\alpha_j} \leq f^{\alpha_{j+1}} \quad \text{when} \quad \alpha_j \leq \alpha_{j+1},
\end{equation}
which we always have when using quadratic penalties for prox-bounded functions (see discussion below Theorem 7.4(d)).

Formulation~$(\mathcal{P}_2)$ has several desirable features.
First, while most heuristic approaches described in the introduction rely on convex approximations, $(\mathcal{P}_2)$ fully captures the nonconvex structure of dose--volume constraints, while giving the modeler flexibility to match them approximately using weights $\alpha_j$.
The weights $\alpha_i$ and $\alpha_j$ can be tuned to balance the trade-off between covering the tumor and approximately meeting OAR constraints.
Second, from a computational perspective, we can design efficient and provably convergent algorithms that aggregate information from multiple tumors and OARs using partial minimization, as described in the next section.
Furthermore, approximate solutions can be computed efficiently, then refined using techniques such as iterative re-weighting schemes or polishing.

\section{Algorithmic Approach}
\label{sec:approach}

Our main strategy is to {\it partially minimize} formulation~$(\mathcal{P}_2)$ with respect to $x$, viewing~$(\mathcal{P}_2)$ as a {\it value-function optimization} problem
\begin{mini!}[2] 
    {w}{g(w) := \min_{x \geq 0} \sum_{i \in \mathcal{I}} \frac{\alpha_i}{2n_i} \| A_i x - d_i \|_2^2}{}{(\mathcal{P}_3)\nonumber}
    \breakObjective{\hspace{2em}+ \sum_{j \in \mathcal{J}^\text{dv}} \frac{\alpha_j}{2n_j} \| w_j - (A_j x - d_j^{\text{dv}}) \|_2^2 + \frac{\lambda}{2} \| x \|_2^2}{\label{eq:value}}
    \addConstraint{w_j \in \Omega_j := \left\{ w_j \in \mathbb{R}^{n_j} :\, \| (w_j)_+ \|_0 \leq \frac{n_jp_j}{100}\right\}, \quad}{j \in \mathcal{J}^\text{dv}}.
\end{mini!}
The algorithm is centered on optimizing the value function $g(w)$ in~\eqref{eq:value}, using a projected gradient approach to take care of the nonconvex constraints encoded by $\Omega_j$.
The high-level iteration is simply 
\begin{equation}
    \label{eq:pgsimple}
    w^{(k+1)} \in \text{proj}_\Omega \left[w^{(k)} - T \nabla g\left(w^{(k)}\right)\right],
\end{equation}
where $w = (w_{j_1}, w_{j_2}, \dots)$, $T = \text{diag}(t_{j_1}, \dots, t_{j_2}, \dots)$, and $\Omega = \Omega_{j_1} \times \Omega_{j_2} \times \dots$ for $j_i \in \mathcal{J}^\text{dv}$.
Evaluation of $g$ and $\nabla g$ is done through computing the partial minimum over $x$, as detailed in Algorithm~\ref{algo:pg}.
Projection onto $\Omega_j$ can be implemented efficiently by ordering the elements of $w_j$ and setting the lowest $\left\lceil\sfrac{(100-p_j)n_j}{100}\right\rceil$ entries of $w_j$ to $\min\{0,w_j\}$ \cite{penfold2017sparsity}.\\[-5ex]

\begin{figure}[H]
    \centering
    \begin{minipage}{0.8\linewidth}
        \begin{algorithm}[H]
            \caption{Projected gradient descent for $g(w)$ in~$(\mathcal{P}_3)$.}
            \label{algo:pg}
            \begin{algorithmic}
                \State{Input $\epsilon$, $x^{(0)}$. Initialize $k = 0$, $\text{err}> \epsilon$, $w_j^{(0)} = \text{proj}_{\Omega_j} \left(A_jx^{(0)} - d_j^\text{dv}\right)$. Set $t_j \leq \sfrac{n_j}{\alpha_j}$.}
                \While{ err $> \epsilon$}
	            \State{$x^{(k+1)} \leftarrow \argmin_{x \geq 0} f\left(x,w^{(k)}\right)$}
	            \For{$j \in \mathcal{J}^\text{dv}$}
		        \State{$w_j^{(k+1)} \leftarrow \text{proj}_{\Omega_j} \left\{ w_j^{(k)} - \frac{\alpha_j t_j}{n_j} \left[w_j^{(k)} - \left(A_j x^{(k+1)} - d_j^\text{dv} \right) \right] \right\}$}
	            \EndFor
	            \State{err  $\leftarrow\sum_{j \in \mathcal{J}^\text{dv}} \frac{1}{t_j}\left\|w_j^{(k+1)} - w_j^{(k)}\right\|_2$}
	            \State{$k \leftarrow k + 1$}
                \EndWhile
                \State \Return{$\argmin_{x \geq 0} f\left(x,w^{(k)}\right)$}
            \end{algorithmic}
        \end{algorithm}
    \end{minipage}
\end{figure}

The differentiability of $g(w)$, formula for the derivative, and its  Lipschitz constant follow from~\cite[Theorem 1]{zheng2018relax}. In particular, for one PTV and one OAR, define 
\begin{equation}
    g(w) = \min_{x \geq 0} \frac{\alpha_1}{2n_1} \| A_1x - d_1 \|_2^2 +  \frac{\alpha_2}{2n_2}\|w - (A_2x-d_2^\text{dv})\|_2^2 + \frac{\lambda}{2} \| x \|_2^2.
\end{equation}
Next, taking $h(x) = \delta_+(x) + \frac{1}{2}\|A_1x-d_1\|_2^2 + \frac{\lambda}{2} \| x \|_2^2$, where $\delta_+(x)$ is the indicator function of the nonnegative orthant, 
\begin{equation}
    \delta_+(x) = \begin{cases} 0 & \text{if } x \geq 0, \\ +\infty & \text{otherwise,} \end{cases}
\end{equation}
we have
\begin{align}
    g(w) &= \min_x h(x) + \frac{\alpha_2}{2n_2} \| w - (A_2 x - d_2^\text{dv}) \|_2^2, \\
    x(w) &\in \argmin_x h(x) + \frac{\alpha_2}{2n_2} \| w - (A_2x - d_2^\text{dv}) \|_2^2, 
\end{align}
and from \cite[Theorem 1]{zheng2018relax}, we get
\begin{equation}
    \nabla g(w) = \frac{\alpha_2}{n_2} \left\{w - \left[ A_2 x(w) - d_2^\text{dv} \right]\right\} \quad \text{and} \quad \text{lip}(\nabla g) \leq \frac{\alpha_2}{n_2}.
\end{equation}
Similarly, for multiple OARs we have $\sfrac{\partial g}{\partial w_j}(w) = \sfrac{\alpha_j}{n_j}\left\{ w_j - \left[A_j x(w) - d_j^\text{dv}\right]\right\}$ and $\text{lip}\left(\sfrac{\partial g}{\partial w_j}\right) \leq \sfrac{\alpha_j}{n_j}$ for all $j \in \mathcal{J}^\text{dv}$, and for all $w, w' \in \mathbb{R}^{\sum_{j \in \mathcal{J}^\text{dv}} n_j}$ we have
\begin{align}
    \left\| \nabla g(w) - \nabla g\left(w'\right) \right\|_2^2 &= \sum_{j \in \mathcal{J}^\text{dv}} \left\| \tfrac{\partial g}{\partial w_j} (w) - \tfrac{\partial g}{\partial w_j} (w') \right\|_2^2, \\
    &\leq \sum_{j \in \mathcal{J}^\text{dv}} \left(\frac{\alpha_j}{n_j}\right)^2 \left\| w_j - w_j' \right\|_2^2, \\
    &\leq \max_{j \in \mathcal{J}^\text{dv}} \left(\frac{\alpha_j}{n_j}\right)^2 \sum_{j \in \mathcal{J}^\text{dv}} \left\|w_j - w_j'\right\|_2^2, \\
    &= \max_{j \in \mathcal{J}^\text{dv}} \left(\frac{\alpha_j}{n_j}\right)^2 \left\|w - w'\right\|_2^2,
    \end{align}
so $\text{lip}\left(\nabla g\right) \leq \max_{j \in \mathcal{J}^\text{dv}} \left(\sfrac{\alpha_j}{n_j}\right)$.

Algorithm~\ref{algo:pg} is projected gradient descent for the value function~\eqref{eq:value}, and is equivalent to the simple iteration~\eqref{eq:pgsimple}.
Under fairy weak assumptions\footnote{Specifically, the sequence $\{w^{(k)}\}_{k \in \mathbb{N}}$ generated by \eqref{eq:pgsimple} must be bounded, and $g + \delta_\Omega$ must be a Kurdyka-\L{}ojasiewicz function (which covers a wide range of functions, including the class of real semi-algebraic functions).}, the nonconvex gradient projection method converges for a fixed step size $t < \sfrac{1}{\text{lip}(\nabla g)}$ \cite[Theorem 5.3]{attouch2013convergence}.
On the other hand, Algorithm~\ref{algo:pg} allows the flexibility to choose different step sizes $t_j \leq \sfrac{n_j}{\alpha_j}$ for each $j \in \mathcal{J}^\text{dv}$.
We prove the convergence of this variant of projected gradient descent, starting with some definitions to quantify stationarity and optimality.

\begin{definition}[Stationary Point]
    A vector $w \in \Omega$ is called a stationary point for $(\mathcal{P}_3)$ if
    \begin{equation}
        0 \in \nabla g(w) + N_\Omega(w),\label{eq:stationary}
    \end{equation}
    where $N_\Omega$ is the normal cone to $\Omega$ at $w$.\label{def:stationary}
\end{definition}

This definition follows from the basic first-order conditions for optimality \cite[Theorem 6.12]{rockafellar2009variational}, where \eqref{eq:stationary} is a necessary condition for $w$ to be a local minimizer of $g$ on $\Omega$.
This definition of stationarity motivates the following measure of optimality.

\begin{definition}[$\epsilon$-accurate Solution]
    Given fixed step sizes $t_j > 0$ for all $j \in \mathcal{J}^\text{dv}$, we call a vector $w \in \Omega$ an $\epsilon$-accurate solution for $(\mathcal{P}_3)$ if
    \begin{equation}
        \sum_{j \in \mathcal{J}^\text{dv}} \frac{1}{t_j^2} \left\| w_j - \bar{w}_j \right\|_2^2 \leq \epsilon,
    \end{equation}
    where $\bar{w}_j \in \text{proj}_{\Omega_j}\left[ w_j - t_j \sfrac{\partial g}{\partial w_j}(w)\right]$.\label{def:epAccurate}
\end{definition}

By the first-order optimality conditions of the problem defining the projection in \eqref{eq:pgsimple}, for each $j \in \mathcal{J}^\text{dv}$ the vectors $w_j$ and $\bar{w}_j$ satisfy
\begin{equation}
    \frac{1}{t_j}\left(w_j - \bar{w}_j\right) \in \frac{\partial g}{\partial w_j}(w) + N_{\Omega_j}\left(\bar{w}_j\right),
\end{equation}
so if $\sum_{j \in \mathcal{J}^\text{dv}}\sfrac{1}{t_j^2}\left\| w_j - \bar{w}_j \right\|_2^2 = 0$, then $0 \in \sfrac{\partial g}{\partial w_j}(w) + N_{\Omega_j}(w_j)$ for all $j \in \mathcal{J}^\text{dv}$.
Furthermore, by \cite[Proposition 6.41]{rockafellar2009variational}, we have $0 \in \nabla g(w) + N_\Omega(w)$, so $w$ is a stationary point for problem $(\mathcal{P}_3)$.

We now establish the convergence of Algorithm~\ref{algo:pg}, with proofs provided in Appendix~\ref{sec:proofs}.

\begin{theorem}
    For fixed step sizes $t_j \in \left(0, \sfrac{n_j}{\alpha_j}\right]$ for all $j \in \mathcal{J}^\text{dv}$, Algorithm 1 generates a sequence of iterates $\left\{w^{(k)}\right\}_{k \in \mathbb{N}}$ such that $g\left(w^{(k)}\right)\downarrow g^* \geq 0$, in particular $g^* \leq g\left(w^{(k)}\right)$ for all $k \in \mathbb{N}$.
    Furthermore, for any convergent subsequence $w^{(p_k)} \to w_p^*$, the vector $w_p^*$ is a local minimizer of $g$ on $\Omega$, and $g\left(w_p^*\right) = g^*$. \label{thm:g_converge}
\end{theorem}

We note that the convergence of our sequence of objective values $\left\{g\left(w^{(k)}\right)\right\}_{k \in \mathbb{N}}$ only implies subsequence convergence for our iterates $\left\{w^{(k)}\right\}_{k \in \mathbb{N}}$.
In what follows, we prove the convergence of our sequence of iterates to a local minimizer of $g$ on $\Omega$ for the case where $t_j < \sfrac{n_j}{\alpha_j}$ for all $j \in \mathcal{J}^\text{dv}$.

\begin{theorem}
    For fixed step sizes $t_j \in \left(0,\sfrac{n_j}{\alpha_j}\right)$ for all $j \in \mathcal{J}^\text{dv}$, the sequence of iterates $\left\{w^{(k)}\right\}_{k \in \mathbb{N}}$ generated by Algorithm~\ref{algo:pg} converges to a local minimizer of $g$ on $\Omega$.
    Furthermore, we reach an $\epsilon$-accurate solution in $\mathcal{O}(\sfrac{1}{\epsilon}$) iterations.\label{thm:w_converge}
\end{theorem}

If we choose $t_j = \sfrac{n_j}{\alpha_j}$ for each $j \in \mathcal{J}^\text{dv}$, Algorithm~\ref{algo:pg} reduces to block coordinate descent, detailed in Algorithm~\ref{algo:bcd}, and its convergence to a function value corresponding to a local minimizer follows immediately from Theorem~\ref{thm:g_converge}.\\[-5ex]

\begin{figure}[H]
    \centering
    \begin{minipage}{0.8\linewidth}
        \begin{algorithm}[H]
            \caption{Block coordinate descent: Special case of Algorithm~\ref{algo:pg} with $t_j = \sfrac{n_j}{\alpha_j}$.}
            \label{algo:bcd}
            \begin{algorithmic}
                \State{Input $\epsilon$, $x^{(0)}$. Initialize $k = 0$, $\text{err} > \epsilon$, $w_j^{(0)} = \text{proj}_{\Omega_j} \left(A_jx^{(0)} - d_j^\text{dv}\right)$.}
                \While{ err $> \epsilon$}
	            \State{$x^{(k+1)} \leftarrow \argmin_{x \geq 0} f\left(x,w^{(k)}\right)$}
	            \State{$w^{(k+1)} \leftarrow \argmin_{w \in \Omega} f\left(x^{(k+1)},w\right)$}
	            \State{err $\leftarrow \sum_{j\in \mathcal{J}^\text{dv}} \frac{\alpha_j}{n_j}\left\|w_j^{(k+1)} - w_j^{(k)} \right\|_2$}
	            \State{$k \leftarrow k+1$}
                \EndWhile 
                \State \Return{$\argmin_{x \geq 0} f\left(x,w^{(k)}\right)$}
            \end{algorithmic}
        \end{algorithm}
    \end{minipage}
\end{figure}

\subsection{Simple example}

To demonstrate how the relaxed formulation~$(\mathcal{P}_2)$ and Algorithm~\ref{algo:bcd} behave compared to the idealized problem~$(\mathcal{P}_1)$, we revisit our example from the introduction.
For iterations $k = 0, 3, 32$ of Algorithm~\ref{algo:bcd} to solve~$(\mathcal{P}_2)$, we plot the contour lines of the objective function 
\begin{equation}
    \label{eq:simple2}
    f(x,w^{(k)}) = \frac{1}{2} \| A_1 x - 81 \|_2^2 + \frac{5}{2} \| w^{(k)} - (A_2 x - 20) \|_2^2 + \frac{\lambda}{2} \| x \|_2^2
\end{equation}
in terms of the beamlet intensity variable $x$ for fixed $w^{(k)}$.
Here $A_1 \in \mathbb{R}^{1 \times 2}$ and $A_2 \in \mathbb{R}^{2 \times 2}$ are the beamlet-to-voxel maps for the prostate tumor and rectum, $x \in \mathbb{R}^2$ contains the beamlet intensities, $w^{(k)} \in \mathbb{R}^2$ approximates the deviation between the dose received by the rectum and the dose--volume constraint, and $\lambda$ is set at $5 \times 10^{-6}$.
In each figure, $x^{(k+1)}$ corresponds to the global minimum of the relaxed problem at iteration $k$, A and B correspond to the local and global minima of the idealized problem, and the nonconvex feasible region induced by the dose--volume constraint that no more than 50\% of the rectum volume exceeds 20 Gy is given by the union of regions $v_1 \leq 20 $ and $v_2 \leq 20$.

\begin{figure}
    \centering
    \hspace{1em}
    \begin{overpic}[scale=0.4]{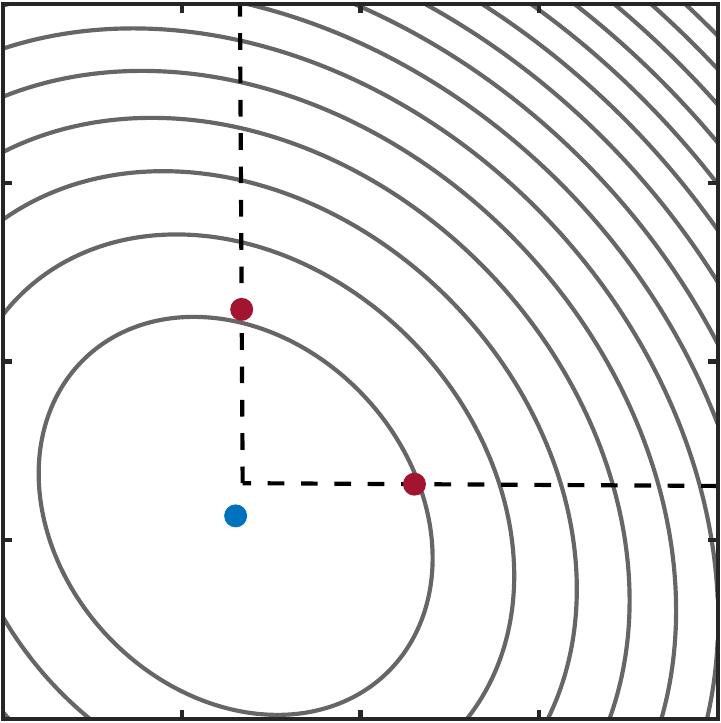}
        \put(41.5,102){\small$k = 0$}
        \put(-1,-5){\scriptsize0}
        \put(21,-5){\scriptsize0.5}
        \put(48.5,-5){\scriptsize1}
        \put(71,-5){\scriptsize1.5}
        \put(88,-5){\scriptsize2}
        \put(91,-4){\tiny$\times10^4$}
        \put(-4,-1){\scriptsize0}
        \put(-9,23){\scriptsize0.5}
        \put(-4,48.5){\scriptsize1}
        \put(-9,73){\scriptsize1.5}
        \put(-4,98){\scriptsize2}
        \put(-1,101){\tiny$\times10^4$}
        \put(25,55){\small A}
        \put(55,25){\small B}
        \put(30,21.5){\small $x^{(1)}$}
        \put(35.5,76.5){\rotatebox{90}{\footnotesize$v_1\leq20$}}
        \put(77,36){\footnotesize$v_2\leq20$}
        \put(47,-11){\small$x_1$}
        \put(-13,46){\rotatebox{90}{\small$x_2$}}
    \end{overpic} \hspace{0.25em}
    \begin{overpic}[scale=0.4]{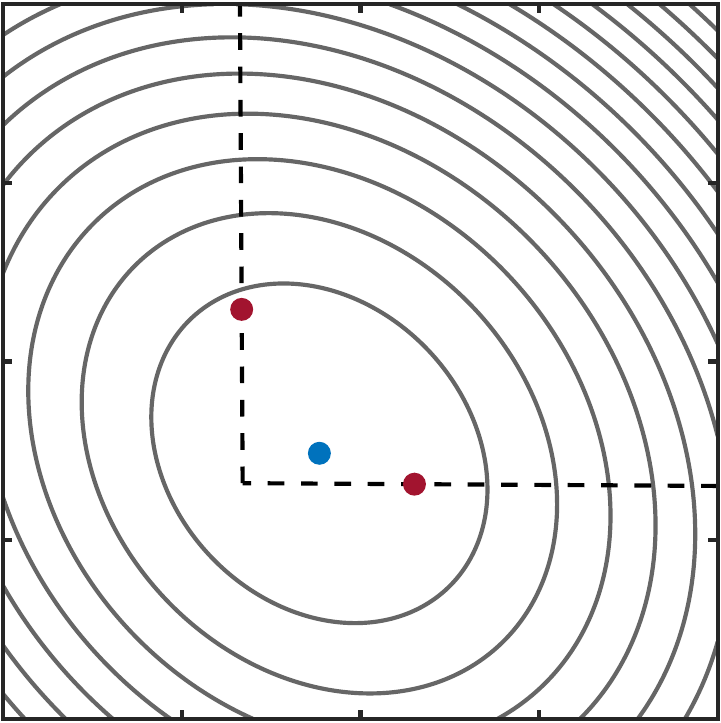}
        \put(41.5,102){\small$k = 3$}
        \put(-1,-5){\scriptsize0}
        \put(21,-5){\scriptsize0.5}
        \put(48.5,-5){\scriptsize1}
        \put(71,-5){\scriptsize1.5}
        \put(88,-5){\scriptsize2}
        \put(91,-4){\tiny$\times10^4$}
        \put(25,55){\small A}
        \put(55,25){\small B}
        \put(42.5,40.75){\small $x^{(4)}$}
        \put(35.5,76.5){\rotatebox{90}{\footnotesize$v_1\leq20$}}
        \put(77,36){\footnotesize$v_2\leq20$}
        \put(47,-11){\small$x_1$}
    \end{overpic} \hspace{0.25em}
    \begin{overpic}[scale=0.4]{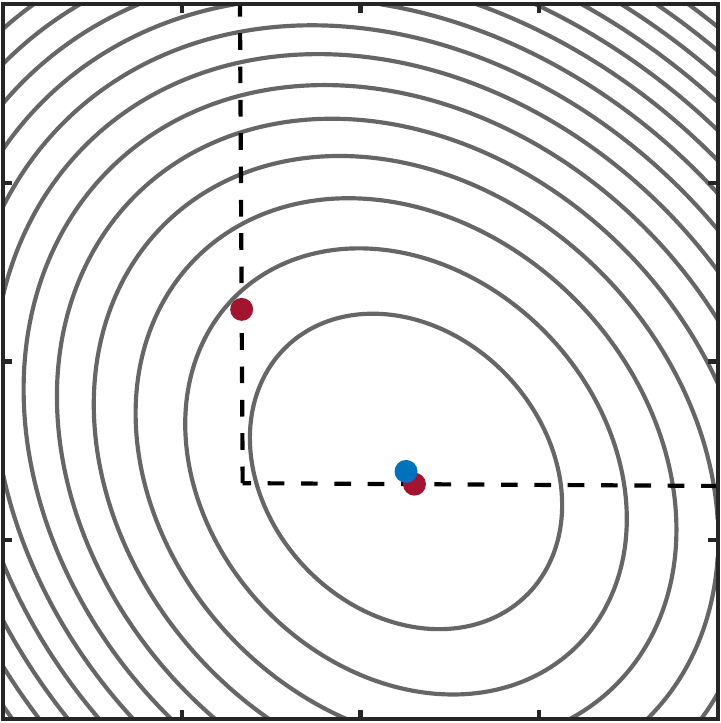}
        \put(41.5,102){\small$k = 32$}
        \put(-1,-5){\scriptsize0}
        \put(21,-5){\scriptsize0.5}
        \put(48.5,-5){\scriptsize1}
        \put(71,-5){\scriptsize1.5}
        \put(88,-5){\scriptsize2}
        \put(91,-4){\tiny$\times10^4$}
        \put(25,55){\small A}
        \put(55,25){\small B}
        \put(54,38){\small $x^{(33)}$}
        \put(35.5,76.5){\rotatebox{90}{\footnotesize$v_1\leq20$}}
        \put(77,36){\footnotesize$v_2\leq20$}
        \put(47,-11){\small$x_1$}
    \end{overpic} \vspace{3ex}
    \caption{Contours of~\eqref{eq:simple2} in $x$ for fixed iterates $w^{(k)}$ with $k = 0, 3, 32$.
    The points A and B correspond to the local and global minima of the idealized problem~$(\mathcal{P}_1)$.
    Starting with the zero vector $x^{(0)}$, the global minimum of the relaxed problem ($x^{(k+1)}$) moves towards the global minimum of the idealized problem (B).\label{fig:contour1}}
\end{figure}

In Figure~\ref{fig:contour1} (left), we initialize the problem with $x^{(0)}$ as the zero vector.
The contours of~$(\mathcal{P}_2)$ are more circular than those of~$(\mathcal{P}_1)$ in Figure~\ref{fig:simple}, indicating that the relaxation has improved the problem conditioning.
Additionally, as $w^{(k)}$ changes with each iteration, the global minimum of the objective function in terms of $x$ at point $x^{(k+1)}$ moves closer to the global minimum of~$(\mathcal{P}_1)$ at point B.
For stopping tolerance $\epsilon = 10^{-3}$, our algorithm converges after 32 iterations, and in Figure~\ref{fig:contour1} (right) we see that the global minimum of our relaxed problem, corresponding to our final beamlet intensities, is just shy of the global minimum of our original problem.
The dose--volume constraint is met approximately, and we may choose larger weight coefficient $\alpha_2$ or decrease the dose level $d_2^{\text{dv}}$ to move closer to the feasible region if necessary.

Next, we plot our iterates along with the contour lines of the objective function 
\begin{equation}
    \label{eq:simple3}
    g(w) = \min_{x \geq 0} \frac{1}{2} \| A_1 x - 81 \|_2^2 + \frac{5}{2} \| w - (A_2 x - 20) \|_2^2 + \frac{\lambda}{2} \| x \|_2^2
\end{equation}
in terms of the variable $w$.
In Figure~\ref{fig:contour2} (left), points A and B again correspond to the local and global minima of the idealized problem, and the nonconvex feasible region \mbox{$\Omega=\{w\in\mathbb{R}^2 \,:\, \| (w)_+ \|_0 \leq 1 \}$} is given by the union of regions $w_1 \leq 0$ and $w_2 \leq 0$.
The result of the gradient step is plotted as squares, and the projection onto $\Omega$ is plotted as circles.
Just as we saw in $x$-space, our algorithm in $w$-space converges to a point near the global minimum at B. 

\subsection{Initialization}

Because our formulation is nonconvex, the initialization can affect both the solution and the time to convergence.
For example, using the initialization $x^{(0)} = (0, 2\times10^3)$ for our simple example, we converge near the local minimum at A after 37 iterations (middle panel of Figure~\ref{fig:contour2}).
Fortunately, as demonstrated in other papers (see e.g., \cite{wu2002multiple, llacer2003absence}), local minima do not have a large effect on the solutions to these least-squares formulations, especially when the convex tumor terms have weights relative in size to the OAR term weights \cite{rowbottom2002configuration}.
Given enough time, different initializations tend to converge to  similar objective values, fluence maps, and dose distributions.

\begin{figure}
    \hspace{1em}
    \begin{overpic}[scale=0.4]{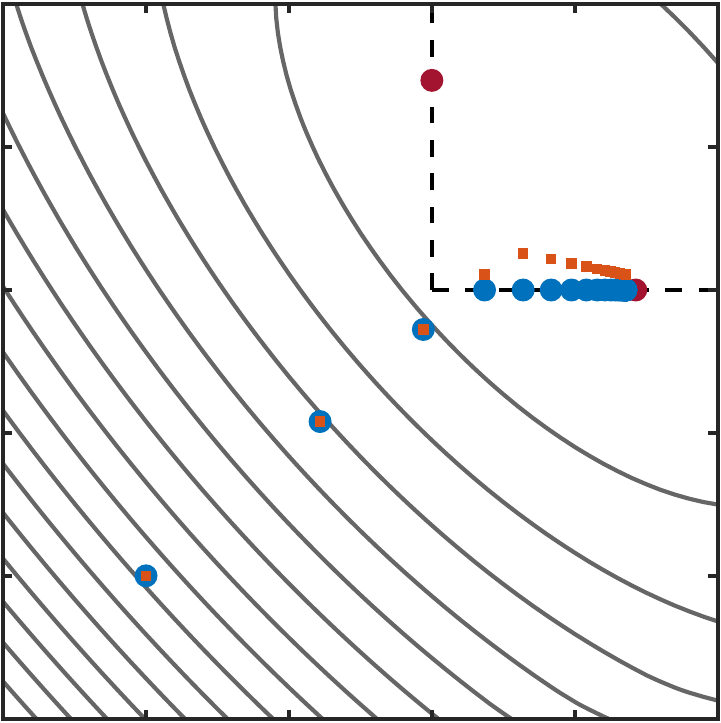}
        \put(-5,-5){\scriptsize-30}
        \put(15,-5){\scriptsize-20}
        \put(34.5,-5){\scriptsize-10}
        \put(58,-5){\scriptsize0}
        \put(76,-5){\scriptsize10}
        \put(96,-5){\scriptsize20}
        \put(-9,0){\scriptsize-30}
        \put(-9,18){\scriptsize-20}
        \put(-9,38){\scriptsize-10}
        \put(-4,58){\scriptsize0}
        \put(-7,78){\scriptsize10}
        \put(-7,98){\scriptsize20}
        \put(52.5,86){\small A}
        \put(85.5,51.5){\small B}
        \put(18,13){\small $w^{(0)}$}
        \put(47,-11){\small$w_1$}
        \put(-13,48){\rotatebox{90}{\small$w_2$}}
    \end{overpic} \hspace{1em}
    \begin{overpic}[scale=0.4]{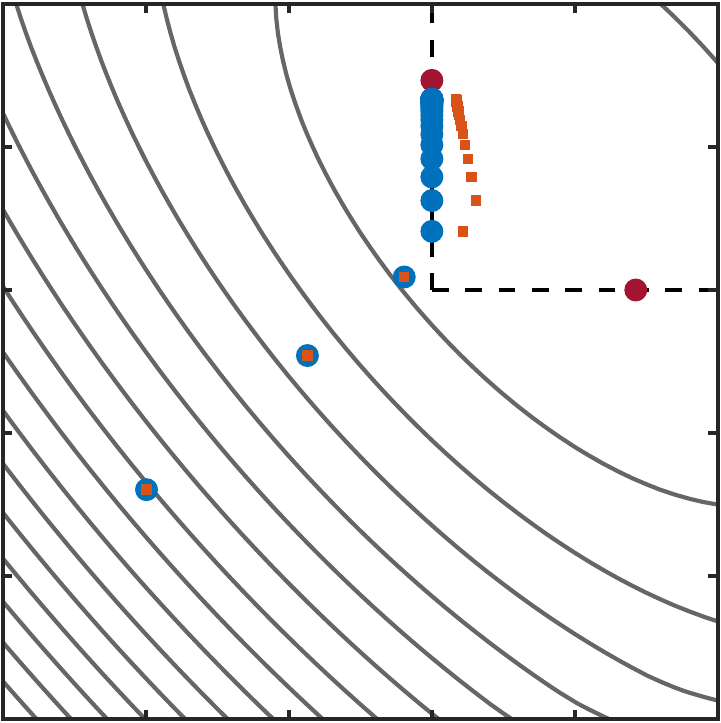}
        \put(-5,-5){\scriptsize-30}
        \put(15,-5){\scriptsize-20}
        \put(34.5,-5){\scriptsize-10}
        \put(58,-5){\scriptsize0}
        \put(76,-5){\scriptsize10}
        \put(96,-5){\scriptsize20}
        \put(-9,0){\scriptsize-30}
        \put(-9,18){\scriptsize-20}
        \put(-9,38){\scriptsize-10}
        \put(-4,58){\scriptsize0}
        \put(-7,78){\scriptsize10}
        \put(-7,98){\scriptsize20}
        \put(52.5,86){\small A}
        \put(85.5,51.5){\small B}
        \put(18,25){\small $w^{(0)}$}
        \put(47,-11){\small$w_1$}
    \end{overpic} \hspace{1em}
    \begin{overpic}[scale=0.4]{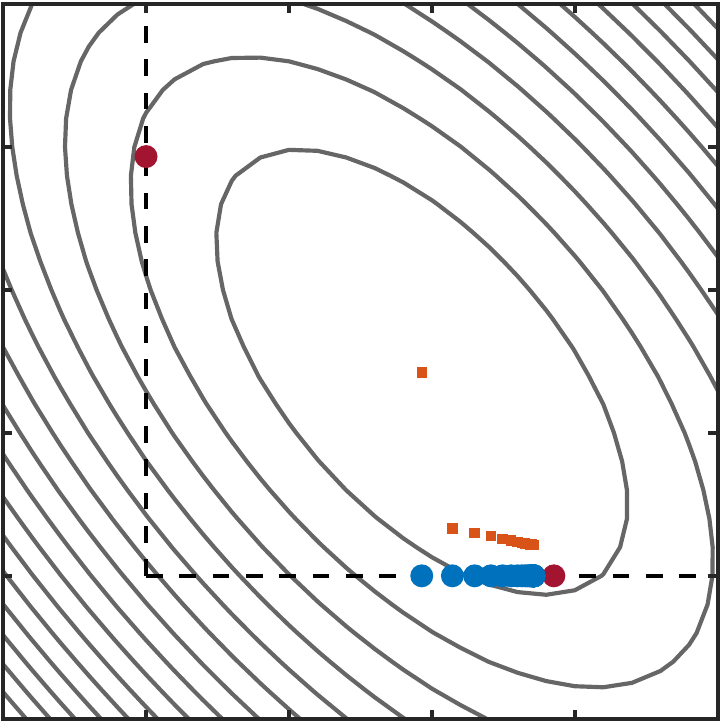}
        \put(-1,-5){\scriptsize-5}
        \put(18.5,-5){\scriptsize0}
        \put(38.5,-5){\scriptsize5}
        \put(56,-5){\scriptsize10}
        \put(76,-5){\scriptsize15}
        \put(96,-5){\scriptsize20}
        \put(-6.5,0){\scriptsize-5}
        \put(-4,18){\scriptsize0}
        \put(-4,38){\scriptsize5}
        \put(-7,58){\scriptsize10}
        \put(-7,78){\scriptsize15}
        \put(-7,98){\scriptsize20}
        \put(12,76){\small A}
        \put(74.5,12){\small B}
        \put(45,13){\small $w^{(0)}$}
        \put(47,-11){\small$w_1$}
    \end{overpic} \vspace{3ex}
    \caption{Contours of~\eqref{eq:simple3} with respect to $w$.
    The points A and B correspond to the local and global minima of the idealized problem~$(\mathcal{P}_1)$, the squares correspond to the iterates $w^{(k)}$ before projection onto the feasible set $\Omega$, and the circles correspond to the iterates after projection.
    Left: Starting with the zero vector $x^{(0)}$, we converge after 32 iterations at a point near the global minimum at B.
    Center: Starting with the initialization $x^{(0)} = (0,2\times10^3)$, we converge near the local minimum at A after 37 iterations.
    Right: Starting with the initialization $x^{(0)} = \argmin_{x \geq 0} \sfrac{1}{2} \| A_1 x - 81 \|_2^2 + \sfrac{\lambda}{2} \| x \|_2^2$, we converge near the global minimum at B after 26 iterations.\label{fig:contour2}}
\end{figure}

For the remainder of our examples, we initialize the beamlet intensity vector as the solution to $(\mathcal{P}_2)$ without any dose--volume terms, i.e., 
\begin{equation}
    \label{eq:init}
    x^{(0)} = \argmin_{x \geq 0} \sum_{i \in \mathcal{I}} \frac{\alpha_i}{2n_i} \| A_i x - d_i \|_2^2 + \frac{\lambda}{2} \| x \|_2^2.
\end{equation}
In Figure~\ref{fig:contour2} (right), we use this initialization for our simple example, converging near the global minimum at B after 26 iterations.
For this initialization, we see that $A_2x^{(0)} - 20$, the value of $w^{(0)}$ before projection onto $\Omega$, corresponds to the global minimum of $g(w)$.
This is true for the general problem $(\mathcal{P}_2)$ as well.
Specifically, for all $j \in \mathcal{J}^\text{dv}$ we have
\begin{equation} 
    \argmin_{w_j} f(x,w) = A_j x - d_j^\text{dv},
\end{equation}
which means that
\begin{equation}
    \min_{w} f(x,w) = \sum_{i \in \mathcal{I}} \frac{\alpha_i}{2n_i} \| A_i x - d_i \|_2^2 + \frac{\lambda}{2} \| x \|_2^2. 
\end{equation}
Therefore letting $x^{(0)} = \argmin_{x \geq 0} \left( \min_w f(x,w) \right)$, we are guaranteed that
\begin{equation} 
    A_j x^{(0)} - d_j^\text{dv} = \argmin_{w_j} g(w)
\end{equation} 
for all $j \in \mathcal{J}^\text{dv}$ because
\begin{equation} 
    \min_{x \geq 0} \left( \min_w f(x,w) \right) = \min_w g(w).
\end{equation}
Since we are ultimately interested in solving $\min_{w\in\Omega} g(w)$, using an initialization for $x^{(0)}$ that corresponds to $w^{(0)} = \text{proj}_{\Omega}(\argmin_w g(w))$ is a natural choice.

\section{Numerical Results}
\label{sec:results}

Our examples are performed on data from an anonymized prostate cancer patient included in the CORT dataset \cite{craft2014shared}, available in the GigaScience repository, GigaDB \cite{craft2014shared2}.
The dataset includes beamlet-to-voxel maps for beam angles ranging from $0^\circ$ to $358^\circ$ in increments of $2^\circ$, PTVs in the prostate and lymph nodes, and OARs including the rectum, bladder, femoral heads, and unspecified tissue (Figure~\ref{fig:CORT}).
In the CORT dataset, beamlet intensities are measured in monitor units (MU), where 100 MU delivers a dose of 1 Gy in 10 cm depth in water in the center of a 10 cm $\times$ 10 cm radiation field \cite{craft2014shared}.

For a fixed set of beams, we consider the problem of treating tumors with dose--volume constraints of increasing difficulty on various combinations of critical structures.
In Section~\ref{sec:onePoneOoneDV}, we look at an example with one PTV and one OAR, and show how the proposed approach works for a single dose--volume constraint on the OAR.
In Section~\ref{sec:onePoneOmulDV}, we increase the complexity of the previous example by imposing multiple dose--volume constraints on the same PTV and OAR.
In Section~\ref{sec:mulPTVmulOAR}, we show how the approach handles multiple PTVs and OARs.
Finally, in Section~\ref{sec:compare}, we compare our approach to other published methods.
For all of our examples, we use seven equally spaced beams with angles ranging from $0^\circ$ to $312^\circ$ in increments of $52^\circ$.
We set our regularization parameter at $\lambda = 10^{-8}$, our stopping tolerance at $\epsilon = 10^{-3}$, and for simplicity we let $\alpha_i = \alpha_j = 1$ for all $i \in \mathcal{I}$ and all $j \in \mathcal{J}^\text{dv}$, unless otherwise stated.
Whenever there is overlap between PTVs and OARs, the voxels in common are assigned to the PTV.
However, this choice is not required for our model, and a single voxel may be assigned to multiple structures if desired.
All examples were run on a computer with a 2.9 GHz dual-core Intel Core i5 processor with 8 GB RAM.

\begin{figure}
    \centering
    \begin{overpic}[scale=0.45]{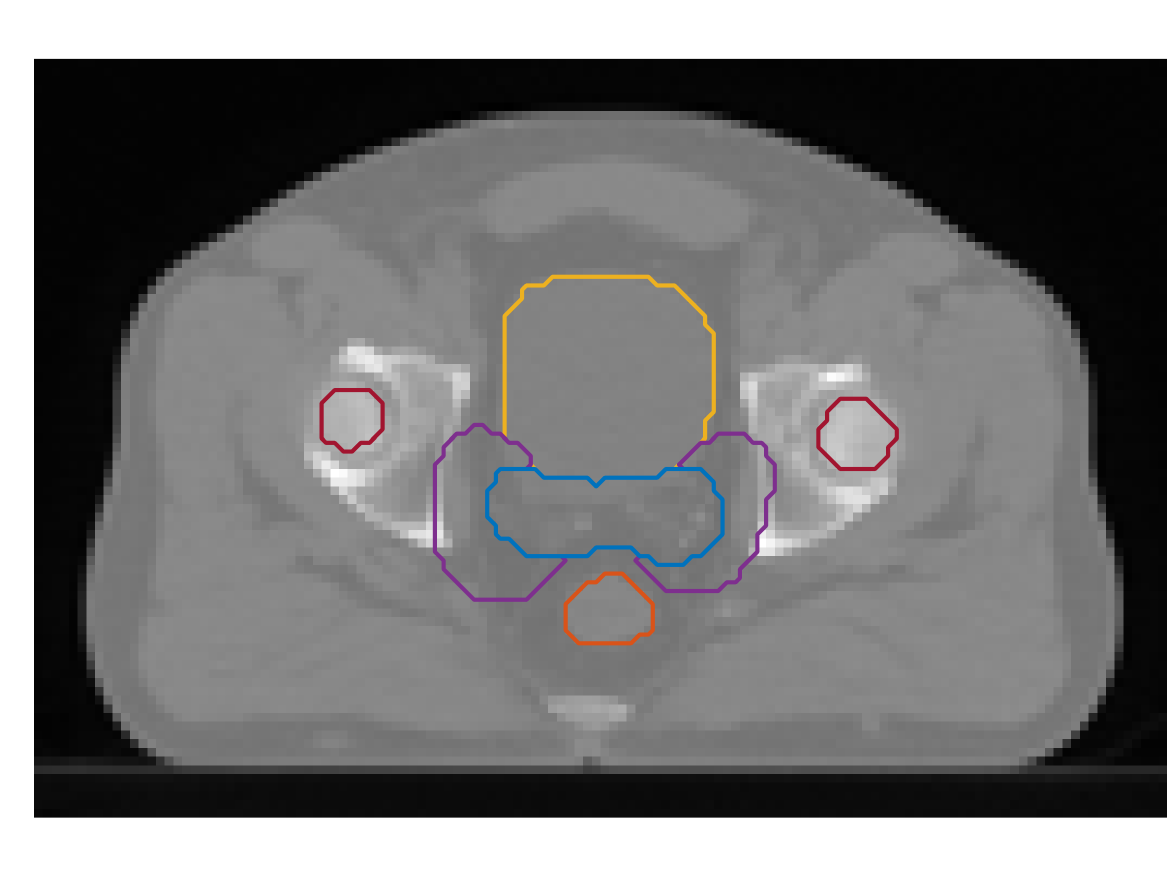}
        \put(44.5,30){\small\color{prostate}Prostate}
        \put(67.5,30){\small\color{nodes}Lymph}
        \put(67.5,26){\small\color{nodes}Nodes}
        \put(45,16){\small\color{rectum}Rectum}
        \put(45,45){\small\color{bladder}Bladder}
        \put(15.5,43){\small\color{femheads}Femoral}
        \put(15.5,39){\small\color{femheads}Heads}
        \put(36,57){\small Unspecified Tissue}
    \end{overpic}
    \vspace{-2ex}
    \caption{Organs from the CORT prostate case on a representative axial CT slice.
    The prostate corresponds to planning target volume PTV\_68 in the CORT dataset, while the lymph nodes correspond to planning target volume PTV\_56.\label{fig:CORT}}
\end{figure}

We use the cumulative dose--volume histogram as a way of evaluating the quality of the resulting treatment plans.
In a dose--volume histogram, each point represents the percent of the organ volume that receives at least a particular dose.
For example, we aim to deliver a uniform dose of 81 Gy to the prostate tumor, corresponding to the dose--volume histogram on the left panel of Figure~\ref{fig:ptvDvh}.
In this case, 100\% of the organ volume receives exactly 81 Gy.
As we cannot deliver a perfectly uniform dose, our goal is to find a treatment plan that is as close as possible to this idealized dose profile.
For instance, in the right panel of Figure~\ref{fig:ptvDvh} we see the dose--volume histogram of the solution to \eqref{eq:init} for one prostate tumor with regularization parameter $\lambda = 10^{-8}$.
We use the beamlet intensities from this solution as our initialization for our examples in Sections~\ref{sec:onePoneOoneDV}, \ref{sec:onePoneOmulDV}, and \ref{sec:compareOne}.

\begin{figure}
    \centering
    \hspace{1em}
    \begin{overpic}[scale=0.4]{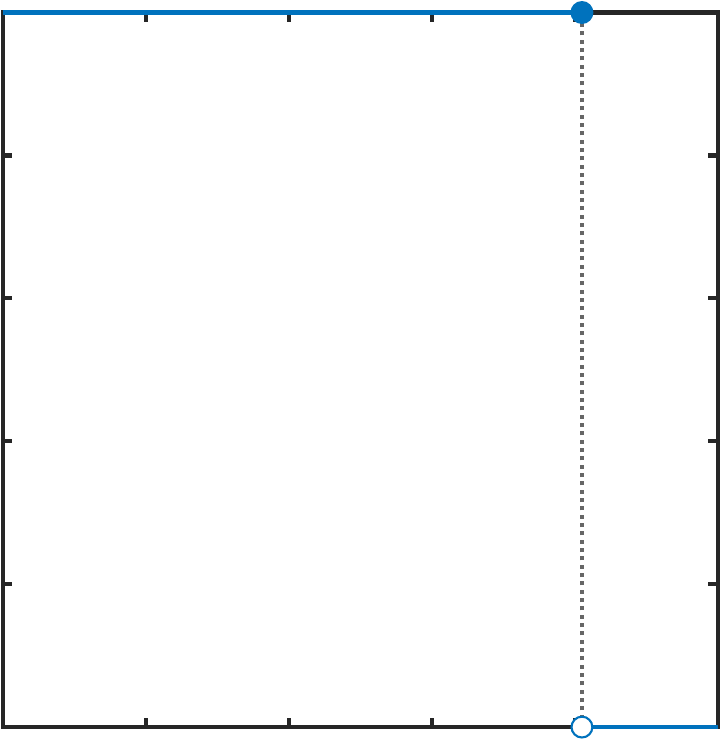}
        \put(-1,-4){\scriptsize0}
        \put(16,-4){\scriptsize20}
        \put(35.5,-4){\scriptsize40}
        \put(55,-4){\scriptsize60}
        \put(75,-4){\scriptsize80}
        \put(92,-4){\scriptsize100}
        \put(-4,0){\scriptsize0}
        \put(-7,19){\scriptsize20}
        \put(-7,38){\scriptsize40}
        \put(-7,58){\scriptsize60}
        \put(-7,77){\scriptsize80}
        \put(-10,96){\scriptsize100}
        \put(34,-12){\small Dose (Gy)}
        \put(-20,18){\rotatebox{90}{\small Relative Volume (\%)}}
    \end{overpic} \hspace{1em}
    \begin{overpic}[scale=0.4]{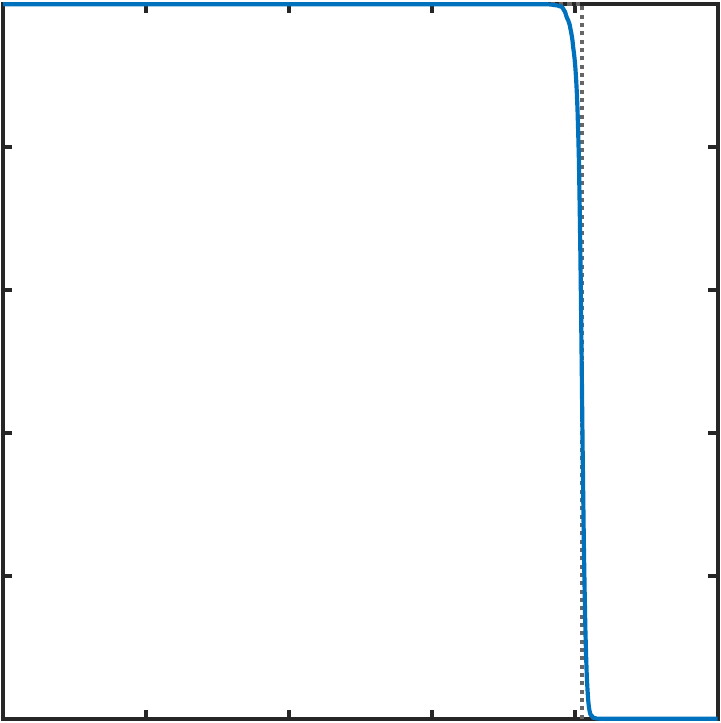}
        \put(-1,-5){\scriptsize0}
        \put(17,-5){\scriptsize20}
        \put(36.5,-5){\scriptsize40}
        \put(57,-5){\scriptsize60}
        \put(76,-5){\scriptsize80}
        \put(94,-5){\scriptsize100}
        \put(35,-13){\small Dose (Gy)}
    \end{overpic}
    \vspace{4ex}
    \caption{Dose--volume histograms for a uniform PTV prescription of 81 Gy.
    Left: An ideal dose profile, with 100\% of the target volume receiving exactly 81 Gy.
    Right: A more realistic dose profile, corresponding to our initialization for Sections~\ref{sec:onePoneOoneDV} and \ref{sec:onePoneOmulDV}, obtained by solving \eqref{eq:init} for one prostate tumor with regularization parameter $\lambda = 10^{-8}$.\label{fig:ptvDvh}}
\end{figure}

Cumulative dose--volume histograms can be used to determine whether or not an upper dose--volume constraint on an OAR has been met by a particular treatment plan.
For example, we consider the dose--volume constraint that no more than 50\% of the rectum volume exceeds 50 Gy.
In Figure~\ref{fig:oarDvh}, the rectangular region within the dotted lines contains all points where at most 50\% of the rectum volume receives more than 50 Gy.
Therefore any dose curve that lies outside of the box does not meet the constraint, while any curve that travels inside the box meets the constraint.
In the left panel of Figure~\ref{fig:oarDvh}, approximately 57\% of the rectum volume receives more than 50 Gy, so the dose--volume constraint is not met.
However, on the right panel of Figure~\ref{fig:oarDvh} approximately 42\% of the rectum volume receives more than 50 Gy, so the dose--volume constraint is met. 

\begin{figure}
    \centering
    \hspace{1em}
    \begin{overpic}[scale=0.4]{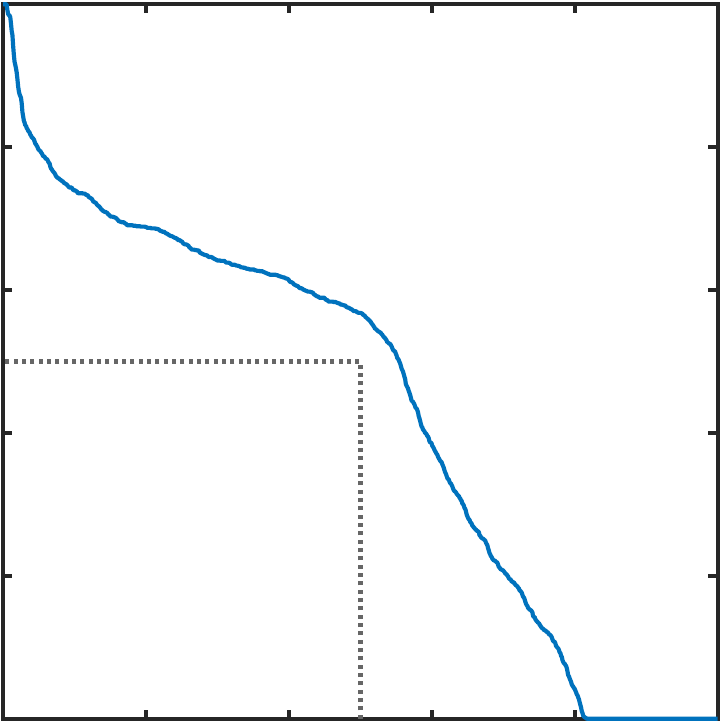}
        \put(-1,-5){\scriptsize0}
        \put(17,-5){\scriptsize20}
        \put(36.5,-5){\scriptsize40}
        \put(57,-5){\scriptsize60}
        \put(76,-5){\scriptsize80}
        \put(94,-5){\scriptsize100}
        \put(-4,-1){\scriptsize0}
        \put(-7,18){\scriptsize20}
        \put(-7,38){\scriptsize40}
        \put(-7,58){\scriptsize60}
        \put(-7,78){\scriptsize80}
        \put(-10,97){\scriptsize100}
        \put(35,-13){\small Dose (Gy)}
        \put(-20,18){\rotatebox{90}{\small Relative Volume (\%)}}
    \end{overpic} \hspace{1em}
    \begin{overpic}[scale=0.4]{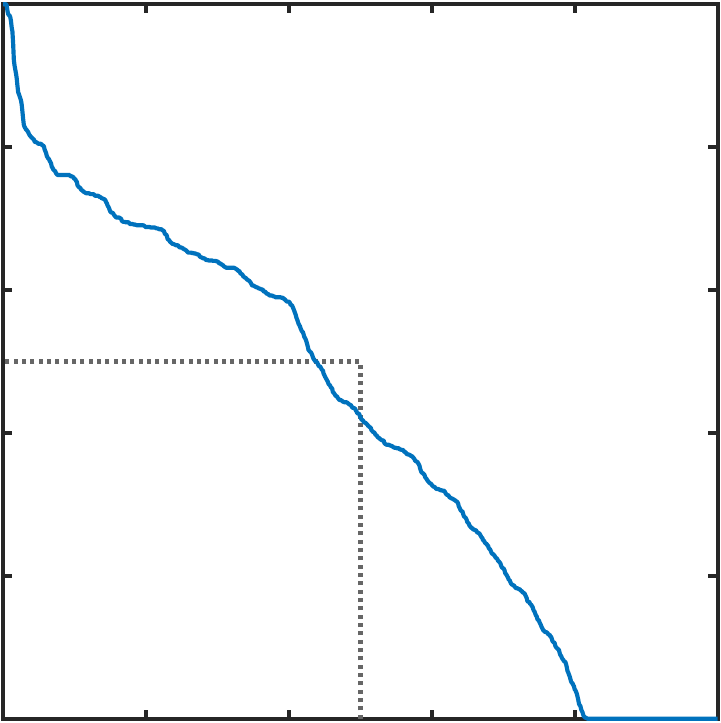}
        \put(-1,-5){\scriptsize0}
        \put(17,-5){\scriptsize20}
        \put(36.5,-5){\scriptsize40}
        \put(57,-5){\scriptsize60}
        \put(76,-5){\scriptsize80}
        \put(94,-5){\scriptsize100}
        \put(35,-13){\small Dose (Gy)}
    \end{overpic}
    \vspace{4ex}
    \caption{Dose--volume histograms for an OAR dose--volume constraint that no more than 50\% of the rectum volume exceeds 50 Gy.
    The solid lines correspond to the dose that the organ receives, and the dotted lines enclose the region that satisfies the dose--volume constraint.
    Left: A treatment plan that does not meet the dose--volume constraint, with approximately 57\% of the organ volume exceeding 50 Gy.
    Right: A treatment plan that does meet the dose--volume constraint, with approximately 42\% of the organ volume exceeding 50 Gy.\label{fig:oarDvh}}
\end{figure}

\subsection{One PTV and one OAR with one dose--volume constraint}
\label{sec:onePoneOoneDV}

For our first example, we consider the problem of delivering a uniform dose of 81 Gy to a prostate tumor, while satisfying the dose--volume constraint that no more than 30\% of the rectum volume receives more than 30 Gy.
In this case, the PTV has 6770 voxels, the OAR has 1764 voxels, and there are 986 beamlets.
In Figure~\ref{fig:dose1}, we see a representative axial slice of the calculated dose along with the intensities of four of the beams (due to symmetry, the remaining three beams have similar intensity patterns).
In this slice, a nearly uniform dose of 81 Gy is delivered to the tumor, while only a small percentage of the rectum volume receives more than 30 Gy.

\begin{figure}
    \centering
    \begin{minipage}[T]{0.45\textwidth} 
        \begin{overpic}[scale=0.45]{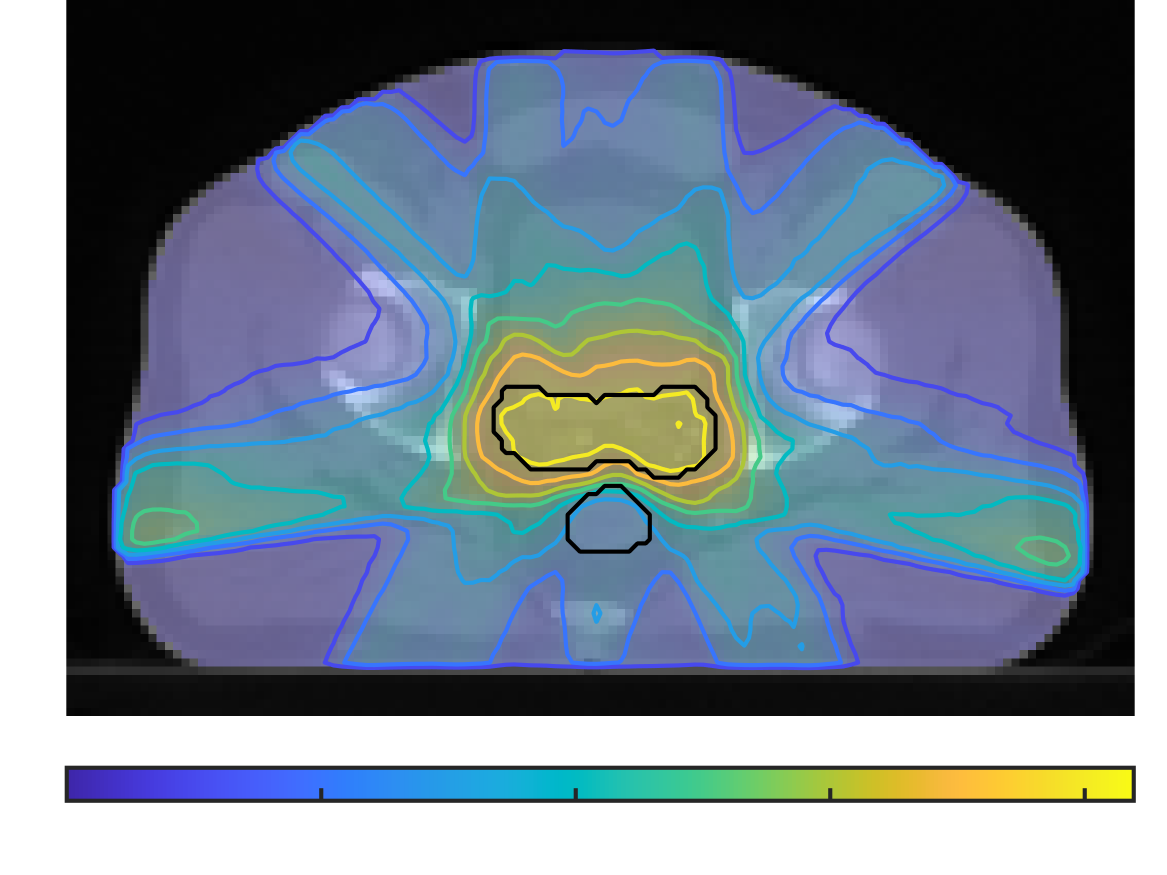}
        	\put(5,3){\scriptsize0}
        	\put(25.5,3){\scriptsize20}
        	\put(47.5,3){\scriptsize40}
        	\put(69.5,3){\scriptsize60}
        	\put(91.5,3){\scriptsize80}
            \put(43,-1){\scriptsize Dose (Gy)}
        	\put(44.5,37){\small Prostate}
        	\put(45,24){\small Rectum}
        	\put(51,71.5){\small\color{white}$0^\circ$}
        	\put(77,64){\small\color{white}$52^\circ$}
        	\put(89,29){\small\color{white}$104^\circ$}
        	\put(64,14.5){\small\color{white}$156^\circ$}
    	\end{overpic}
    \end{minipage} \hspace{4em}
    \begin{minipage}[T]{0.4\textwidth}
    	\begin{overpic}[scale=0.425]{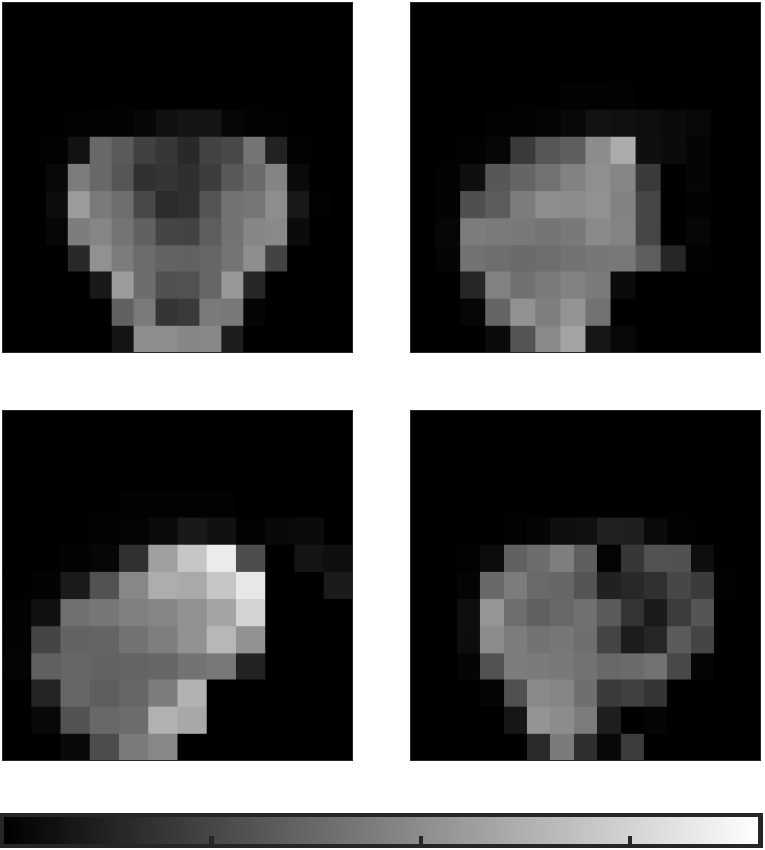}
        	\put(2,94){\small\color{white}$0^\circ$}
        	\put(50,94){\small\color{white}$52^\circ$}
        	\put(2,45.5){\small\color{white}$104^\circ$}
        	\put(50,45.5){\small\color{white}$156^\circ$}
        	\put(-0.5,-5){\scriptsize0}
        	\put(19.5,-5){\scriptsize1000}
        	\put(44.5,-5){\scriptsize2000}
        	\put(69,-5){\scriptsize3000}
            \put(20,-11){\scriptsize Beamlet Intensity (MU)}
    	\end{overpic}
    \vspace{3.5ex}
    \end{minipage} \vspace{2ex}
    \caption{Solution for example in Section~\ref{sec:onePoneOoneDV}.
    Left: The prostate tumor is irradiated by seven equally spaced beams with a dose--volume constraint on the rectum.
    Right: Beamlets intensities for four of the seven beams, calculated to deliver a nearly uniform dose of 81 Gy to the PTV while ensuring that no more than 30\% of the OAR volume receives more than 30 Gy.
    Due to symmetry, the remaining three beams have similar intensity patterns.\label{fig:dose1}}
\end{figure}

Figure~\ref{fig:obj1} shows the behavior of our model using Algorithm~\ref{algo:bcd}.
The objective $g(w)$ decreases monotonically (top left), and Algorithm~\ref{algo:bcd} converges after 42 iterations.
The trade-off between the competing PTV and OAR objectives is shown in the remaining left panels of Figure~\ref{fig:obj1}.
Because our initial beamlet intensities correspond to a solution with no dose--volume constraints, at initialization the tumor term achieves its lowest value and the rectum term achieves its highest value.
Once we begin iterating, the inclusion of the OAR term causes the tumor term value to increase and the rectum term to decrease.
On the right panel of Figure~\ref{fig:obj1}, we show the behavior of the auxiliary variable $w$ (top right) and our approximation of the dose--volume constraint (bottom right).
The percent of rectum voxels exceeding 30 Gy approaches the dose--volume constraint; however, because of the relaxation, we do not meet the dose--volume constraint exactly.
Nevertheless, we improve upon the initialization value according to the trade-off implicitly specified by our weights $\alpha_1$ and $\alpha_2$.

\begin{figure}
    \centering
    \begin{minipage}[T]{0.45\textwidth} \hspace{4em}
        \begin{overpic}[scale=0.4]{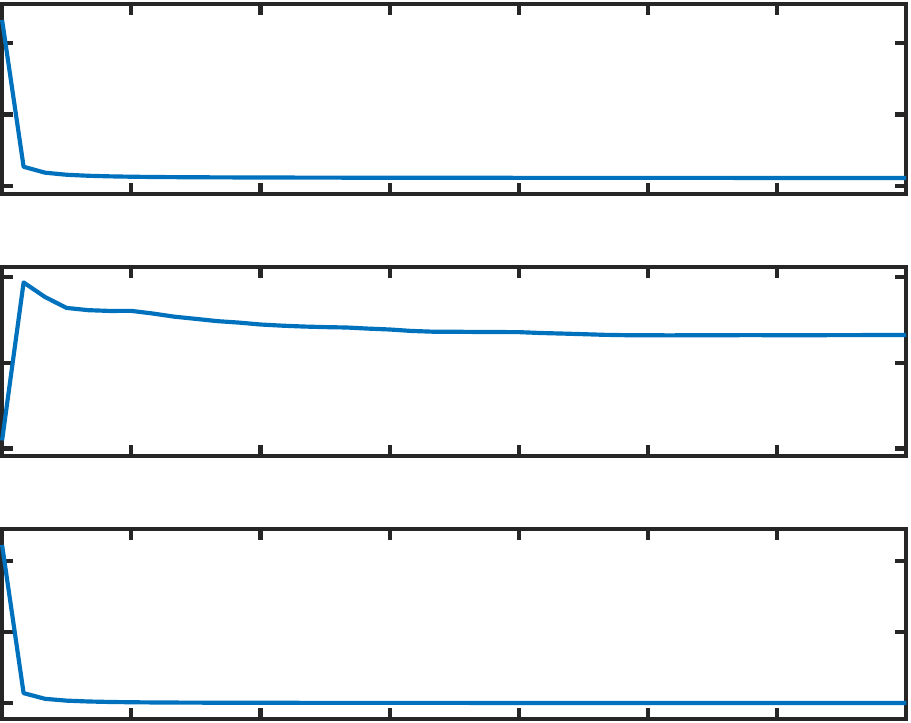}
            \put(-1,-4){\scriptsize0}
        	    \put(13,-4){\scriptsize6}
        	    \put(26,-4){\scriptsize12}
        	    \put(40,-4){\scriptsize18}
        	    \put(54.5,-4){\scriptsize24}
        	    \put(68.5,-4){\scriptsize30}
        	    \put(83,-4){\scriptsize36}
        	    \put(97,-4){\scriptsize42}
        	    \put(-3,0.5){\scriptsize0}
        	    \put(-5.5,8.5){\scriptsize50}
            \put(-8,16){\scriptsize100}
        	    \put(-7,28.5){\scriptsize0.2}
        	    \put(-7,38.5){\scriptsize0.4}
            \put(-7,47.5){\scriptsize0.6}
        	    \put(-3,57.5){\scriptsize0}
        	    \put(-5.5,65){\scriptsize50}
        	    \put(-8,73){\scriptsize100}
            \put(-13,-2){\rotatebox{90}{\scriptsize OAR Term}}
        	    \put(-13,28){\rotatebox{90}{\scriptsize PTV Term}}
        	    \put(-13,58.5){\rotatebox{90}{\scriptsize Objective}}
       	    \put(38,-9){\scriptsize Iteration (k)}
        \end{overpic}
    \end{minipage} \hspace{2em}
    \begin{minipage}[T]{0.45\textwidth}
        \hspace{1em}
    	\begin{overpic}[scale=0.4]{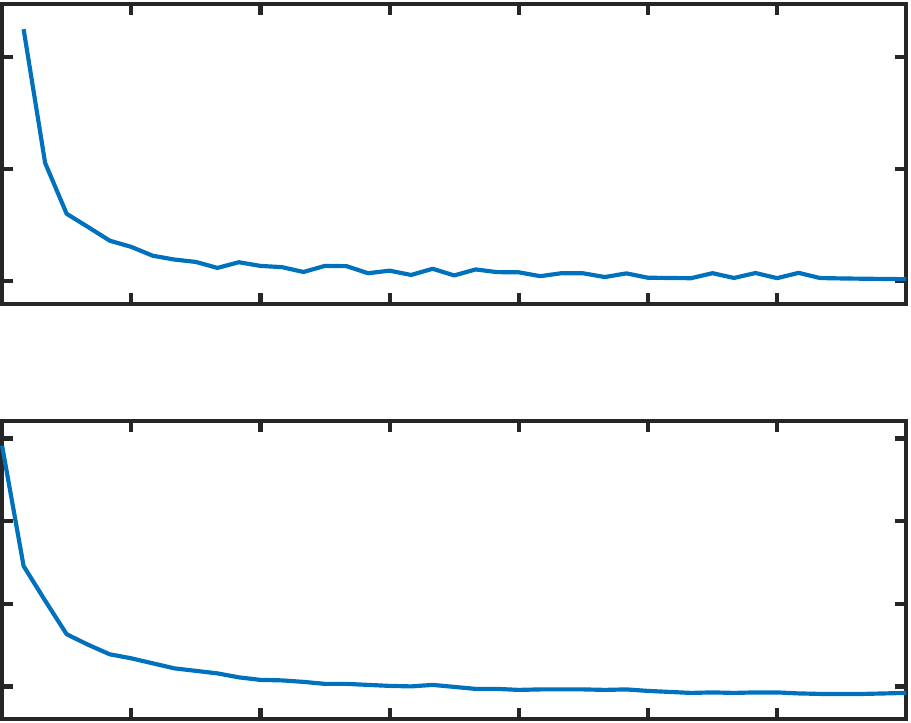}
        	    \put(-1,-4){\scriptsize0}
        	    \put(13,-4){\scriptsize6}
        	    \put(26,-4){\scriptsize12}
        	    \put(40,-4){\scriptsize18}
        	    \put(54.5,-4){\scriptsize24}
        	    \put(68.5,-4){\scriptsize30}
        	    \put(83,-4){\scriptsize36}
        	    \put(97,-4){\scriptsize42}
        	    \put(-5.5,2.5){\scriptsize35}
        	    \put(-5.5,11.5){\scriptsize45}
        	    \put(-5.5,20.5){\scriptsize55}
        	    \put(-5.5,29.5){\scriptsize65}
        	    \put(-3.5,47){\scriptsize0}
        	    \put(-3.5,59){\scriptsize5}
        	    \put(-5.5,71.5){\scriptsize10}
        	    \put(0,80){\tiny$\times10^{-2}$}
        	    \put(-10.5,46){\rotatebox{90}{\scriptsize Conv. Criteria}}
        	    \put(-11,-2){\rotatebox{90}{\scriptsize\% OAR $>$ 30 Gy}}
        	    \put(38,-9){\scriptsize Iteration (k)}
        \end{overpic}
    \end{minipage}
    \vspace{5ex}
    \caption{Convergence of example in Section~\ref{sec:onePoneOoneDV} using Algorithm~\ref{algo:bcd}.
    Left: Iterates of the objective function $g(w)$ and individual organ terms (PTV term = $\sfrac{1}{(2\times6770)} \| A_1 x - 81 \|_2^2$ and OAR term = $\sfrac{1}{(2\times1764)} \| w - (A_2x - 30) \|_2^2$).
    While both $g(w)$ (top) and the OAR term (bottom) decrease monotonically, the PTV term (middle) adjusts to accommodate the competing OAR objective.
    Right: Convergence of the auxiliary variable measured by $\sfrac{1}{t} \| w^{(k)} - w^{(k-1)} \|_2$ (top) and our approximation of the dose--volume constraint (bottom).\label{fig:obj1}}
\end{figure}

To demonstrate how the weight factors and dose--volume parameters can influence solutions, we implement a variation of the penalty decomposition method for the convex-cardinality problem proposed in \cite{lu2010penalty}.
Specifically, we use the iterative re-weighting scheme given in Algorithm~\ref{algo:reweight}, where ($\mathcal{P}_2$) is solved multiple times with increasing OAR weights, decreasing dose--volume parameters, and decreasing stopping tolerance values. 
In this example, we use initial OAR weight $\alpha_2^{(0)} = 1$, dose parameter $d_2^{\text{dv}(0)} = 30$, volume parameter $p_2^{(0)} = 30$, and stopping tolerance $\epsilon^{(0)} = 10^{-3}$, with weight update parameter $\sigma = 0.01$ and stopping tolerance update parameter $\gamma = 0.99$.

\begin{figure}[H]
    \centering
    \begin{minipage}{0.85\linewidth}
        \begin{algorithm}[H]
            \caption{Iterative re-weighting for Algorithm~\ref{algo:bcd}.\label{algo:reweight}}
            \begin{algorithmic}
                \State{Input $\sigma$, $\gamma$, $x^{(0)}$. Initialize $k = 0$, $\alpha_j^{(0)}$, $d_j^{\text{dv}(0)}$, $p_j^{(0)}$, $\epsilon^{(0)}$.}
                \While{not converged}
                    \State{$\left(x^{(k+1)},w^{(k+1)}\right) \leftarrow$ solution to ($\mathcal{P}_2$) from Algorithm~\ref{algo:bcd} with $\left(x^{(k)}, \alpha_j^{(k)}, d_j^{\text{dv}(k)}, p_j^{(k)}, \epsilon^{(k)}\right)$}
	            \For{$j \in \mathcal{J}^\text{dv}$}
	                \If{$\big\|\big(A_jx^{(k+1)} - d_j^{\text{dv}(0)}\big)_+\big\|_0 > \frac{n_jp_j^{(0)}}{100}$}
		            \State{$\alpha_j^{(k+1)} \leftarrow (1 + \sigma)\alpha_j^{(k)}$}
		            \State{$d_j^{\text{dv}(k+1)} \leftarrow (1 - \sigma)d_j^{\text{dv}(k)}$}
	                    \State{$p_j^{(k+1)} \leftarrow (1 - \sigma)p_j^{(k)}$}
	                \EndIf
	            \EndFor
	            \State{$\epsilon^{(k+1)} \leftarrow \gamma\epsilon^{(k)}$}
                    \State{$k \leftarrow k + 1$}
                \EndWhile
                \State \Return{$x^{(k)}$}
            \end{algorithmic}
        \end{algorithm}
    \end{minipage}
\end{figure}

We compare the approximation of the dose--volume constraint in our solutions with and without re-weighting in Table~\ref{tab:ex1} and Figure~\ref{fig:dvh1}.
To illustrate the trade-off between the competing PTV and OAR objectives, we include results from Algorithm~\ref{algo:reweight} using different convergence criteria: (a) stopping once the dose--volume constraint is met, and (b) stopping once the PTV D95 (i.e., the minimum dose delivered to 95\% of the PTV) is less than 98\% of its initial value, where weight factors and dose--volume parameters are updated regardless of whether or not the dose--volume constraint has been satisfied.
There is little difference between the initial tumor dose and the final tumor dose for each solution, both in terms of the PTV D95 and the dose--volume histogram.
At the same time, all solutions have reduced the dose to the rectum to approximately satisfy the dose--volume constraint, where re-weighting produces a lower dose at the expense of increased computation time.
Without re-weighting, the constraint is roughly met, with only 34.16\% of the volume receiving more than 30 Gy, in contrast to the initialization, where 64.14\% of the volume exceeds 30 Gy.
With re-weighting, the constraint is satisfied, with 29.61\% and 19.84\% of the volume receiving more then 30 Gy for stopping conditions (a) and (b), respectively.

\begin{figure}
    \centering
    \begin{table}[H]
        \centering
        \footnotesize
        \caption{Results for Section~\ref{sec:onePoneOoneDV}.\label{tab:ex1}}
        \begin{tabular}{cccc}
            \hline
            & PTV & OAR \\
            & D95 (Gy) & \% $>$ 30 Gy & Time (s) \\
            \hline
            Initialization & 79.65 & 64.14 & 0.18 \\
            Solution w/ Alg.~\ref{algo:bcd} & 79.17 & 34.16 & 7.22 \\
            Solution w/ Alg.~\ref{algo:reweight} (a) & 79.03 & 29.61 & 33.61 \\
            Solution w/ Alg.~\ref{algo:reweight} (b) & 78.05 & 19.84 & 163.51 \\
            \hline
        \end{tabular}
    \end{table}

    \begin{overpic}[scale=0.4]{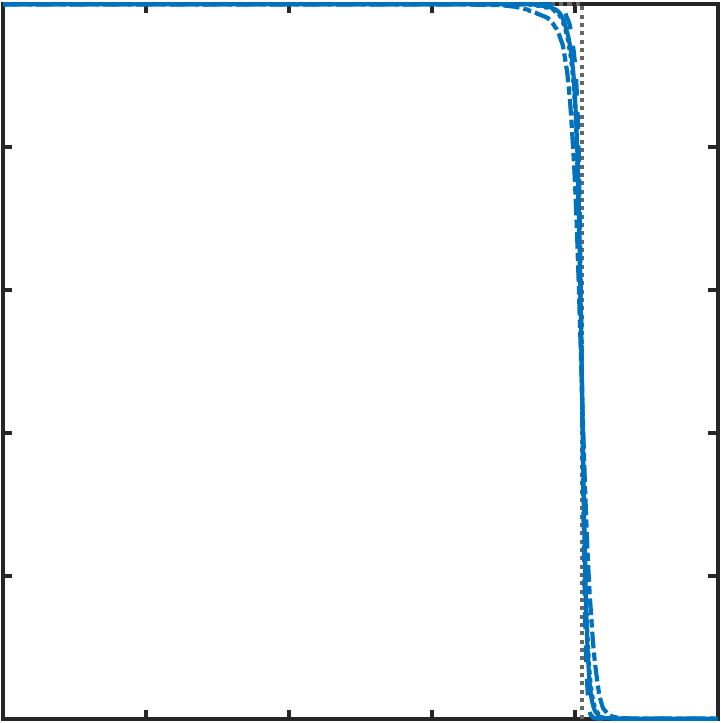}
        \put(-1,-5){\scriptsize0}
        \put(17,-5){\scriptsize20}
        \put(36.5,-5){\scriptsize40}
        \put(57,-5){\scriptsize60}
        \put(76,-5){\scriptsize80}
        \put(94,-5){\scriptsize100}
        \put(-4,-1){\scriptsize0}
        \put(-7,18){\scriptsize20}
        \put(-7,38){\scriptsize40}
        \put(-7,58){\scriptsize60}
        \put(-7,78){\scriptsize80}
        \put(-10,97){\scriptsize100}
        \put(35,-13){\small Dose (Gy)}
        \put(-20,18){\rotatebox{90}{\small Relative Volume (\%)}}
        \put(6,90){\small Prostate}
    \end{overpic} \hspace{1em}
    \begin{overpic}[scale=0.4]{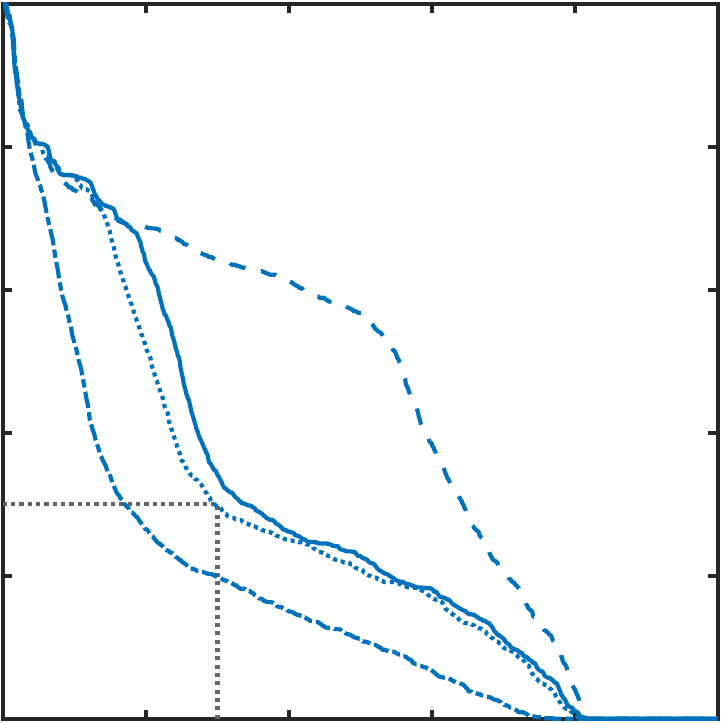}
        \put(-1,-5){\scriptsize0}
        \put(17,-5){\scriptsize20}
        \put(36.5,-5){\scriptsize40}
        \put(57,-5){\scriptsize60}
        \put(76,-5){\scriptsize80}
        \put(94,-5){\scriptsize100}
        \put(35,-13){\small Dose (Gy)}
        \put(70,90){\small Rectum}
        \put(111,90){\color{prostate}\bf-}
        \put(116,90){\color{prostate}\bf-}
        \put(122,90){\scriptsize Initialization}
        \put(111,80){\color{prostate}\bf-}
        \put(113.5,80){\color{prostate}\bf-}
        \put(116,80){\color{prostate}\bf-}
        \put(122,80){\scriptsize Solution w/ Alg.~\ref{algo:bcd}}
        \put(111,71){\color{prostate}\bf...}
        \put(122,70){\scriptsize Solution w/ Alg.~\ref{algo:reweight} (a)}
        \put(111,61){\color{prostate}\bf.}
        \put(114,59.75){\color{prostate}\bf-}
        \put(117,61){\color{prostate}\bf.}
        \put(122,60){\scriptsize Solution w/ Alg.~\ref{algo:reweight} (b)}
    \end{overpic}
    \vspace{4ex}
    \caption{Dose--volume histograms for Section~\ref{sec:onePoneOoneDV} with and without re-weighting.
    Left: There is little notable difference between the initial and final dose to the PTV.
    Right: Without re-weighting, the constraint that no more than 30\% of the OAR volume receives more than 30 Gy is approximately met, with 34.16\% of the volume exceeding 30 Gy, as opposed to the initialization where 64.14\% does.
    With re-weighting, 29.61\% and 19.84\% the volume exceeds 30 Gy for stopping conditions (a) and (b), respectively.\label{fig:dvh1}}
\end{figure}

In general, as the dose to the rectum decreases, the PTV D95 decreases, while the dose to the tumor deviates more from the target uniform dose.
If the dose--volume constraint makes the problem completely infeasible, the approach still returns some solution, but it need not meet the dose--volume constraint nor the target tumor dose.
In this case, the planner can influence the trade-off between treatment goals by modifying weighting factors, dose--volume parameters, and convergence criteria.

\subsection{One PTV and one OAR with multiple dose--volume constraints}
\label{sec:onePoneOmulDV}

Next we consider the problem of delivering a uniform dose of 81 Gy to the prostate tumor, with multiple dose--volume constraints on the tumor and rectum:
\begin{itemize}
    \item No more than 5\% of the tumor volume receives less than 81 Gy,
    \item no more than 0\% of the tumor volume receives more than 85 Gy,
    \item no more than 20\% of the rectum volume receives more than 60 Gy,
    \item no more than 40\% of the rectum volume receives more than 40 Gy, and
    \item no more than 60\% of the rectum volume receives more than 20 Gy.
\end{itemize}

The first constraint, a lower dose--volume constraint on the PTV, is implemented in the same way as the upper dose--volume constraint on the OAR, with the exception that the corresponding term in the objective function is changed to
\begin{equation}
    \frac{\alpha_i}{2n_i} \| w_i - (d_i^{\text{dv}} - A_ix) \|_2^2,
\end{equation}
and the projection onto $\Omega_i$ is computed by ordering the elements of $w_i$ and setting the highest $\lceil\sfrac{(100-p_i)n_i}{100}\rceil$ entries of $w_i$ to $\max\{0,w_i\}$.
The second constraint corresponds to a maximum-dose constraint, but expressing it as an upper dose--volume constraint allows us to apply our relaxation rather than enforcing a hard constraint.

As in the previous example, the PTV has 6770 voxels, the OAR has 1764 voxels, and there are 986 beamlets.
One downside to adding multiple dose--volume constraints on a particular structure is the inclusion of more terms in the objective function, increasing the problem dimension. 
However, we can reduce the number of objective terms by combining each of the $K_j$ dose--volume constraint terms corresponding to organ $j$,
\begin{equation}
    \frac{\alpha_{j_k}}{2n_j} \| w_{j_k} - (A_j x - d_{j_k}^\text{dv}) \|_2^2 \quad \text{s.t.} \quad \| (w_{j_k})_+ \|_0 \leq \frac{n_jp_{j_k}}{100} \quad \text{for} \quad k = 1,\dots,K_j,
\end{equation}
into one term,
\begin{equation}
    \frac{\alpha_j}{2n_j} \| y_j - A_j x \|_2^2 \quad \text{s.t.} \quad y_j \in \tilde{\Omega}_j,\label{eq:multDVCs}
\end{equation}
where $\tilde{\Omega}_j = \tilde{\Omega}_{j_1} \cap \cdots \cap \tilde{\Omega}_{j_{K_j}}$ represents the set of vectors that satisfy all of the organ's $K_j$ dose--volume constraints, with each constraint defined as
\begin{equation}
    y_j \in \tilde{\Omega}_{j_k} := \left\{ y_j \in \mathbb{R}^{n_j} \,:\, \|(y_j - d_{j_k}^\text{dv})_+\|_0 \leq \frac{n_jp_{j_k}}{100}\right\} \quad\text{for}\quad k = 1,\dots,K_j.
\end{equation}
Projections onto $\tilde{\Omega}_j$ can be computed efficiently by ordering the elements of $y_j$, setting the lowest $(100 - p_{j_k})$\% of $y_j$ to $\min\{d_{j_k},y_j\}$ for upper dose--volume constraints, and setting the highest $(100 - p_{j_k})$\% of $y_j$ to $\max\{d_{j_k},y_j\}$ for lower dose--volume constraints, for all $k = 1,\dots,K_j$.

Using Algorithm~\ref{algo:bcd} with $\alpha_1 = 10$ for the PTV dose--volume constraint term, we converge after 21 iterations.
In Table~\ref{tab:ex2} and Figure~\ref{fig:dvh2}, we see that the dose--volume constraints have been approximately met.
Specifically, none of the PTV volume receives more than 85 Gy, 24.47\% of the PTV volume receives less than 81 Gy (53.26\% initially), and the PTV D95 increases from 79.65 Gy to 80.26 Gy.
Furthermore, 61.95\% of the OAR volume exceeds 20 Gy (68.81\% initially), 42.9\% exceeds 40 Gy (61.29\% initially), and 21.84\% exceeds 60 Gy (38.47\% initially).
Importantly, it is possible to specify the shape of the PTV and OAR dose--profiles in terms of multiple dose--volume constraints without any significant increase in computation.

\begin{figure}
    \centering
    \begin{table}[H]
        \centering
        \footnotesize
        \caption{Results for Section~\ref{sec:onePoneOmulDV}.\label{tab:ex2}}
        \begin{tabular}{ccccccc}
            \hline
            & PTV & PTV & OAR & OAR & OAR \\
            & D95 (Gy) & \% $<$ 81 Gy & \% $>$ 20 Gy & \% $>$ 40 Gy & \% $>$ 60 Gy & Time (s) \\
            \hline
            Initialization & 79.65 & 53.26 & 68.81 & 61.29 & 38.47 & 0.18 \\
            Solution w/ Alg.~\ref{algo:bcd} & 80.26 & 24.27 & 61.95 & 42.90 & 21.84 & 8.89 \\
            Solution w/ Alg.~\ref{algo:reweight} & 81.32 & 4.09 & 59.10 & 39.26 & 19.78 & 55.86 \\
            \hline
        \end{tabular}
    \end{table}

    \hspace{1em}
    \begin{overpic}[scale=0.4]{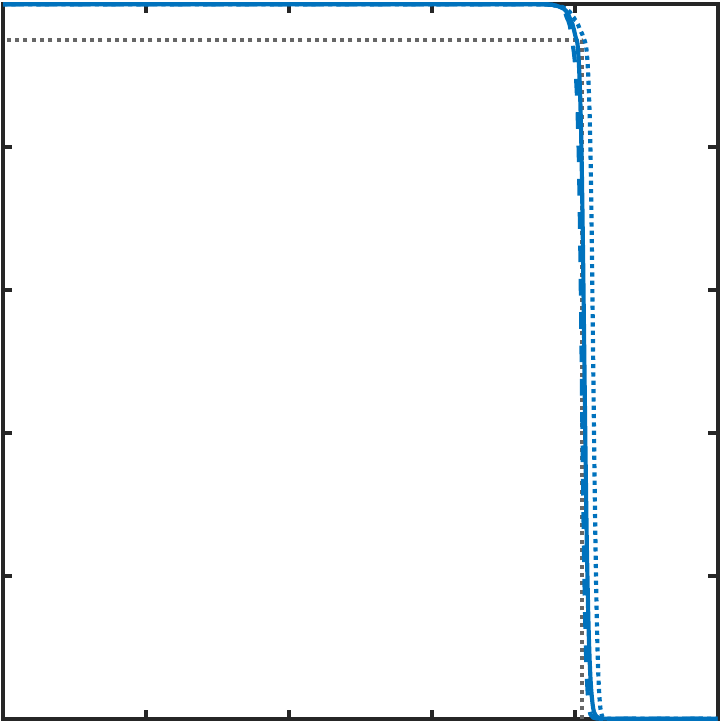}
        \put(-1,-5){\scriptsize0}
        \put(17,-5){\scriptsize20}
        \put(36.5,-5){\scriptsize40}
        \put(57,-5){\scriptsize60}
        \put(76,-5){\scriptsize80}
        \put(94,-5){\scriptsize100}
        \put(-4,-1){\scriptsize0}
        \put(-7,18){\scriptsize20}
        \put(-7,38){\scriptsize40}
        \put(-7,58){\scriptsize60}
        \put(-7,78){\scriptsize80}
        \put(-10,97){\scriptsize100}
        \put(35,-13){\small Dose (Gy)}
        \put(-20,18){\rotatebox{90}{\small Relative Volume (\%)}}
        \put(6,85){\small Prostate}
    \end{overpic} \hspace{1em}
    \begin{overpic}[scale=0.4]{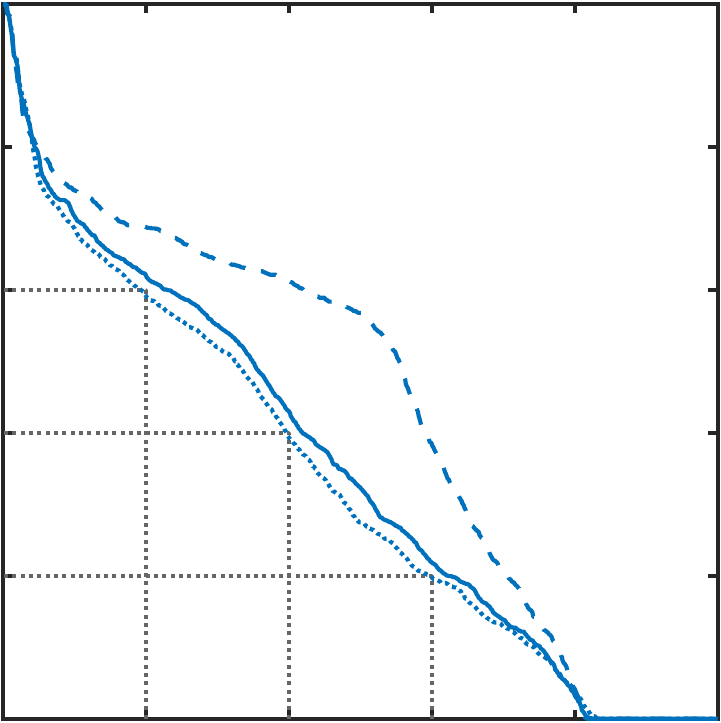}
        \put(-1,-5){\scriptsize0}
        \put(17,-5){\scriptsize20}
        \put(36.5,-5){\scriptsize40}
        \put(57,-5){\scriptsize60}
        \put(76,-5){\scriptsize80}
        \put(94,-5){\scriptsize100}
        \put(35,-13){\small Dose (Gy)}
        \put(70,90){\small Rectum}
        \put(111,90){\color{prostate}\bf-}
        \put(116,90){\color{prostate}\bf-}
        \put(122,90){\scriptsize Initialization}
        \put(111,80){\color{prostate}\bf-}
        \put(113.5,80){\color{prostate}\bf-}
        \put(116,80){\color{prostate}\bf-}
        \put(122,80){\scriptsize Solution w/ Alg.~\ref{algo:bcd}}
        \put(111,71){\color{prostate}\bf...}
        \put(122,70){\scriptsize Solution w/ Alg.~\ref{algo:reweight}}
    \end{overpic}
    \vspace{4ex}
    \caption{Dose--volume histograms for Section~\ref{sec:onePoneOmulDV}.
    Left: Without re-weighting, the dose--volume constraints on the PTV are approximately met, with 24.27\% of the volume receiving less than 81 Gy (53.26\% initially) and none of the volume exceeding 85 Gy.
    With re-weighting, only 4.09\% of the volume receives less than 81 Gy.
    Right: Without re-weighting, the dose--volume constraints on the OAR are approximately met, with 61.95\% of the volume exceeding 20 Gy (68.61\% initially), 42.9\% exceeding 40 Gy (61.29\% initially), and 21.84\% exceeding 60 Gy (38.47\% initially).
    With re-weighting, only 59.1\% of the volume exceeds 20 Gy, 39.25\% exceeds 40 Gy, and 19.78\% exceeds 60 Gy.\label{fig:dvh2}}
\end{figure}

Next, we use Algorithm~\ref{algo:reweight} with weight factors and dose--volume parameters initialized with the values used in Algorithm~\ref{algo:bcd}, initial stopping tolerance $\epsilon^{(0)} = 10^{-3}$, weight update parameter $\sigma = 0.01$, and stopping tolerance update parameter $\gamma = 0.99$.
For the lower dose--volume constraint on the PTV, the dose parameter increases each iteration,
\begin{equation}
    d_i^{\text{dv}(k+1)} = (1 + \sigma)d_i^{\text{dv}(k)},
\end{equation}
rather than decreasing like the upper dose--volume constraints.
Stopping once all dose--volume constraints are met, none of the PTV volume receives more than 85 Gy, 4.09\% of the PTV volume receives less than 81 Gy, and the PTV D95 is 81.32 Gy.
In addition, only 59.1\% of the OAR volume exceeds 20 Gy, 39.26\% exceeds 40 Gy, and 19.78\% exceeds 60 Gy.
In cases of infeasibility due to competing PTV and OAR dose--volume constraints, additional convergence criteria for Algorithm~\ref{algo:reweight} can be specified, e.g., setting a maximum number of iterations.

\subsection{Multiple PTVs and OARs}
\label{sec:mulPTVmulOAR}

Now we consider the problem of delivering a uniform dose of 81 Gy to the prostate tumor and a uniform dose of 60 Gy to the lymph nodes, while satisfying the dose--volume constraints that no more than 50\% of the rectum receives more than 50 Gy and no more than 30\% of the bladder receives more than 30 Gy.
In this case the first PTV has 6770 voxels, the second PTV has 8483 voxels, the first OAR has 1643 voxels, the second OAR has 9708 voxels, and there are 986 beamlets.
We converge after 48 iterations of Algorithm~\ref{algo:bcd}.

Due to the proximity of the two PTVs and the difference between their target doses, it is much more difficult to deliver a uniform dose to both targets.
The increase in geometric constraints and competing objectives results in more complicated dose distributions and beamlet intensity patterns, illustrated in Figure~\ref{fig:dose4}.
More of the prostate volume receives less than 81 Gy than in previous examples, while more of the lymph node volume exceeds 60 Gy than we would like (Figure~\ref{fig:dvh3}, top row).
In this case, it may be important to add lower and upper dose--volume constraints on the PTVs to ensure that appropriate doses are delivered.
The dose--volume constraints on the rectum and bladder are both approximately met, with 57.33\% of the rectum exceeding 50 Gy (82.47\% initially) and 38.14\% of the bladder exceeding 30 Gy (91.87\% initially) (Figure~\ref{fig:dvh3}, bottom row).

\begin{figure}
    \centering
    \begin{minipage}[T]{0.45\textwidth} 
        \begin{overpic}[scale=0.45]{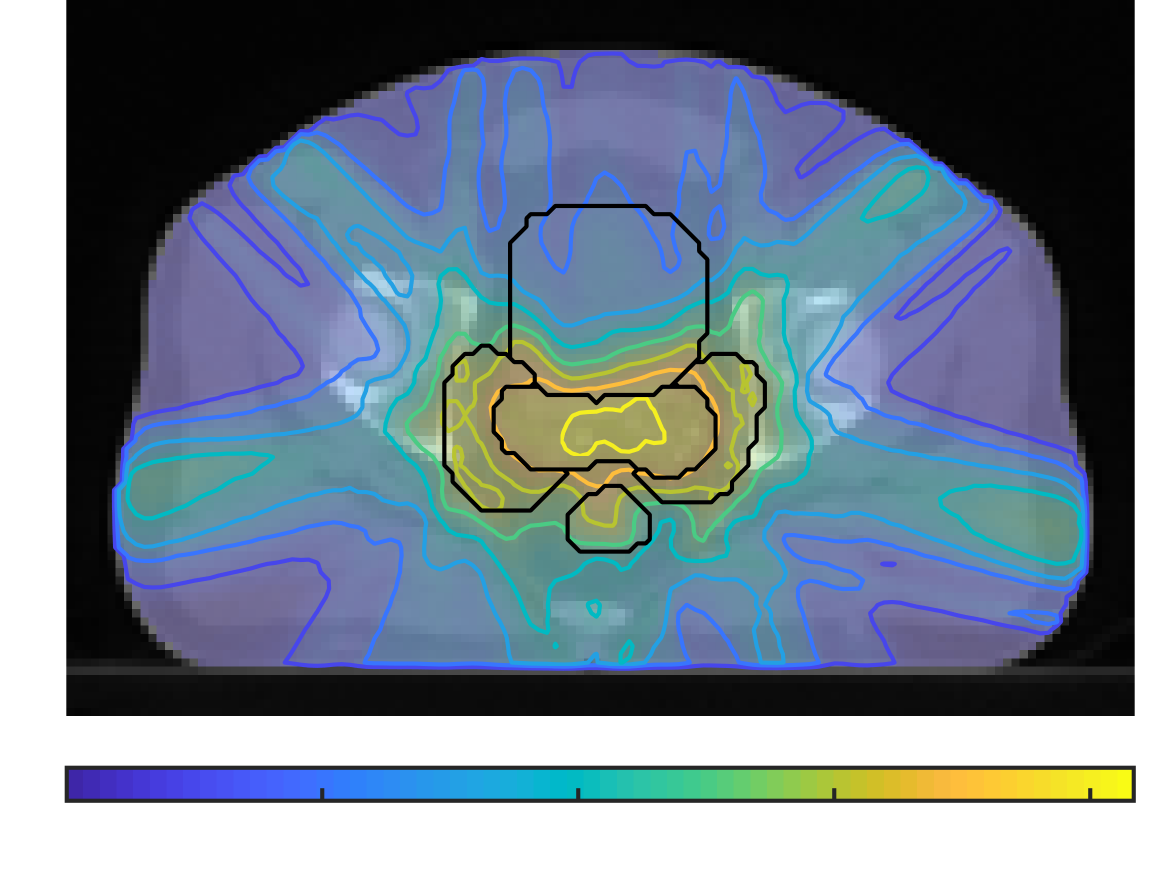}
        	    \put(5,3){\scriptsize0}
        	    \put(26,3){\scriptsize20}
        	    \put(47.5,3){\scriptsize40}
        	    \put(69.5,3){\scriptsize60}
        	    \put(91.5,3){\scriptsize80}
            \put(43,-1){\scriptsize Dose (Gy)}
        	    \put(44.5,37){\small Prostate}
        	    \put(45,24){\small Rectum}
        	    \put(66.5,39){\small Lymph}
        	    \put(66.5,35){\small Nodes}
        	    \put(45,51){\small Bladder}
        	    \put(51,71.5){\small\color{white}$0^\circ$}
        	    \put(79,63){\small\color{white}$52^\circ$}
        	    \put(89,27){\small\color{white}$104^\circ$}
            \put(61,14.5){\small\color{white}$156^\circ$}
    	\end{overpic}
    \end{minipage} \hspace{4em}
    \begin{minipage}[T]{0.4\textwidth}
    	\begin{overpic}[scale=0.425]{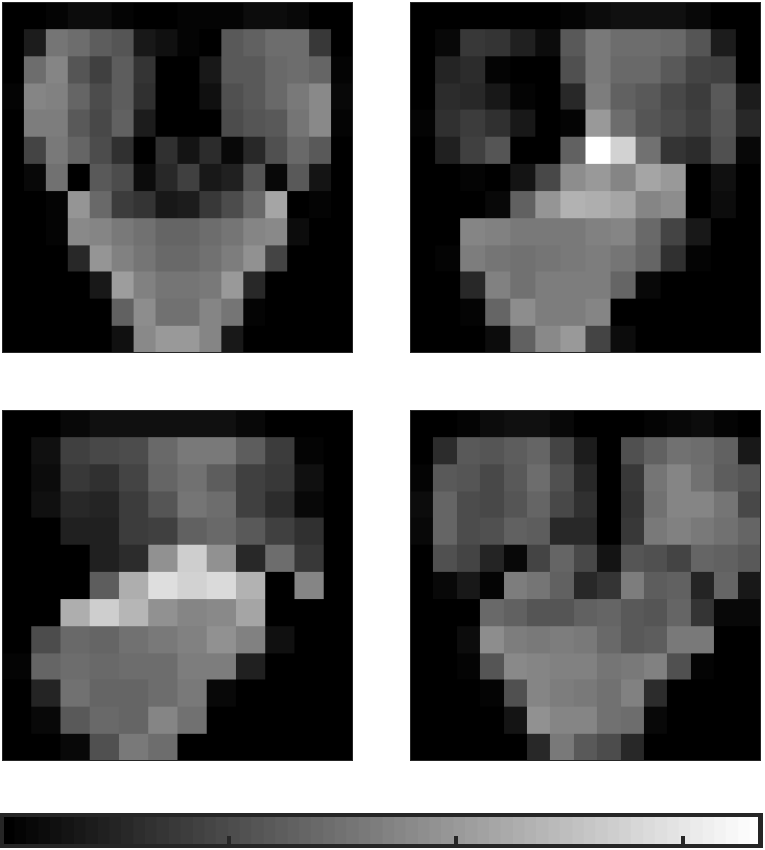}
        	    \put(2,94){\small\color{white}$0^\circ$}
        	    \put(50,94){\small\color{white}$52^\circ$}
        	    \put(2,45.5){\small\color{white}$104^\circ$}
        	    \put(50,45.5){\small\color{white}$156^\circ$}
        	    \put(-0.5,-5){\scriptsize0}
        	    \put(22,-5){\scriptsize1000}
        	    \put(49,-5){\scriptsize2000}
        	    \put(75.5,-5){\scriptsize3000}
            \put(20,-11){\scriptsize Beamlet Intensity (MU)}
    	\end{overpic}
        \vspace{3.2ex}
    \end{minipage}
    \vspace{2ex}
    \caption{Solution for example in Section~\ref{sec:mulPTVmulOAR}.
    Right: Targets in the prostate and lymph nodes are irradiated by seven equally spaced beams with dose--volume constraints on the rectum and bladder.
    Left: Beamlet intensities for four of the seven beams, calculated to deliver a nearly uniform dose of 81 Gy to the prostate and 60 Gy to the lymph nodes while ensuring that no more than 50\% of the rectum volume receives more than 50 Gy and no more than 30\% of the bladder volume receives more than 30 Gy.
    Due to symmetry, the remaining three beams have similar intensity patterns.\label{fig:dose4}}
\end{figure}

Using Algorithm~\ref{algo:reweight} with weight factors and dose--volume parameters initialized with the values used in Algorithm~\ref{algo:bcd}, initial stopping tolerance $\epsilon^{(0)} = 10^{-3}$, weight update parameter $\sigma = 0.01$, and stopping tolerance update parameter $\gamma = 0.99$, we are able to meet both of the dose--volume constraints on the OARs.
Specifically, only 47.53\% of the rectum volume receives more than 50 Gy and 29.93\% of the bladder volume exceeds 30 Gy, with little difference in the PTV doses.

\begin{figure}
    \centering
    \begin{table}[H]
        \centering
        \footnotesize
        \caption{Results for Section~\ref{sec:mulPTVmulOAR}.\label{tab:ex3}}
        \begin{tabular}{cccccc}
            \hline
            & PTV 1 & PTV 2 & OAR 1 & OAR 2 \\
            & D95 (Gy) & D95 (Gy) & \% $>$ 50 Gy & \% $>$ 30 Gy & Time (s) \\
            \hline
            Initialization & 76.96 & 58.53 & 82.47 & 91.87 & 0.37 \\
            Solution w/ Alg.~\ref{algo:bcd} & 76.70 & 57.51 & 57.33 & 38.14 & 15.52 \\
            Solution w/ Alg.~\ref{algo:reweight} & 76.48 & 57.31 & 47.53 & 29.93 & 84.83 \\
            \hline
        \end{tabular}
    \end{table}

    \begin{overpic}[scale=0.4]{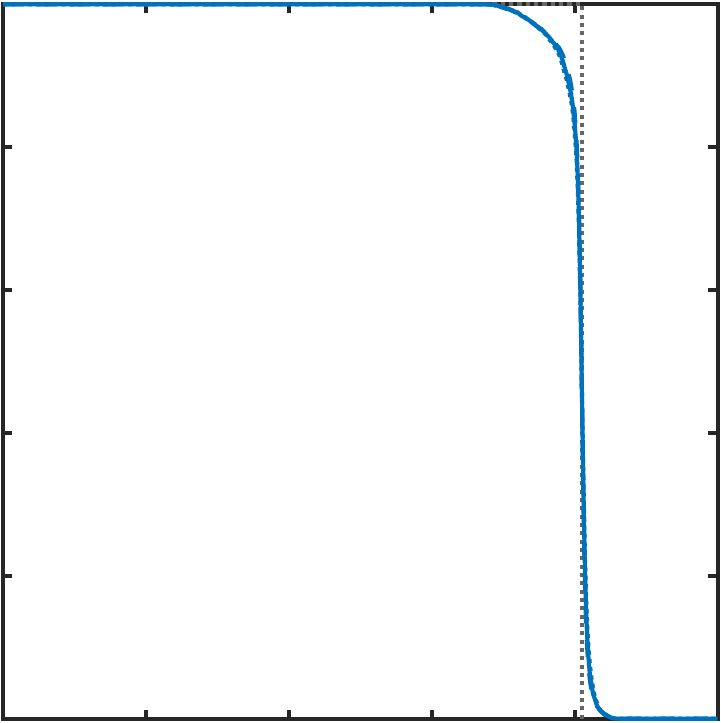}
        \put(-4,-1){\scriptsize0}
        \put(-7,18){\scriptsize20}
        \put(-7,38){\scriptsize40}
        \put(-7,58){\scriptsize60}
        \put(-7,78){\scriptsize80}
        \put(-10,97){\scriptsize100}
        \put(-20,18){\rotatebox{90}{\small Relative Volume (\%)}}
        \put(6,90){\small Prostate}
    \end{overpic} \hspace{0.5em}
    \begin{overpic}[scale=0.4]{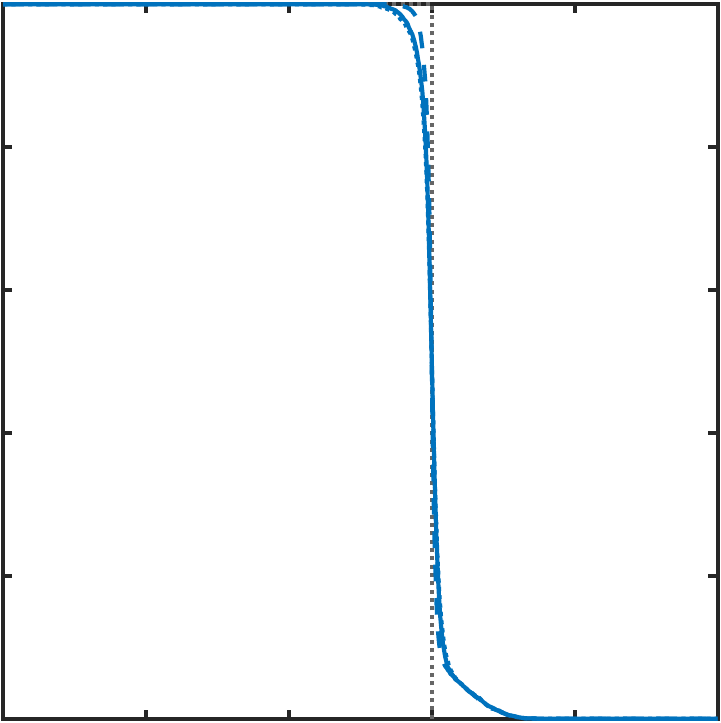}
        \put(6,90){\small Lymph Nodes}
        \put(111,90){\color{prostate}\bf-}
        \put(116,90){\color{prostate}\bf-}
        \put(122,90){\scriptsize Initialization}
        \put(111,80){\color{prostate}\bf-}
        \put(113.5,80){\color{prostate}\bf-}
        \put(116,80){\color{prostate}\bf-}
        \put(122,80){\scriptsize Solution w/ Alg.~\ref{algo:bcd}}
        \put(111,71){\color{prostate}\bf...}
        \put(122,70){\scriptsize Solution w/ Alg.~\ref{algo:reweight}}
    \end{overpic} \\ \vspace*{2.5ex} \hspace{0.1em}
    \begin{overpic}[scale=0.4]{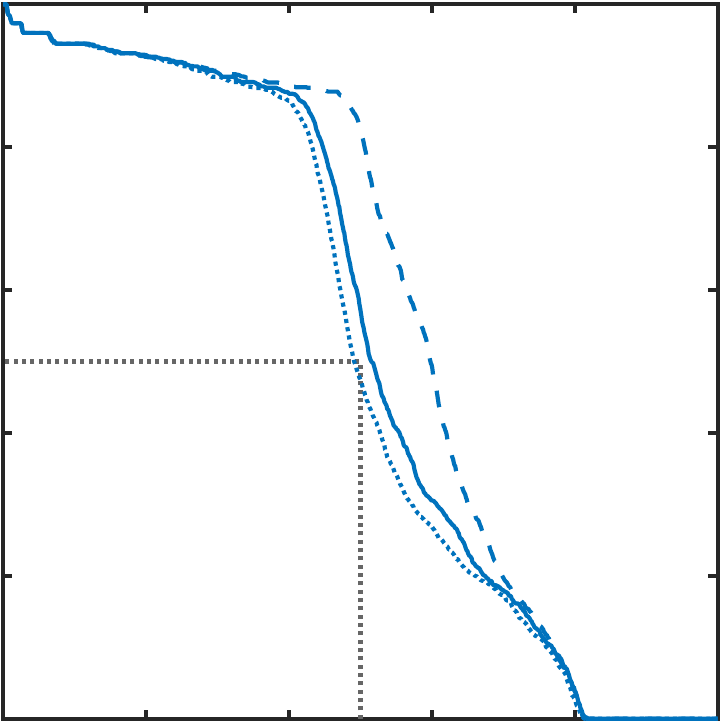}
        \put(-1,-5){\scriptsize0}
        \put(17,-5){\scriptsize20}
        \put(36.5,-5){\scriptsize40}
        \put(57,-5){\scriptsize60}
        \put(76,-5){\scriptsize80}
        \put(94,-5){\scriptsize100}
        \put(-4,-1){\scriptsize0}
        \put(-7,18){\scriptsize20}
        \put(-7,38){\scriptsize40}
        \put(-7,58){\scriptsize60}
        \put(-7,78){\scriptsize80}
        \put(-10,97){\scriptsize100}
        \put(35,-13){\small Dose (Gy)}
        \put(-20,18){\rotatebox{90}{\small Relative Volume (\%)}}
        \put(70,90){\small Rectum}
    \end{overpic} \hspace{0.5em}
    \begin{overpic}[scale=0.4]{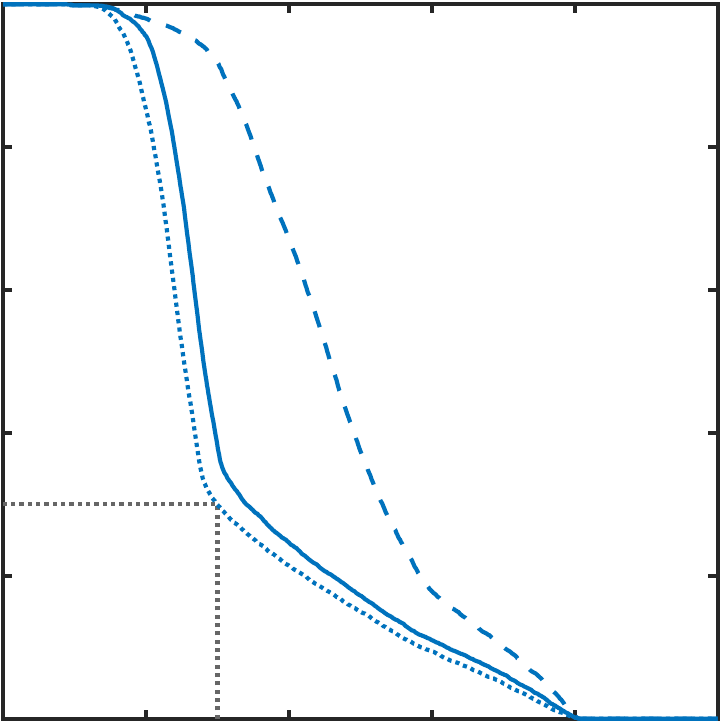}
        \put(-1,-5){\scriptsize0}
        \put(17,-5){\scriptsize20}
        \put(36.5,-5){\scriptsize40}
        \put(57,-5){\scriptsize60}
        \put(76,-5){\scriptsize80}
        \put(94,-5){\scriptsize100}
        \put(35,-13){\small Dose (Gy)}
        \put(70,90){\small Bladder}
    \end{overpic}
    \vspace{4ex}
    \caption{Dose--volume histograms for Section~\ref{sec:mulPTVmulOAR}.
    Top row: There is little noticeable difference between the dose to the PTVs in the prostate and lymph nodes, though their proximity causes less prostate voxels to receive their target dose and more lymph node voxels to exceed their target dose than we would like.
    Bottom row: Without re-weighting, the dose--volume constraints to the OARs are approximately met, with 57.33\% of the rectum exceeding 50 Gy (82.47\% initially) and 38.14\% of the bladder exceeding 30 Gy (91.87\% initially).
    With re-weighting, only 47.53\% of the rectum volume exceeds 50 Gy and 29.93\% of the bladder volume exceeds 30 Gy.\label{fig:dvh3}}
\end{figure}

\subsection{Comparisons with other methods}
\label{sec:compare}

Finally, we compare results from Algorithm~\ref{algo:bcd}, Algorithm~\ref{algo:reweight}, and four other methods for solving the FMO problem with dose--volume constraints.
We start with two methods that, like our approach, approximate dose--volume constraints with objective terms, and thus provide flexibility for cases of infeasiblity.
In particular, as mentioned in Section~\ref{sec:dvc}, many least-squares models include a penalty term for OAR voxels that exceed a specified dose (see e.g., \cite{ehrgott2010mathematical,bortfeld1990methods, wu2000algorithms,spirou1998gradient,llacer2001comparative, llacer2003absence,zhang2006fluence,zhang2008dose}).
This penalty may apply to the entire OAR, or in the case of dose--volume objectives, to a fraction of the OAR volume.
Following the approach used in \cite{spirou1998gradient, llacer2001comparative, llacer2003absence}, we replace our dose--volume terms in \eqref{eq:fObj} with
\begin{equation} \label{eq:iter}
    \sum_{j \in \mathcal{J}^{\text{dv}}}\frac{\alpha_j}{2n_j} \| (A_jx - \tilde{d}_j^{\text{dv}})_+ \|_2^2,
\end{equation}
where the vector $\tilde{d}_j^{\text{dv}}$ changes each iteration.
Specifically, after each iteration we sort the voxel dose vector $A_jx$, assigning the highest $p_j$\% to the dose value $\tilde{d}_j^{\text{dv}} = 10^{6}$ and the remaining $(100-p_j)$\% to $\tilde{d}_j^{\text{dv}} = d_j^{\text{dv}}$.
The beamlet vector $x$ is updated each iteration with a step of gradient descent, where only the the lowest $(100-p_j)$\% of voxels that exceed $\tilde{d}_j^\text{dv}$ influence the direction of the gradient, followed by a projection onto the nonnegative orthant.
If the vector $\tilde{d}_j^\text{dv}$ converges, then the problem becomes convex \cite{llacer2003absence}.

Alternatively, we can penalize voxels that exceed a given dose by adding nonnegative slack variables to the OAR terms.
For example, following the model used in the Sensitivity-Driven Greedy Algorithm \cite{zhang2006fluence, zhang2008dose}, we replace our dose--volume terms in \eqref{eq:fObj} with
\begin{equation} \label{eq:slack}
    \sum_{j \in \mathcal{J}^{\text{dv}}}\frac{\alpha_j}{2n_j} \| A_jx + s_j - y_j \|_2^2 \quad \text{s.t.} \quad s_j \geq 0, \quad y_j \in \tilde{\Omega}_j,
\end{equation}
which is similar to our modified OAR terms in \eqref{eq:multDVCs}.
The alternating projections approach of the Sensitivity-Driven Greedy Algorithm is equivalent to the block coordinate descent approach of Algorithm~\ref{algo:bcd}, with two important differences.
First, the nonnegative least-squares step to compute $x^{(k+1)}$ also includes the slack variables $s_j^{(k+1)}$, so the number of variables in this subproblem increases by $\sum_{j \in \mathcal{J}^\text{dv}}n_j$.
Second, to ensure the convergence of the dose variables, the $y^{(k+1)}$ update includes a projection onto the modified set \mbox{$\tilde{\Omega}_j\cap\{ y_j \in \mathbb{R}^{n_j} \,:\, y_j \geq y_j^{(k)}\}$}.
Using the initialization $y_j^{(0)} = d_j^\text{dv}$, this approach can be interpreted as iteratively selecting the voxels that may exceed $d_j^\text{dv}$, and once $\lfloor \sfrac{n_j p_j}{100} \rfloor$ elements of vector $y_j^{(k)}$ surpass $d_j^\text{dv}$, the problem becomes convex.

Next we look at two methods that solve a convex problem by imposing maximum-dose constraints on a chosen subvolume of each OAR. 
As mentioned in Section~\ref{sec:dvc}, many methods take a multi-stage approach of 1) solving an approximation of the original problem, 2) determining which voxels may exceed $d_j^\text{dv}$ based on the resulting solution, and 3) polishing the solution from Step 1 by solving a convex problem using the subvolumes determined in Step 2.
For example, one heuristic to choose the $p_j\%$ of the voxels that may receive more than $d_j^\text{dv}$ Gy is to solve the FMO problem without dose--volume constraints \cite{saberian2016optimal, saberian2016theoretical, hou2003optimization}.
Following this approach, we first solve $(\mathcal{P}_1)$ without any dose--volume constraints (which is equivalent to our initialization in \eqref{eq:init}), and then re-solve with maximum-dose constraints applied to the $(100 - p_j)$\% of the voxels that received the lowest dose in the previous solution.

Another heuristic is to solve an approximation of $(\mathcal{P}_1)$ where the dose--volume constraints in \eqref{eq:p1dvc} are replaced by the convex constraint
\begin{equation} \label{eq:convex}
    \|(\nu_j + A_jx - d_j^{\text{dv}})_+\|_1 \leq \frac{\nu_jn_jp_j}{100} \quad \text{s.t.} \quad \nu_j \geq 0,
\end{equation}
where $\nu_j$ is optimized along with the beamlet vector $x$ \cite{fu2019convex}.
If this problem is feasible, the solution is guaranteed to satisfy the dose--volume constraints, but it may not be optimal for $(\mathcal{P}_1)$.
Therefore, the problem is polished by re-solving with maximum-dose constraints applied to the $(100-p_j)$\% of the voxels that received the lowest dose in the approximate solution.
We compare results from these two convex methods with solutions from our re-weighting scheme in Algorithm~\ref{algo:reweight} and solutions from Algorithm~\ref{algo:bcd} which are then polished to satisfy the dose--volume constraints.

\subsubsection{One PTV and one OAR.}
\label{sec:compareOne}

First we consider the problem from Section~\ref{sec:onePoneOoneDV} with stopping tolerance $\epsilon = 10^{-2}$, where an additional lower dose--volume constraint that no more than 5\% of the tumor volume receives less than 81 Gy has been applied to the PTV.
In Table~\ref{tab:ex4a} and Figure~\ref{fig:dvh4a} (found in Appendix~\ref{sec:compareFigs}), we see the results of using approximate dose--volume objectives.
Because none of the solutions come close to satisfying the lower dose--volume constraint on the PTV, and only one satisfies the upper dose--volume constraint on the OAR, we look for the best candidate method for re-weighting or polishing.
Compared to the solution with \eqref{eq:iter}, our solution achieves a higher PTV D95 and a lower OAR dose in less time.
Because the method with \eqref{eq:iter} does not solve a subproblem to update the beamlet vector, it unsurprisingly requires less time per iteration than the other two methods, though it does require more iterations overall.
The method with \eqref{eq:slack} requires significantly more time than the other two approaches due to the additional slack variables, but it produces the only solution that meets the dose--volume constraint on the OAR.
However, its lower OAR dose comes at the expense of a lower PTV D95, and in less time we can use re-weighting or polishing with $(\mathcal{P}_2)$ to meet all of the dose--volume constraints (first two rows of Table~\ref{tab:ex4b}).

\begin{table}
    \centering
    \footnotesize
    \caption{Results for Section~\ref{sec:compareOne} with dose--volume objectives.\label{tab:ex4a}}
    \begin{tabular}{lccccc}
        \hline
        & PTV & PTV & OAR \\
        & D95 (Gy) & \% $<$ 81 Gy & \% $>$ 30 Gy & Iter. ($k$) & Time (s) \\
        \hline
        Solution w/ $(\mathcal{P}_2)$ & 79.65 & 55.32 & 38.35 & 5 & 1.83\\
        Solution w/ \eqref{eq:iter} & 78.74 & 44.33 & 56.55 & 173 & 5.20 \\
        Solution w/ \eqref{eq:slack} & 77.89 & 43.74 & 20.63 & 5 & 198.70 \\
        \hline
    \end{tabular}
\end{table}

Next we examine results computed through either re-weighting or polishing.
All four solutions in Table~\ref{tab:ex4b} and Figure~\ref{fig:dvh4b} satisfy the dose--volume constraints on the PTV and OAR, and are thus feasible solutions of $(\mathcal{P}_1)$.
However, the lowest objective value for \eqref{eq:justDvcEq} is achieved by solving $(\mathcal{P}_2)$ with Algorithm~\ref{algo:bcd}, then re-solving with maximum-dose constraints applied to subvolumes of the PTV and OAR.
The solution with \eqref{eq:init} has the highest objective value, perhaps because it is the only method that does not include any approximations of the dose--volume constraints in Step 1, while the solution with \eqref{eq:convex} requires the most time. 
The solution computed with Algorithm~\ref{algo:reweight}, on the other hand, requires the least amount of time, though it is not as good with respect to the objective value as the solution computed with Algorithm~\ref{algo:bcd} and polishing.

\begin{table}
    \centering
    \footnotesize
    \caption{Results for Section~\ref{sec:compareOne} with dose--volume constraints.\label{tab:ex4b}}
    \begin{tabular}{lccccc}
        \hline
        & PTV & PTV & OAR \\
        & D95 (Gy) & \% $<$ 81 Gy & \% $>$ 30 Gy & Obj. & Time (s) \\
        \hline
        Solution w/ $(\mathcal{P}_2)^a$ & 81.32 & 2.97 & 29.43 & 7.80 & 46.94 \\
        Solution w/ $(\mathcal{P}_2)^b$ & 81.00 & 4.37 & 29.43 & 6.95 & 65.56 \\
        Solution w/ \eqref{eq:init} & 81.00 & 3.26 & 28.40 & 9.12 & 63.60 \\
        Solution w/ \eqref{eq:convex} & 81.00 & 2.20 & 28.16 & 8.94 & 379.46 \\
        \hline
    \end{tabular}
    \begin{flushleft}\vspace{2.5ex}\hspace{9.5em}{\footnotesize $^a$Re-weighted \\ \hspace{9.5em} $^b$Polished}\end{flushleft}
\end{table}

\subsubsection{Two PTVs and two OARs.}
\label{sec:compareMultiple}

Next we consider the problem from Section~\ref{sec:mulPTVmulOAR} with stopping tolerance $\epsilon = 10^{-2}$, with additional lower dose--volume constraints that no more than 5\% of the tumor volumes receives less than 81 Gy and 60 Gy for the prostate and lymph nodes, respectively.
In Table~\ref{tab:ex5a} and Figure~\ref{fig:dvh5a}, we see the results of using approximate dose--volume objectives, with patterns similar to those in Table~\ref{tab:ex4a}.
Specifically, none of the solutions come close to meeting the lower dose--volume constraints on the PTVs, our solution from Algorithm~\ref{algo:bcd} achieves the highest PTV D95 values and requires the least amount of time, and the solution with \eqref{eq:slack} is the only method to satisfy the upper dose--volume constraints on the OARs, at the expense of significantly more time.
While the method with \eqref{eq:iter} produces a lower dose to the bladder, our method is still the best candidate for re-weighting or polishing, where it achieves solutions that meet all dose--volume constraints in less time than it takes the method with \eqref{eq:slack} to compute its approximate solution (first two rows of Table~\ref{tab:ex5b}).

\begin{table}
    \centering
    \footnotesize
    \caption{Results for Section~\ref{sec:compareMultiple} with dose--volume objectives.\label{tab:ex5a}}
    \begin{tabular}{lcccccccc}
        \hline
        & PTV 1 & PTV 1 & PTV 2 & PTV 2 & OAR 1 & OAR 2 \\
        & D95 (Gy) & \% $<$ 81 Gy & D95 (Gy) & \% $<$ 60 Gy & \% $>$ 50 Gy & \% $>$ 30 Gy & Iter. ($k$) & Time (s) \\
        \hline
        Solution w/ $(\mathcal{P}_2)$ & 76.95 & 57.86 & 57.51 & 52.95 & 64.58 & 49.47 & 3 & 2.80 \\
        Solution w/ \eqref{eq:iter} & 76.57 & 42.07 & 56.87 & 41.18 & 77.48 & 44.45 & 195 & 12.98 \\
        Solution w/ \eqref{eq:slack} & 74.18 & 45.73 & 57.34 & 42.25 & 25.68 & 27.55 & 3 & 2014.94 \\
        \hline
    \end{tabular}
\end{table}

We see the results for this example computed through either re-weighting or polishing in Table~\ref{tab:ex5b} and Figure~\ref{fig:dvh5b}.
This time, the PTV and OAR subvolumes determined through the solution computed with \eqref{eq:init} does not lead to a feasible solution for $(\mathcal{P}_1)$, though the subvolumes chosen by the other three methods do.
Additionally, the method with \eqref{eq:convex} delivers the lowest dose to the OARs, at the expense of a higher objective value and computation time, while the solution found using $(\mathcal{P}_2)$ and re-weighting produces the lowest objective value in the least amount of time.

\begin{table}
    \centering
    \footnotesize
    \caption{Results for Section~\ref{sec:compareMultiple} with dose--volume constraints.\label{tab:ex5b}}
    \begin{tabular}{lcccccccc}
        \hline
        & PTV 1 & PTV 1 & PTV 2 & PTV 2 & OAR 1 & OAR 2 \\
        & D95 (Gy) & \% $<$ 81 Gy & D95 (Gy) & \% $<$ 60 Gy & \% $>$ 50 Gy & \% $>$ 30 Gy & Obj. & Time (s) \\
        \hline
        Solution w/ $(\mathcal{P}_2)^a$ & 81.17 & 4.71 & 60.32 & 3.90 & 48.39 & 29.84 & 27.71 & 138.36 \\
        Solution w/ $(\mathcal{P}_2)^b$ & 81.00 & 4.46 & 60.01 & 3.68 & 44.31 & 28.78 & 28.64 & 309.83 \\
        Solution w/ \eqref{eq:convex} & 81.00 & 3.07 & 60.01 & 2.69 & 33.41 & 25.02 & 36.02 & 1118.13 \\
        \hline
    \end{tabular}
    \begin{flushleft}\vspace{2.5ex}\hspace{0.5em}{\footnotesize $^a$Re-weighted \\ \hspace{0.5em} $^b$Polished}\end{flushleft}
\end{table}

\subsubsection{Two PTVs and four OARs.}
\label{sec:compareMultiple2}

For our last example, we use the same PTVs, OARs, and prescriptions as the previous section, with the exception of the bladder, where the upper dose--volume constraint is changed so that at most 40\% of the organ volume may receive more than 40 Gy.
Moreover, we add the left and right femoral heads.
For illustration purposes, these new OARs have maximum-dose constraints of 30 Gy, expressed as upper dose--volume constraints that no more than 0\% of the OAR volumes receive more than 30 Gy.
The left and right femoral heads have 5957 and 5974 voxels, respectively.

Tables~\ref{tab:ex6a1} and \ref{tab:ex6a2} and Figure~\ref{fig:dvh6a} contain the results computed using approximate dose--volume objectives.
While the solution from \eqref{eq:slack} achieves the lowest OAR doses, meeting the dose--volume constraints for both the rectum and bladder, the method scales poorly.
In particular, the addition of the femoral heads leads to an unreasonably long computation time.
The other two methods produce solutions with similar PTV D95 doses, though neither do well at approximating the lower dose--volume constraints on the PTVs.
On the other hand, the solution from Algorithm~\ref{algo:bcd} requires less time and delivers a lower dose to the OARs than the method with \eqref{eq:iter}.

\begin{table}
    \centering
    \footnotesize
    \caption{PTV results for Section~\ref{sec:compareMultiple2} with dose--volume objectives.\label{tab:ex6a1}}
    \begin{tabular}{lccccc}
        \hline
        & PTV 1 & PTV 1 & PTV 2 & PTV 2 \\
        & D95 (Gy) & \% $<$ 81 Gy & D95 (Gy) & \% $<$ 60 Gy & Iter. ($k$) \\
        \hline
        Solution w/ $(\mathcal{P}_2)$ & 76.83 & 58.18 & 58.25 & 53.77 & 2 \\
        Solution w/ \eqref{eq:iter} & 76.86 & 43.68 & 58.13 & 46.09 & 97 \\
        Solution w/ \eqref{eq:slack} & 74.17 & 45.21 & 57.78 & 42.80 & 2 \\
        \hline
        \end{tabular}
\end{table}

\begin{table}
    \centering
    \footnotesize
    \caption{OAR results for Section~\ref{sec:compareMultiple2} with dose--volume objectives.\label{tab:ex6a2}}
    \begin{tabular}{lccccc}
        \hline
        & OAR 1 & OAR 2 & OAR 3 & OAR 4 \\
        & \% $>$ 50 Gy & \% $>$ 40 Gy & \% $>$ 30 Gy & \% $>$ 30 Gy & Time (s) \\
        \hline
        Solution w/ $(\mathcal{P}_2)$ & 65.61 & 52.89 & 2.37 & 2.38 & 2.16 \\
        Solution w/ \eqref{eq:iter} & 73.89 & 53.02 & 3.88 & 3.63 & 7.38 \\
        Solution w/ \eqref{eq:slack} & 21.06 & 19.37 & 1.65 & 1.79 & 29246.23 \\
        \hline
    \end{tabular}
\end{table}

Finally, Table~\ref{tab:ex6b} and Figure~\ref{fig:dvh6b} include the results computed using either re-weighting or polishing.
In this case, the solver for the convex approximation of the dose--volume constraints in \eqref{eq:convex} fails to find a feasible solution.
The other three methods lead to beamlets that satisfy all of the dose--volume constraints.
The subvolumes chosen by the initialization in \eqref{eq:init} produces the solution with the lowest OAR doses, at the expense of a higher objective value and computation time, while the solutions using $(\mathcal{P}_2)$ with re-weighting and polishing require the least amount of time and achieve the lowest objective value, respectively.
Overall, these results demonstrate that, compared to other convex and nonconvex methods, our approach can efficiently deliver similar or better solutions to the FMO problem in the presence of competing PTV and OAR dose--volume constraints.

\begin{table}
    \centering
    \footnotesize
    \caption{Results for Section~\ref{sec:compareMultiple2} with dose--volume constraints.\label{tab:ex6b}}
    \begin{tabular}{lcccccccc}
        \hline
        & PTV 1 & PTV 1 & PTV 2 & PTV 2 & OAR 1 & OAR 2 \\
        & D95 (Gy) & \% $<$ 81 Gy & D95 (Gy) & \% $<$ 60 Gy & \% $>$ 50 Gy & \% $>$ 40 Gy & Obj. & Time (s) \\
        \hline
        Solution w/ $(\mathcal{P}_2)^a$ & 81.05 & 4.93 & 60.18 & 3.97 & 46.68 & 38.12 & 23.67 & 215.38 \\
        Solution w/ $(\mathcal{P}_2)^b$ & 81.00 & 4.89 & 60.01 & 4.10 & 48.27 & 37.58 & 15.25 & 417.34\\
        Solution w/ \eqref{eq:init} & 81.01 & 4.17 & 60.26 & 1.19 & 37.49 & 32.75 & 52.90 & 474.25 \\
        \hline
        \end{tabular}
    \begin{flushleft}\vspace{2.5ex}\hspace{0.5em}{\footnotesize $^a$Re-weighted \\ \hspace{0.5em} $^b$Polished}\end{flushleft}
\end{table}

\section{Conclusions and Future Research}
\label{sec:conclusion} 

The main challenge of radiation treatment planning is to deliver a prescribed dose to the tumor while sparing surrounding healthy tissues.
The fluence map optimization problem for intensity-modulated radiation therapy can be formulated as a large-scale inverse problem with competing objectives and constraints on the planning target volumes and organs-at-risk.
Unfortunately, the clinically relevant dose--volume constraints are nonconvex, so solving treatment planning problems exactly with dose--volume constraints is NP-hard.
We proposed a new approach that is well adapted to handle nonconvex dose--volume constraints without relying on convex relaxations.
To solve this formulation, we developed a customized algorithm that is guaranteed to converge to a stationary point of the proposed model.
The overall approach is scalable, handles multiple objectives on PTVs and OARs, and returns actionable treatment plans as demonstrated with examples on the CORT dataset. 

While this paper focuses primarily on the methodology and feasibility of our new approach, future work on this project will be directed towards practical considerations.
We will test our model on additional datasets including different tumor types, patient geometries, and prescriptions.
The comparison of the results using our approach and commercial treatment planning systems is left for future work.

\section*{Acknowledgements}
This work is in part supported by NSF Grant DGE-1258485 AM002 and the Washington Research Foundation Data Science Professorship.
The first author acknowledges the generous support of the University of Washington Applied Mathematics Boeing Fellowship.

\bibliographystyle{unsrt}
\bibliography{references}

\appendix
\begin{section}{Availability of Materials} 
The figures in this article, as well as the plotting scripts necessary to reproduce them, are available openly
under the CC-BY license \cite{maass2021figures}. All examples are performed on the CORT prostate dataset \cite{craft2014shared, craft2014shared2}. Solutions to nonnegative least-squares problems are solved with the function {\tt minConf\_TMP} from the minConf package \cite{schmidt2008minconf, schmidt2009optimizing}, and solutions for \eqref{eq:convex} in Section~\ref{sec:compare} were solved with CVX, a package for specifying and solving convex programs \cite{cvx, gb08}.
\end{section}

\begin{section}{Proofs for Section~\ref{sec:approach}}
\label{sec:proofs}

In this section, we provide proofs for Theorems~\ref{thm:g_converge}~and~\ref{thm:w_converge}, which follow from several auxiliary lemmas. 
For simplicity, we let $m$ be the number of beamlets and $N:= \sum_{j \in \mathcal{J}^\text{dv}}n_j$ be the number of dose--volume constrained voxels, so that $x \in \mathbb{R}^m$ and $w \in \mathbb{R}^N$.
We begin by proving that all fixed points of Algorithm~\ref{algo:pg} are local minimizers of $g$ on $\Omega$.

\begin{lemma}
    The objective $f(x,w)$ is strongly convex on $\mathbb{R}^m \times \mathbb{R}^N$.\label{lem:f_convex}
\end{lemma}

\begin{proof}
    For simplicity, we consider the case with one PTV and one OAR, but the arguments used generalize to include additional terms.
    Rearranging variables, the objective $f(x,w)$ can be expressed using a single quadratic term,
    \begin{align}
        f(x, w) &= \frac{\alpha_1}{2n_1} \left\| A_1 x - d_1 \right\|_2^2 + \frac{\alpha_2}{2n_2} \left\| w - \left[ A_2 x - d_2^\text{dv} \right] \right\|_2^2 + \frac{\lambda}{2} \| x \|_2^2, \\
        \begin{split}
            &= \frac{1}{2} \left\| \begin{bmatrix} \sqrt{\frac{\alpha_1}{n_1}} A_1 &\quad 0 \end{bmatrix} \begin{bmatrix} x \\ w \end{bmatrix} - \sqrt{\tfrac{\alpha_1}{n_1}} d_1 \right\|_2^2 \\
            &\quad+ \frac{1}{2} \left\| \begin{bmatrix} \sqrt{\frac{\alpha_2}{n_2}} A_2 &\quad -\sqrt{\frac{\alpha_2}{n_2}} I \end{bmatrix} \begin{bmatrix} x \\ w \end{bmatrix} - \sqrt{\tfrac{\alpha_2}{n_2}} d_2^\text{dv} \right\|_2^2 \\
            &\quad+ \frac{1}{2} \left\| \begin{bmatrix} \sqrt{\lambda} I &\quad 0 \end{bmatrix} \begin{bmatrix} x \\ w \end{bmatrix} \right\|_2^2,
        \end{split} \\
        &= \frac{1}{2} \left\| \begin{bmatrix} \sqrt{\frac{\alpha_1}{n_1}} A_1 & 0 \\ \sqrt{\frac{\alpha_2}{n_2}} A_2 &\quad -\sqrt{\frac{\alpha_2}{n_2}} I \\ \sqrt{\lambda} I & 0 \end{bmatrix} \begin{bmatrix} x \\ w \end{bmatrix} - \begin{bmatrix} \sqrt{\frac{\alpha_1}{n_1}} d_1 \\ \sqrt{\frac{\alpha_2}{n_2}} d_2^\text{dv} \\ 0 \end{bmatrix} \right\|_2^2, \\
        &= \frac{1}{2} \left\| \tilde{A} \tilde{x} - \tilde{d} \right\|_2^2,
    \end{align}
    with symmetric Hessian matrix
    \begin{equation}
        \nabla^2 f(\tilde{x}) = \tilde{A}^T\tilde{A} = \begin{bmatrix} \frac{\alpha_1}{n_1} A_1^TA_1 + \frac{\alpha_2}{n_2} A_2^TA_2 + \lambda I &\quad -\frac{\alpha_2}{n_2} A_2^T \\[1ex] -\frac{\alpha_2}{n_2} A_2 & \frac{\alpha_2}{n_2} I \end{bmatrix}.
    \end{equation}
    
    Consider the matrix $\left(\sfrac{\alpha_2}{n_2}\right) I$ and the Schur complement of $\tilde{A}^T\tilde{A}$ in $\left(\sfrac{\alpha_2}{n_2}\right) I$,
    \begin{align}
        \tilde{A}^T\tilde{A} \,/ \left( \frac{\alpha_2}{n_2} I \right) &= \frac{\alpha_1}{n_1} A_1^TA_1 + \frac{\alpha_2}{n_2} A_2^TA_2 + \lambda I - \left( -\frac{\alpha_2}{n_2} A_2^T \right)\left( \frac{\alpha_2}{n_2} I \right) \left( -\frac{\alpha_2}{n_2} A_2 \right), \\
        &= \frac{\alpha_1}{n_1} A_1^TA_1 + \frac{\alpha_2}{n_2} A_2^TA_2 + \lambda I - \frac{\alpha_2}{n_2} A_2^TA_2, \\
        &= \frac{\alpha_1}{n_1} A_1^TA_1 + \lambda I.
    \end{align}
    Because the coefficients $\sfrac{\alpha_1}{n_1}$, $\sfrac{\alpha_2}{n_2}$, and $\lambda$ are positive, the matrices $\left(\sfrac{\alpha_2}{n_2}\right) I$ and $\tilde{A}^T\tilde{A} \, / \left[ \left(\sfrac{\alpha_2}{n_2}\right) I \right]$ are positive definite (note that this does not depend on the rank of matrices $A_1$ or $A_2$), which implies that the Hessian matrix $\tilde{A}^T\tilde{A}$ is also positive definite \cite{boyd2004convex}.
    Therefore the objective $f(x, w)$ with one PTV and one OAR is strongly convex on $\mathbb{R}^m \times \mathbb{R}^{n_2}$.
    
    For multiple PTVs and OARs, the lower-right block of the Hessian matrix becomes a block diagonal matrix with matrices $\left(\sfrac{\alpha_j}{n_j}\right) I$ for $j \in \mathcal{J}^\text{dv}$ along the diagonal, and the Schur complement becomes the matrix $\sum_{i \in \mathcal{I}} A_i^TA_i + \lambda I$.
    Because the coefficients $\sfrac{\alpha_i}{n_i}$ for $i \in \mathcal{I}$, $\sfrac{\alpha_j}{n_j}$ for $j \in \mathcal{J}^\text{dv}$, and $\lambda$ are positive, the arguments above can again be used to show that the objective $f(x, w)$ with multiple PTVs and OARs is strongly convex on $\mathbb{R}^m \times \mathbb{R}^N$.
\end{proof}

\begin{lemma}
    The objective $g(w)$ is strongly convex on $\mathbb{R}^N$.\label{lem:g_convex}
\end{lemma}

\begin{proof}
    Because the objective $f(x, w)$ is strongly convex on $\mathbb{R}^m \times \mathbb{R}^N$ [Lemma~\ref{lem:f_convex}], there exists a constant $\mu > 0$ such that for all vector pairs $(x, w), \left(x', w'\right) \in \mathbb{R}^m \times \mathbb{R}^N$ and constants $t \in [0, 1]$, we have
    \begin{equation}
        f\left[tx + (1-t)x', tw + (1-t)w'\right] \leq tf(x, w) + (1-t)f\left(x', w'\right) - \frac{\mu}{2}t(1-t)\left( \left\| x - x' \right\|_2^2 + \left\| w - w' \right\|_2^2 \right).
    \end{equation}
    Let $x(w) \in \argmin_{x \geq 0}f(x, w)$. Since $f(x, w)$ is a strongly convex function and the nonnegative orthant is a nonempty closed convex set, the vector $x(w) \in \mathbb{R}_{\geq0}^m$ exists and is uniquely defined for all vectors $w \in \mathbb{R}^N$ \cite{boyd2004convex}, and $g(w) = \min_{x \geq 0}f(x, w) = f(x(w), w)$.
    Furthermore, for all vectors $w, w' \in \mathbb{R}^N$ and constants $t \in [0, 1]$, we have
    \begin{align}
        g\left[tw + (1-t)w'\right] &= \min_{x\geq0} f\left[x, tw + (1-t)w'\right], \\
        &\leq f\left[ tx(w) + (1-t)x\left(w'\right), tw + (1-t)w' \right],\label{ineq:nonneg} \\
        \begin{split}
            &\leq t f\left[x(w), w\right] + (1-t) f\left[x\left(w'\right), w'\right] \\
            &\quad- \frac{\mu}{2}t(1-t)\left( \left\| x(w) - x\left(w'\right) \right\|_2^2 + \left\| w - w' \right\|_2^2 \right),
        \end{split} \\
        &= tg(w) + (1-t)g\left(w'\right) - \frac{\mu}{2}t(1-t)\left( \left\| x(w) - x\left(w'\right) \right\|_2^2 + \left\| w - w' \right\|_2^2 \right), \\
        &\leq tg(w) + (1-t)g\left(w'\right) - \frac{\mu}{2}t(1-t)\left\| w - w' \right\|_2^2,
    \end{align}
    where line \eqref{ineq:nonneg} follows from the fact that the nonnegative orthant is a convex set and $x(w) \geq 0$ for all vectors $w \in \mathbb{R}^N$, so $tx(w) + (1-t)x\left(w'\right) \geq 0$ for all vectors $w, w' \in \mathbb{R}^N$ and constants $t \in [0, 1]$.
    Therefore $g(w)$ is strongly convex on $\mathbb{R}^N$.
\end{proof}

\begin{theorem}
    The fixed points of Algorithm~\ref{algo:pg} are local minimizers of $g$ on $\Omega$.\label{thm:fixedpoint}
\end{theorem}

\begin{proof}
    Fixed points $w^* = (w_{j_1}^*, w_{j_2}^*, \dots)$ for $j_i \in \mathcal{J}^\text{dv}$ of Algorithm~\ref{algo:pg} satisfy
    \begin{equation}
        w_j^* \in \text{proj}_{\Omega_j}\left[w_j^* - t_j \frac{\partial g}{\partial w_j} \left(w^*\right)\right].
    \end{equation}
    Based on the form of the projection onto $\Omega_j$, for each element $i$ of a stationary vector $w_j^*$, we know that
    \begin{enumerate}
        \item if $(w_j^*)_i < 0$, then $[\sfrac{\partial g}{\partial w_j}\left(w^*\right)]_i = (w_j^*)_i - (A_j x(w^*) - d_j^\text{dv})_i = 0$, and the dose to voxel $i$ is less than $d_j^\text{dv}$,
        \item if $(w_j^*)_i > 0$, then $[\sfrac{\partial g}{\partial w_j}(w^*)]_i = 0$, $(w_j^*)_i$ is one of the $\lfloor \sfrac{n_j p_j}{100} \rfloor$ largest positive elements of vector $w_j^*$, and the dose to voxel $i$ is greater than $d_j^\text{dv}$, and
        \item if $(w_j^*)_i = 0$, then either $[\sfrac{\partial g}{\partial w_j}(w^*)]_i = 0$, and the dose to voxel $i$ is exactly $d_j^\text{dv}$, or $-\sfrac{1}{t_j} [\sfrac{\partial g}{\partial w_j}(w^*)]_i > 0$ is less than the $\lfloor \sfrac{n_j p_j}{100} \rfloor$ largest positive elements of $w_j^*$, and the dose to voxel $i$ is greater than $d_j^\text{dv}$ (i.e., voxel $i$ violates the dose--volume constraint on OAR $j$).
    \end{enumerate}
    
    Recall that for each $j \in \mathcal{J}^\text{dv}$, the set $\Omega_j = \left\{ w \in \mathbb{R}^{n_j} \,:\, \left\| w_+ \right\|_0 \leq \sfrac{n_j p_j}{100} \right\}$ is the union of finitely many convex sets, each representing a different combination of the $q_j := \lfloor \sfrac{n_j p_j}{100} \rfloor$ elements of the vector $w_j$ allowed to be positive.
    Specifically, for $\ell = 1,\dots,\sfrac{n_j!}{\left(n_j - q_j\right)!q_j!}$, let $M_j^\ell$ be the $n_j \times n_j$ identity matrix with $q_j$ of its diagonal elements set to zero.
    Then the convex set $\Omega_j^\ell := \left\{ w \in \mathbb{R}^{n_j} \,:\, M_j^\ell w \leq 0 \right\}$ represents the set of vectors that satisfy combination $\ell$, and $\Omega_j = \cup_{\ell = 1,\dots, \sfrac{n_j!}{\left(n_j - q_j\right)!q_j!}} \Omega_j^\ell$ is the union of all such convex sets.
    Consequently, the set $\Omega = \prod_{j \in \mathcal{J}^\text{dv}} \Omega_j$, the Cartesian product of the feasible sets for vectors $w_j$, is also the union of finitely many convex sets, each representing a different combination of the $Q := \sum_{j \in \mathcal{J}^\text{dv}}q_j$ elements of $w$ allowed to be positive.
    Because the objective $g$ is strongly convex on any convex subset of its domain $\mathbb{R}^N$ [Lemma~\ref{lem:g_convex}], there is a unique global minimum of $g$ relative to any of these convex subsets of $\Omega$ \cite{boyd2004convex}, each corresponding to a local minimum of $g$ on $\Omega$.
    As a result, there are finitely many local minimizers of $g$ on $\Omega$.
    
    Suppose $w^*$ is a fixed point of Algorithm~\ref{algo:pg}.
    If $w^* \leq 0$, then the vector satisfies all feasible combinations of the $Q$ elements of $w$ allowed to be positive, and thus lies within the intersection of all corresponding convex subsets of $\Omega$.
    Furthermore, from observations 1 and 3 above, we have $\nabla g\left(w^*\right) = 0$, and using the strong convexity of $g$, we can conclude that $w^*$ is the unique global minimimizer of $g$ on $\Omega$ \cite[Theorem 6.12]{rockafellar2009variational}.
    On the other hand, suppose that at least one element of $w^*$ is positive, and let $M^*$ be the $\mathbb{R}^N \times \mathbb{R}^N$ identity matrix where the diagonal elements corresponding to the at most $Q$ positive elements of $w^*$ are set to zero.
    Then $\Omega^* := \left\{ w \in \mathbb{R}^N \,:\, M^*w \leq 0 \right\}$ is a convex subset of $\Omega$ containing $w^*$.
    Again, from observations 1--3 above, we have $[\nabla g\left(w^*\right)]_i = 0$ where $\left(w^*\right)_i \neq 0$ and $[\nabla g\left(w^*\right)]_i \leq 0$ where $\left(w^*\right)_i = 0$, so for all $w \in \Omega^*$ we have
    \begin{equation}
        -\nabla g\left(w^*\right)^T\left(w - w^*\right) = -\nabla g\left(w^*\right)^T w \leq 0,
    \end{equation}
    because positive elements of $w$ correspond to zero elements of the gradient.
    Therefore $-\nabla g\left(w^*\right)^T \in N_{\Omega^*}\left(w^*\right)$ \cite[Theorem 6.9]{rockafellar2009variational}, and by the convexity of $g$ and $\Omega^*$, we know that $w^*$ is the unique global minimizer of $g$ on $\Omega^*$ \cite[Theorem 6.12]{rockafellar2009variational}, and hence a local minimizer of $g$ on $\Omega$.
\end{proof}

We now establish the convergence of Algorithm~\ref{algo:pg} to a local minimum of $g$ on $\Omega$.

\begin{lemma}
    For all vectors $w, w' \in \mathbb{R}^N$, we have
    \begin{equation}
        g(w) \leq g\left(w'\right) + \nabla g\left(w'\right)^T\left(w - w'\right) + \sum_{j \in \mathcal{J}^\text{dv}}\frac{\alpha_j}{2n_j}\left\| w - w'\right\|_2^2.
    \end{equation}\label{lem:descent}
\end{lemma}

\begin{proof}
    Following the proof from \cite[Lemma 1.2.3]{nesterov2013introductory}, for all vectors $w, w' \in \mathbb{R}^N$, we have
    \begin{align}
        g(w) &= g\left(w'\right) + \int_0^1 \nabla g\left[w' + \tau \left(w - w'\right)\right]^T\left(w - w'\right) d\tau, \\
        &= g\left(w'\right) + \nabla g\left(w'\right)^T\left(w - w'\right) + \int_0^1 \left\{ \nabla g\left[w'+ \tau\left(w - w'\right)\right] - \nabla g\left(w'\right)\right\}^T\left(w - w'\right) d\tau.
    \end{align}
    Recall that for OARs $j \in \mathcal{J}^\text{dv}$, we have $\text{lip}\left(\sfrac{\partial g}{\partial w_j}\right) \leq \frac{\alpha_j}{n_j}$.
    Therefore
    \begin{align}
        &\left| g(w) - g\left(w'\right) - \nabla g\left(w'\right)^T\left(w - w'\right) \right| \\
        &\quad= \left| \int_0^1 \left\{ \nabla g\left[w' + \tau\left(w - w'\right)\right] - \nabla g\left(w'\right)\right\}^T\left(w - w'\right) d\tau \right|, \\
        &\quad\leq \int_0^1 \left| \left\{ \nabla g\left[w' + \tau\left(w - w'\right)\right] - \nabla g\left(w'\right)\right\}^T\left(w - w'\right)\right| d\tau, \\
        &\quad= \int_0^1 \left| \sum_{j \in \mathcal{J}^\text{dv}} \left\{ \frac{\partial g}{\partial w_j} \left[w_j' + \tau \left(w_j - w_j'\right) \right] - \frac{\partial g}{\partial w_j}\left(w_j'\right) \right\}^T\left(w_j - w_j'\right) \right| d\tau, \\
        &\quad\leq \int_0^1 \sum_{j \in \mathcal{J}^\text{dv}} \left| \left\{ \frac{\partial g}{\partial w_j} \left[w_j' + \tau \left(w_j - w_j'\right) \right] - \frac{\partial g}{\partial w_j}\left(w_j'\right) \right\}^T\left(w_j - w_j'\right) \right| d\tau, \\
        &\quad \leq \int_0^1 \sum_{j \in \mathcal{J}^\text{dv}} \left\| \frac{\partial g}{\partial w_j} \left[w_j' + \tau\left(w_j - w_j'\right) \right] - \frac{\partial g}{\partial w_j} \left(w_j'\right) \right\|_2 \left\| w_j - w_j'\right\|_2 d\tau, \\
        &\quad\leq \int_0^1 \sum_{j \in \mathcal{J}^\text{dv}} \frac{\tau \alpha_j}{n_j} \left\| w_j - w_j' \right\|_2^2 d\tau = \sum_{j \in \mathcal{J}^\text{dv}} \frac{\alpha_j}{2n_j} \left\| w_j - w_j' \right\|_2^2,
    \end{align}
    which implies that
    \begin{equation}
        g(w) \leq g\left(w'\right) + \nabla g\left(w'\right)^T\left(w - w'\right) + \sum_{j \in \mathcal{J}^\text{dv}} \frac{\alpha_j}{2n_j} \left\| w_j - w_j'\right\|_2^2
    \end{equation}
    for all $w, w' \in \mathbb{R}^N$.
\end{proof}

\begin{lemma}
    For fixed step sizes $t_j \in \left(0,\sfrac{n_j}{\alpha_j}\right]$ for all $j \in \mathcal{J}^\text{dv}$, the sequence of objective values $\left\{g\left(w^{(k)}\right)\right\}_{k \in \mathbb{N}}$ generated by Algorithm~\ref{algo:pg} is nonincreasing.\label{lem:g_decreasing}
\end{lemma}

\begin{proof}
    For each $j \in \mathcal{J}^\text{dv}$, let vectors $w_j^{(k+1)}, w_j^{(k)} \in \Omega_j$ satisfy $w_j^{(k+1)} \in \text{proj}_{\Omega_j}\left[w_j^{(k)} - t_j \sfrac{\partial g}{\partial w_j}\left(w^{(k)}\right)\right]$.
    Then from the definition of the projection operator, we have
    \begin{align}
        \left\|w_j^{(k+1)} - \left[ w_j^{(k)} - t_j\frac{\partial g}{\partial w_j}\left(w^{(k)}\right)\right]\right\|_2^2 &= \min_{w_j \in \Omega_j} \left\| w_j - \left[ w_j^{(k)} - t_j \frac{\partial g}{\partial w_j} g\left(w^{(k)}\right)\right]\right\|_2^2, \\
        &\leq \left\| w_j^{(k)} - \left[ w_j^{(k)} - t_j \frac{\partial g}{\partial w_j} g\left(w^{(k)}\right)\right]\right\|_2^2, \\
        &= \left\| t_j \frac{\partial g}{\partial w_j} \left(w^{(k)}\right) \right\|_2^2,
    \end{align}
    which simplifies to
    \begin{equation}
        \frac{\partial g}{\partial w_j}\left(w^{(k)}\right)^T \left(w_j^{(k+1)} - w_j^{(k)}\right) \leq -\frac{1}{2t_j} \left\| w_j^{(k+1)} - w_j^{(k)}\right\|_2^2.\label{ineq:p3}
    \end{equation}
    Noting that $\Omega \subset \mathbb{R}^N$, we can combine inequality \eqref{ineq:p3} with Lemma~\ref{lem:descent} to get
    \begin{align}
        g\left(w^{(k+1)}\right) &\leq g\left(w^{(k)}\right) + \nabla g\left(w^{(k)}\right)^T\left(w^{(k+1)} - w^{(k)}\right) + \sum_{j \in \mathcal{J}^\text{dv}}\frac{\alpha_j}{2n_j}\left\| w^{(k+1)} - w^{(k)}\right\|_2^2, \\
        &= g\left(w^{(k)}\right) + \sum_{j \in \mathcal{J}^\text{dv}} \frac{\partial g}{\partial w_j}\left(w^{(k)}\right)^T\left(w_j^{(k+1)} - w_j^{(k)}\right) + \sum_{j \in \mathcal{J}^\text{dv}}\frac{\alpha_j}{2n_j}\left\| w^{(k+1)} - w^{(k)}\right\|_2^2, \\
        &\leq g\left(w^{(k)}\right) - \sum_{j \in \mathcal{J}^\text{dv}} \frac{1}{2t_j} \left\| w_j^{(k+1)} - w_j^{(k)} \right\|_2^2 + \sum_{j \in \mathcal{J}^\text{dv}} \frac{\alpha_j}{2n_j}\left\| w_j^{(k+1)} - w_j^{(k)}\right\|_2^2, \\
        &= g\left(w^{(k)}\right) + \sum_{j \in \mathcal{J}^\text{dv}} \left(\frac{\alpha_j}{2n_j} - \frac{1}{2t_j}\right) \left\| w_j^{(k+1)} - w_j^{(k)} \right\|_2^2.\label{ineq:descent}
    \end{align}
    Therefore we have $g\left(w^{(k+1)}\right) \leq g\left(w^{(k)}\right)$ for all vectors $w_j^{(k+1)}, w_j^{(k)}$ satisfying \\ $w_j^{(k+1)} \in \text{proj}_{\Omega_j}\left[w_j^{(k)} - t_j\frac{\partial g}{\partial w_j}\left(w^{(k)}\right)\right]$ with $t \in \left(0, \sfrac{n_j}{\alpha_j}\right]$ for all $j \in \mathcal{J}^\text{dv}$.
\end{proof}

\vspace{1ex}
\begin{proof}{(Proof of Theorem~\ref{thm:g_converge}.)}
    From Lemma~\ref{lem:g_decreasing}, we know that if $t_j \in \left(0, \sfrac{n_j}{\alpha_j}\right]$ for all $j \in \mathcal{J}^\text{dv}$, then the sequence of objective values $\left\{g\left(w^{(k)}\right)\right\}_{k \in \mathbb{N}}$ is nonincreasing, and $g(w) \geq 0$ for all $w \in \mathbb{R}^N$ by definition, so the sequence converges to some constant $g^* \geq 0$ where $g^* \leq g\left(w^{(k)}\right)$ for all $k \in \mathbb{N}$.
    Furthermore, for all convergent subsequences $w^{(p_k)} \to w_p^*$ as $k \to \infty$, we have $g\left(w_p^*\right) = g^*$.
    Because $w_p^*$ is a fixed point of Algorithm~\ref{algo:pg}, it must also be a local minimizer of $g$ on $\Omega$ [Theorem~\ref{thm:fixedpoint}], so $g^*$ is a local minimum of $g$ on $\Omega$.
\end{proof}

Finally, we prove the convergence of the iterates of Algorithm~\ref{algo:pg} to a local minimizer of $g$ on $\Omega$ for the case when $t_j < \sfrac{n_j}{\alpha_j}$ for all $j \in \mathcal{J}^\text{dv}$.

\vspace{1ex}
\begin{proof}{(Proof of Theorem~\ref{thm:w_converge}.)}
    From Theorem~\ref{thm:g_converge} we have $g\left(w^{(k)}\right) \downarrow g^*$ as $k \to \infty$, so for all $\epsilon > 0$, there exists some positive integer $K$ such that for all $k \geq K$, we have $g\left(w^{(k)}\right) - g^* < \epsilon$.
    Furthermore, from Lemma~\ref{lem:g_convex} we know that $g$ is strongly convex.
    In particular, there exists some scalar $\mu > 0$ such that
    \begin{equation}
        g(w') \geq g(w) + \nabla g(w)^T(w' - w) + \frac{\mu}{2} \| w' - w \|_2^2
    \end{equation}
    for all vectors $w, w' \in \Omega$.
    Let $w^*$ be a fixed point of Algorithm~\ref{algo:pg} that satisfies $g(w^*) = g^*$.
    Then for any $k \geq K$, we have
    \begin{align}
        \frac{\mu}{2} \left\| w^{(k)} - w^* \right\|_2^2 &\leq g\left(w^{(k)}\right) - g(w^*) - \nabla g(w^*)^T\left(w^{(k)} - w^*\right), \\
        &= g\left(w^{(k)}\right) - g^* - \nabla g(w^*)^Tw^{(k)}, \\
        &< \epsilon - \nabla g(w^*)^Tw^{(k)}, \label{eq:wDIffBound}
    \end{align}
    where $\nabla g(w^*)^Tw^* = 0$ because nonzero elements of $w^*$ correspond to zero elements of $\nabla g(w^*)$ (see observations 1 and 2 from Theorem~\ref{thm:fixedpoint}).
    We consider two cases based on the gradient of $g$ at $w^*$.
    
    \textbf{Case 1.} Suppose $\nabla g(w^*) = 0$, and set $\delta = \sqrt{\sfrac{2\epsilon}{\mu}}$.
    Then from Theorem~\ref{thm:g_converge} and inequality~\eqref{eq:wDIffBound}, we know that for all $\delta > 0$, there exists some positive integer $K$ such that for all $k \geq K$, we have
    \begin{equation}
        \left\| w^{(k)} - w^* \right\|_2  < \sqrt{\frac{2\epsilon}{\mu}} = \delta.
    \end{equation}
    Therefore $w^{(k)} \to w^*$ as $k \to \infty$, and the sequence of iterates $\left\{w^{(k)}\right\}_{k\in\mathbb{N}}$ converges to a fixed point of Algorithm~\ref{algo:pg}.
    In fact, because $g$ is strongly convex and $\nabla g(w^*) = 0$, the fixed point $w^*$ is the unique global minimizer of $g$ on $\Omega$, and $g(w^*) = g^*$ is the global minimum.
    
    \textbf{Case 2.} Suppose $[\nabla g(w^*)]_i \neq 0$ for at least one element $i$.
    From observation 3 of Theorem~\ref{thm:fixedpoint}, this will occur exactly when $-\sfrac{1}{t_j}[\sfrac{\partial g}{\partial w_j}(w^*)]_i > 0$ is less than the $q_j$ largest positive elements of $w_j^*$, which means that the subvector $w_j^*$ has reached its maximum number of positive entries.
    In other words, all elements of $\nabla g(w^*)$ are zero, except for possibly in entries belonging to subvectors $w_j^*$ that have reached their maximum number of positive elements; in this case, we will have $w_i^* = 0$ and $[\nabla g(w^*)]_i < 0$.
    
    From Theorem~\ref{thm:g_converge}, we know that all convergent subsequences of $\left\{w^{(k)}\right\}_{k \in \mathbb{N}}$ converge to vectors $w^*$ such that $g\left(w^*\right) = g^*$.
    Consider the convergent subsequence $w^{(p_k)} \to w^*$ as $k \to \infty$.
    For all $\delta > 0$, there exists some positive integer $L$ such that for all $k \geq L$, we have
    \begin{equation}
        \left\| w^{(p_k)} - w^* \right\|_2 < \delta.
    \end{equation}
    Let $\omega$ denote the smallest positive element of $w^*$, and let $\delta < \sfrac{\omega}{2}$.
    Then for all $k \geq L$, the iterates $w^{(p_k)}$ must be positive everywhere that $w^*$ is positive, including all subvectors $w_j^*$ that have reached their maximum number of positive elements.
    In particular, this means that $w^{(p_k)}$ cannot be positive where $[\nabla g(w^*)]_i < 0$, which implies that $\nabla g(w^*)^Tw^{(p_k)} \geq 0$.

    Suppose that for some $k \geq L$ and $n > 0$, we also have
    \begin{equation}
        \left\| w^{(p_k+n)} - w^* \right\|_2 < \delta.
    \end{equation}
    Then the iterate $w^{(p_k+n)}$ is positive everywhere that $w^*$ is positive, which means that $\nabla g(w^*)^Tw^{(p_k+n)} \geq 0$.
    If we can guarantee that the next iterate $w^{(p_k+n+1)}$ also satisfies $\nabla g(w^*)^Tw^{(p_k+n+1)} \geq 0$, and we let both $\sqrt{\sfrac{2\epsilon}{\mu}} < \delta$ and $p_k+n+1 \geq K$, then from inequality~$\eqref{eq:wDIffBound}$ we will have
    \begin{equation}
        \left\| w^{(p_k+n+1)} - w^* \right\|_2 < \sqrt{\frac{2\epsilon}{\mu}} < \delta.
    \end{equation}
    The only way that we could have $\nabla g(w^*)^Tw^{(p_k+n+1)} < 0$ would be if there were some element $i$ such that $w_i^{(p_k+n+1)} > 0$ where $w_i^* = 0$ and $[\nabla g(w^*)]_i < 0$, which would need to occur on a subvector $w_j^*$ that has reached its maximum number of positive elements.
    Because $w^{(p_k+n)}$ is positive everywhere that $w^*$ is positive, the subvector $w_j^{(p_k+n)}$ would also have reached its maximum number of positive elements.
    In order for $w_i^{(p_k+n+1)} > 0$ where $w_i^* = 0$ and $w_i^{(p_k+n)} \leq 0$, we would require $w_{\ell}^{(p_k+n+1)} \leq 0$ in an entry $\ell$ where $w_{\ell}^{(p_k+n)} > 0$, that is, we would need to swap the location of a positive element in the subvector.
    However, since we assumed that
    \begin{equation}
        \left\|w^{(p_k+n)} - w^* \right\|_2 < \delta < \frac{\omega}{2},
    \end{equation}
    this would require
    \begin{equation}
        \left\| w^{(p_k+n+1)} - w^{(p_k+n)} \right\|_2 > \frac{\omega}{2}.
    \end{equation}
    
    Now suppose $t_j < \sfrac{n_j}{\alpha_j}$ for all $j \in \mathcal{J}^\text{dv}$, and define $\beta := \min_{j \in \mathcal{J}^\text{dv}}(\sfrac{1}{2t_j} - \sfrac{\alpha_j}{2n_j}) > 0$.
    Then from inequality~\eqref{ineq:descent}, we have for all $k \geq K$
    \begin{align}
        g\left(w^{(k+1)}\right) &\leq g\left(w^{(k)}\right) + \sum_{j \in \mathcal{J}^\text{dv}} \left(\frac{\alpha_j}{2n_j} - \frac{1}{2t_j}\right) \left\| w_j^{(k+1)} - w_j^{(k)} \right\|_2^2 \\
        &\leq g\left(w^{(k)}\right) - \beta \left\| w^{(k+1)} - w^{(k)} \right\|_2^2, \label{ineq:descent2}
    \end{align}
    which means that
    \begin{equation}
        \left\| w^{(k+1)} - w^{(k)} \right\|_2^2 \leq \frac{1}{\beta} \left[g\left(w^{(k)}\right) - g\left(w^{(k+1)}\right)\right] \leq \frac{1}{\beta} \left[ g\left(w^{(k)}\right) - g^*\right] < \frac{\epsilon}{\beta}.
    \end{equation}
    In particular, if we add the assumption that $\sqrt{\sfrac{\epsilon}{\beta}} < \sfrac{\omega}{2}$, then
    \begin{equation}
        \left\| w^{(p_k+n+1)} - w^{(p_k+n)} \right\|_2 < \sqrt{\frac{\epsilon}{\beta}} < \frac{\omega}{2},
    \end{equation}
    and it is impossible to swap the location of the positive elements in the iterates $w^{(p_k+n)}$ and $w^{(p_k+n+1)}$.
    Therefore $\nabla g(w^*)^Tw^{(p_k+n+1)} \geq 0$, and we have
    \begin{equation}
        \left\| w^{(p_k+n+1)} - w^* \right\|_2 < \delta.
    \end{equation}

    By induction, if we choose $\delta < \sfrac{\omega}{2}$ and $\epsilon < \min\left\{ \sfrac{\mu \delta^2}{2}, \sfrac{\beta\omega^2}{4}\right\}$, then for all integers $k \geq \max\left\{p_L, K\right\}$, we have
    \begin{equation}
        \left\| w^{(k)} - w^* \right\|_2 < \delta,
    \end{equation}
    and the sequence of iterates $\left\{w^{(k)}\right\}_{k \in \mathbb{N}}$ converges to a fixed point $w^*$ of Algorithm~\ref{algo:pg} such that $g(w^*) = g^*$, where $g^*$ is a local minimum of $g$ on $\Omega$.
    
    To establish the convergence rate of Algorithm~\ref{algo:pg}, we rearrange the terms in \eqref{ineq:descent2} and sum over $k$ to get
    \begin{align}
        \sum_{\ell=0}^k \left\| w^{(\ell+1)} - w^{(\ell)} \right\|_2^2 &\leq \frac{1}{\beta}\sum_{\ell=0}^k \left[g\left(w^{(\ell)}\right) - g\left(w^{(\ell+1)}\right)\right] \\
        &= \frac{1}{\beta}\left[g\left(w^{(0)}\right) - g\left(w^{(k+1)}\right)\right] \\
        &\leq \frac{1}{\beta} \left[g\left(w^{(0)}\right) - g^*\right].
    \end{align}
    Define $\zeta := \min_{j \in \mathcal{J}^\text{dv}} \left(t_j\right) > 0$.
    For any $\epsilon > 0$, we let
    \begin{align}
        \min_{\ell = 0, \dots, k} \sum_{j \in \mathcal{J}^\text{dv}} \frac{1}{t_j^2} \left\| w_j^{(\ell + 1)} - w_j^{(\ell)} \right\|_2^2 &\leq \sum_{\ell = 0}^k \sum_{j \in \mathcal{J}^\text{dv}} \frac{\left\| w_j^{(\ell + 1)} - w_j^{(\ell)} \right\|_2^2}{(k+1) t_j^2} \\
        &\leq \sum_{\ell = 0}^k \frac{\left\| w^{(\ell + 1)} - w^{(\ell)} \right\|_2^2}{(k+1) \zeta^2} \\
        &\leq \frac{g\left(w^{(0)}\right) - g^*}{(k+1)\beta\zeta^2} < \epsilon,
    \end{align}
    and solving for $k$, we find that we reach an $\epsilon$-accurate solution after $\mathcal{O}\left(\sfrac{1}{\epsilon}\right)$ iterations.
\end{proof}
\end{section}

\newpage
\section{Figures for Section~\ref{sec:compare}
\label{sec:compareFigs}}

\begin{figure}[H]
    \centering
    \begin{overpic}[scale=0.4]{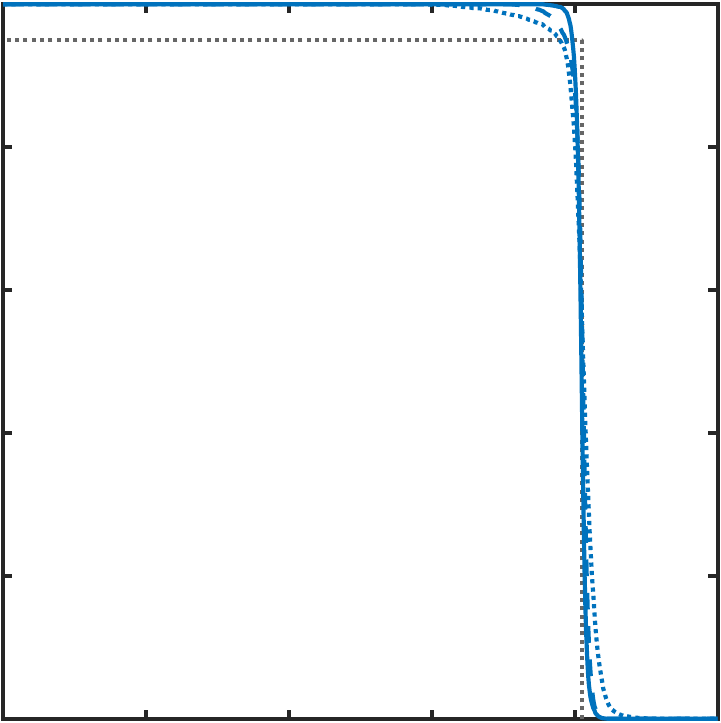}
        \put(-1,-5){\scriptsize0}
        \put(17,-5){\scriptsize20}
        \put(36.5,-5){\scriptsize40}
        \put(57,-5){\scriptsize60}
        \put(76,-5){\scriptsize80}
        \put(94,-5){\scriptsize100}
        \put(-4,-1){\scriptsize0}
        \put(-7,18){\scriptsize20}
        \put(-7,38){\scriptsize40}
        \put(-7,58){\scriptsize60}
        \put(-7,78){\scriptsize80}
        \put(-10,97){\scriptsize100}
        \put(35,-13){\small Dose (Gy)}
        \put(-20,18){\rotatebox{90}{\small Relative Volume (\%)}}
        \put(6,85){\small Prostate}
    \end{overpic} \hspace{1em}
    \begin{overpic}[scale=0.4]{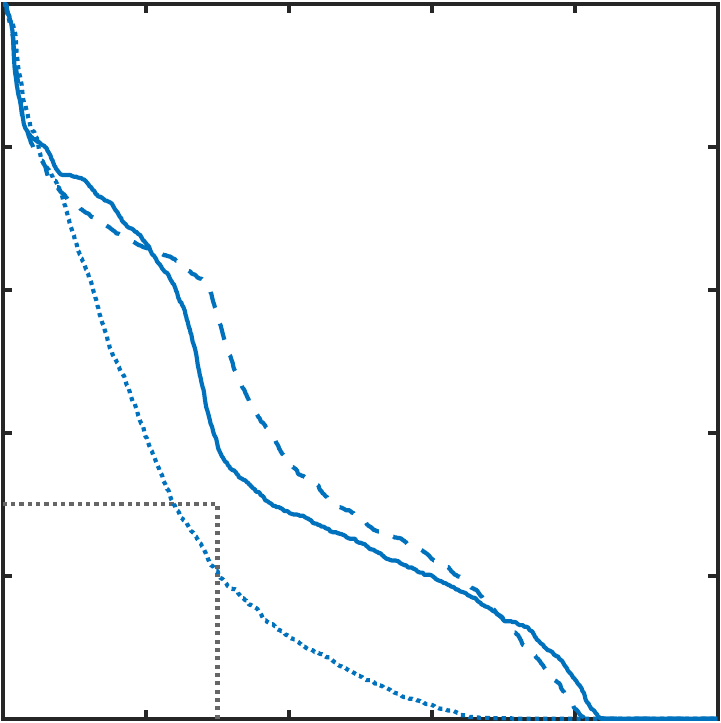}
        \put(-1,-5){\scriptsize0}
        \put(17,-5){\scriptsize20}
        \put(36.5,-5){\scriptsize40}
        \put(57,-5){\scriptsize60}
        \put(76,-5){\scriptsize80}
        \put(94,-5){\scriptsize100}
        \put(35,-13){\small Dose (Gy)}
        \put(70,90){\small Rectum}
        \put(111,90){\color{prostate}\bf-}
        \put(113.5,90){\color{prostate}\bf-}
        \put(116,90){\color{prostate}\bf-}
        \put(122,90){\scriptsize Solution w/ $(\mathcal{P}_2)$}
        \put(111,80){\color{prostate}\bf-}
        \put(116,80){\color{prostate}\bf-}
        \put(122,80){\scriptsize Solution w/ \eqref{eq:iter}}
        \put(111,71){\color{prostate}\bf...}
        \put(122,70){\scriptsize Solution w/ \eqref{eq:slack}}
    \end{overpic}
    \vspace{4ex}
    \caption{Dose--volume histograms for Section~\ref{sec:compareOne} with dose--volume objectives.
    Left: None of the solutions come close to satisfying the lower dose--volume constraint on the PTV, though the solution with $(\mathcal{P}_2)$ achieves the highest PTV D95.
    Right: The solution with \eqref{eq:slack} satisfies the upper dose--volume constraint on the OAR, with only 20.63\% of the rectum volume exceeding 30 Gy, while 38.35\% and 56.55\% of the rectum volume exceeds 30 Gy for the solutions with $(\mathcal{P}_2)$ and \eqref{eq:iter}, respectively.\label{fig:dvh4a}}
\end{figure}

\begin{figure}[H]
    \centering
    \begin{overpic}[scale=0.4]{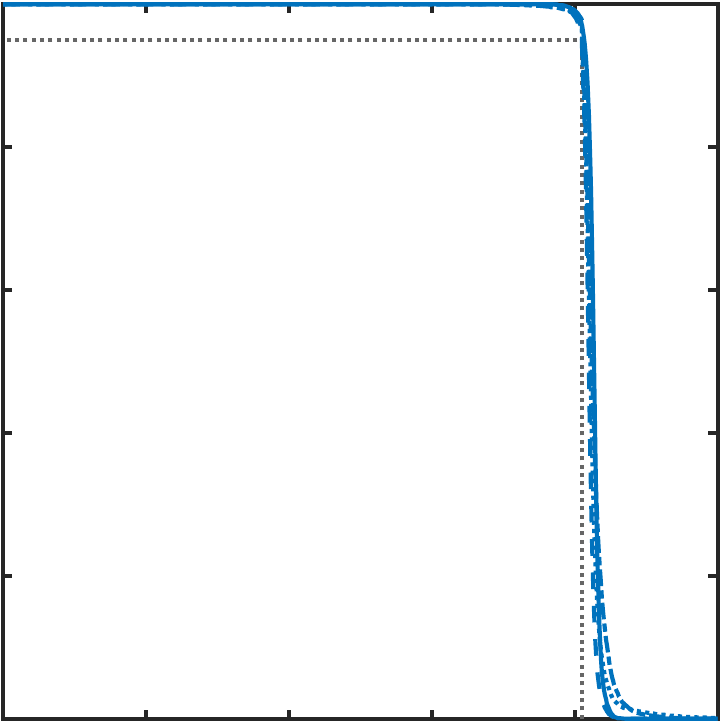}
        \put(-1,-5){\scriptsize0}
        \put(17,-5){\scriptsize20}
        \put(36.5,-5){\scriptsize40}
        \put(57,-5){\scriptsize60}
        \put(76,-5){\scriptsize80}
        \put(94,-5){\scriptsize100}
        \put(-4,-1){\scriptsize0}
        \put(-7,18){\scriptsize20}
        \put(-7,38){\scriptsize40}
        \put(-7,58){\scriptsize60}
        \put(-7,78){\scriptsize80}
        \put(-10,97){\scriptsize100}
        \put(35,-13){\small Dose (Gy)}
        \put(-20,18){\rotatebox{90}{\small Relative Volume (\%)}}
        \put(6,85){\small Prostate}
    \end{overpic} \hspace{1em}
    \begin{overpic}[scale=0.4]{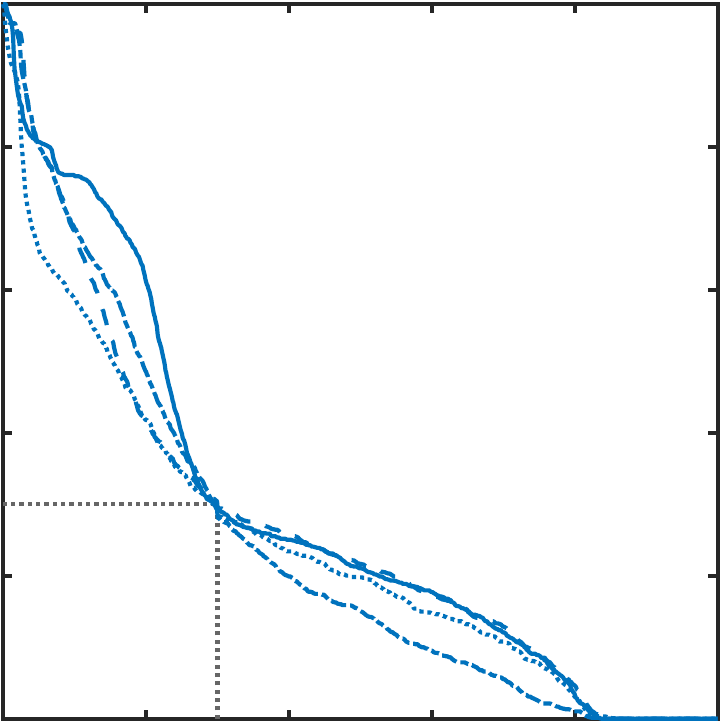}
        \put(-1,-5){\scriptsize0}
        \put(17,-5){\scriptsize20}
        \put(36.5,-5){\scriptsize40}
        \put(57,-5){\scriptsize60}
        \put(76,-5){\scriptsize80}
        \put(94,-5){\scriptsize100}
        \put(35,-13){\small Dose (Gy)}
        \put(70,90){\small Rectum}
        \put(111,90){\color{prostate}\bf-}
        \put(113.5,90){\color{prostate}\bf-}
        \put(116,90){\color{prostate}\bf-}
        \put(122,90){\scriptsize Solution w/ $(\mathcal{P}_2)^a$}
        \put(111,80){\color{prostate}\bf-}
        \put(116,80){\color{prostate}\bf-}
        \put(122,80){\scriptsize Solution w/ $(\mathcal{P}_2)^b$}
        \put(111,71){\color{prostate}\bf...}
        \put(122,70){\scriptsize Solution w/ \eqref{eq:init}}
        \put(111,61){\color{prostate}\bf.}
        \put(114,59.75){\color{prostate}\bf-}
        \put(117,61){\color{prostate}\bf.}
        \put(122,60){\scriptsize Solution w/ \eqref{eq:convex}}
        \put(119,45){\scriptsize$^a$Re-weighted}
        \put(119.5,35){\scriptsize$^b$Polished}
    \end{overpic}
    \vspace{4ex}
    \caption{Dose--volume histograms for Section~\ref{sec:compareOne} with dose--volume constraints.
    Left: All solutions meet the lower dose--volume constraint on the PTV.
    Right: All solutions meet the upper dose--volume constraint on the OAR.\label{fig:dvh4b}}
\end{figure}

\begin{figure}
    \centering
    \begin{overpic}[scale=0.4]{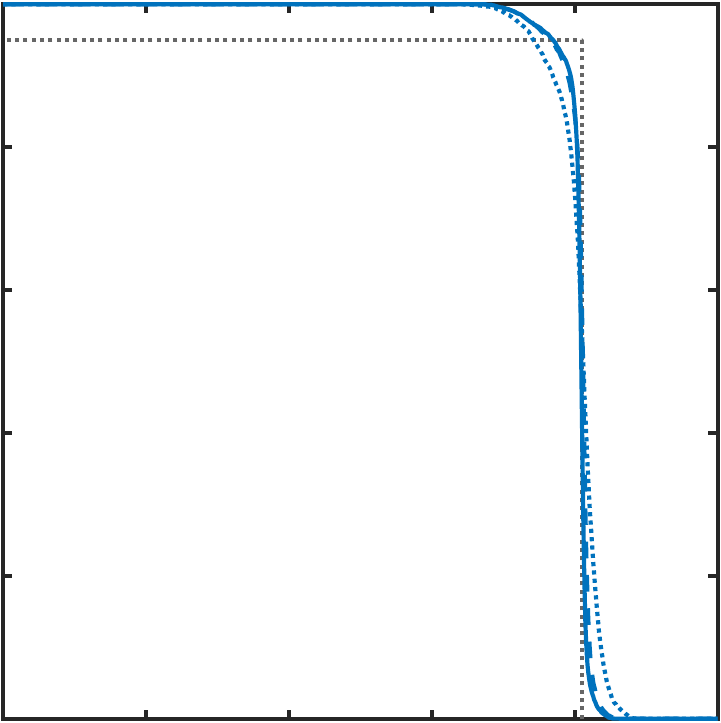}
        \put(-4,-1){\scriptsize0}
        \put(-7,18){\scriptsize20}
        \put(-7,38){\scriptsize40}
        \put(-7,58){\scriptsize60}
        \put(-7,78){\scriptsize80}
        \put(-10,97){\scriptsize100}
        \put(-20,18){\rotatebox{90}{\small Relative Volume (\%)}}
        \put(6,85){\small Prostate}
    \end{overpic} \hspace{0.5em}
    \begin{overpic}[scale=0.4]{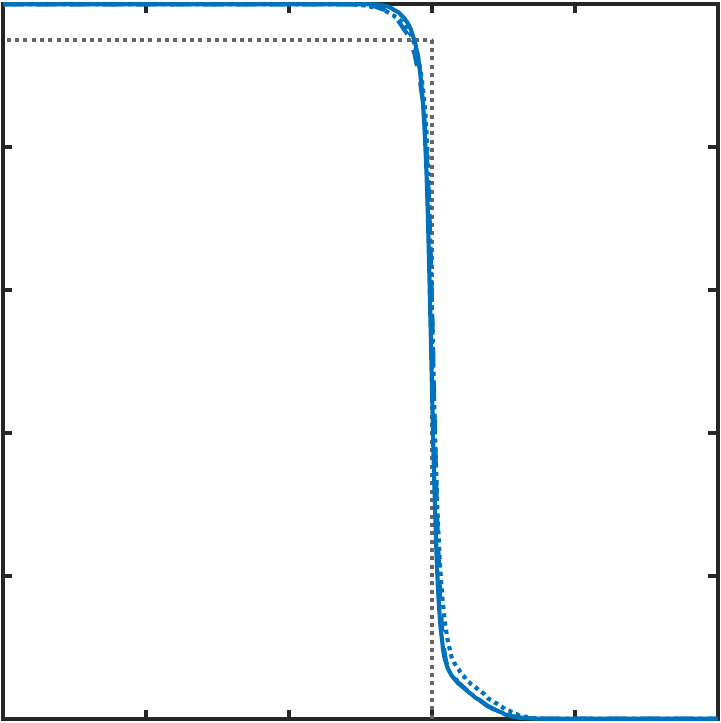}
        \put(6,85){\small Lymph Nodes}
        \put(111,90){\color{prostate}\bf-}
        \put(113.5,90){\color{prostate}\bf-}
        \put(116,90){\color{prostate}\bf-}
        \put(122,90){\scriptsize Solution w/ $(\mathcal{P}_2)$}
        \put(111,80){\color{prostate}\bf-}
        \put(116,80){\color{prostate}\bf-}
        \put(122,80){\scriptsize Solution w/ \eqref{eq:iter}}
        \put(111,71){\color{prostate}\bf...}
        \put(122,70){\scriptsize Solution w/ \eqref{eq:slack}}
    \end{overpic} \\ \vspace*{2.5ex} \hspace{0.1em}
    \begin{overpic}[scale=0.4]{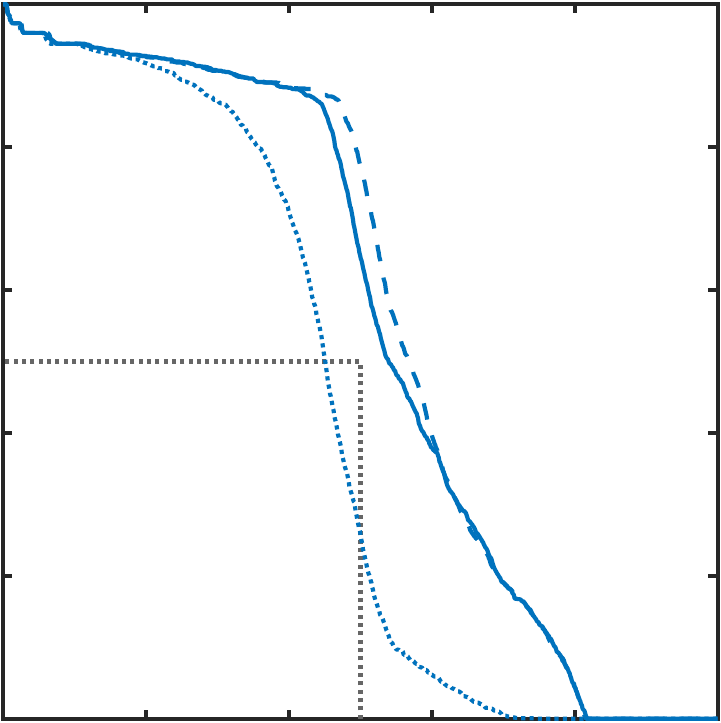}
        \put(-1,-5){\scriptsize0}
        \put(17,-5){\scriptsize20}
        \put(36.5,-5){\scriptsize40}
        \put(57,-5){\scriptsize60}
        \put(76,-5){\scriptsize80}
        \put(94,-5){\scriptsize100}
        \put(-4,-1){\scriptsize0}
        \put(-7,18){\scriptsize20}
        \put(-7,38){\scriptsize40}
        \put(-7,58){\scriptsize60}
        \put(-7,78){\scriptsize80}
        \put(-10,97){\scriptsize100}
        \put(35,-13){\small Dose (Gy)}
        \put(-20,18){\rotatebox{90}{\small Relative Volume (\%)}}
        \put(70,90){\small Rectum}
    \end{overpic} \hspace{0.5em}
    \begin{overpic}[scale=0.4]{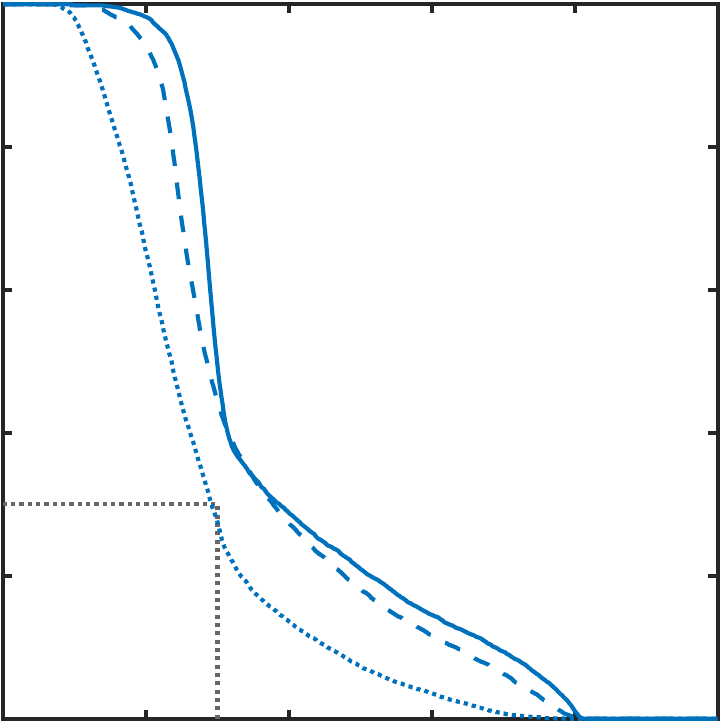}
        \put(-1,-5){\scriptsize0}
        \put(17,-5){\scriptsize20}
        \put(36.5,-5){\scriptsize40}
        \put(57,-5){\scriptsize60}
        \put(76,-5){\scriptsize80}
        \put(94,-5){\scriptsize100}
        \put(35,-13){\small Dose (Gy)}
        \put(70,90){\small Bladder}
    \end{overpic}
    \vspace{4ex}
    \caption{Dose--volume histograms for Section~\ref{sec:compareMultiple} with dose--volume objectives.
    Top row: None of the solutions come close to satisfying the lower dose--volume constraints on the PTVs, though the solution with $(\mathcal{P}_2)$ achieves the highest PTV D95 values.
    Bottom row: The solution with \eqref{eq:slack} satisfies the upper dose--volume constraints on the OARs, with only 25.68\% of the rectum volume exceeding 50 Gy and 27.55\% of the bladder volume exceeding 30 Gy, while 64.58\% and 77.48\% of the rectum volume exceeds 50 Gy and 49.47\% and 44.45\% of the bladder volume exceeds 30 Gy for the solutions with $(\mathcal{P}_2)$ and \eqref{eq:iter}, respectively.\label{fig:dvh5a}}
\end{figure}

\begin{figure}
    \centering
    \begin{overpic}[scale=0.4]{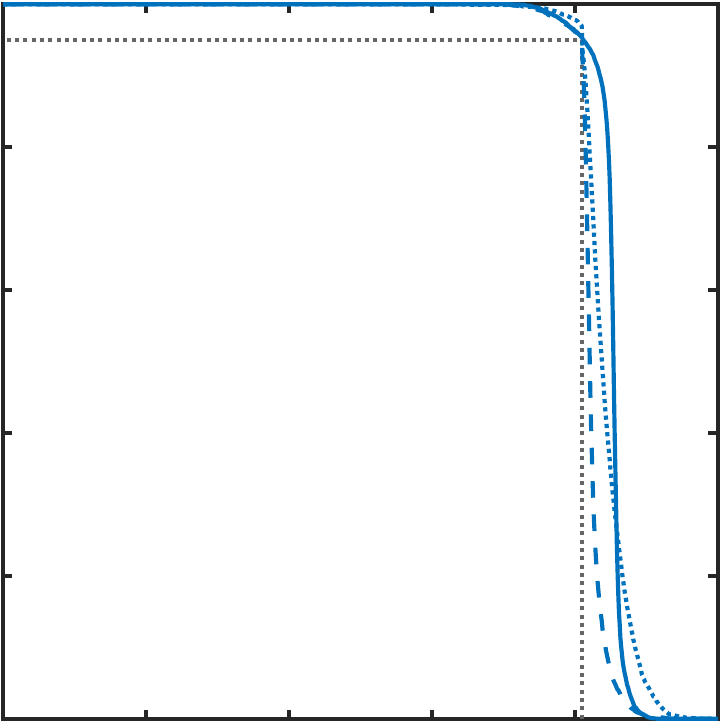}
        \put(-4,-1){\scriptsize0}
        \put(-7,18){\scriptsize20}
        \put(-7,38){\scriptsize40}
        \put(-7,58){\scriptsize60}
        \put(-7,78){\scriptsize80}
        \put(-10,97){\scriptsize100}
        \put(-20,18){\rotatebox{90}{\small Relative Volume (\%)}}
        \put(6,85){\small Prostate}
    \end{overpic} \hspace{0.5em}
    \begin{overpic}[scale=0.4]{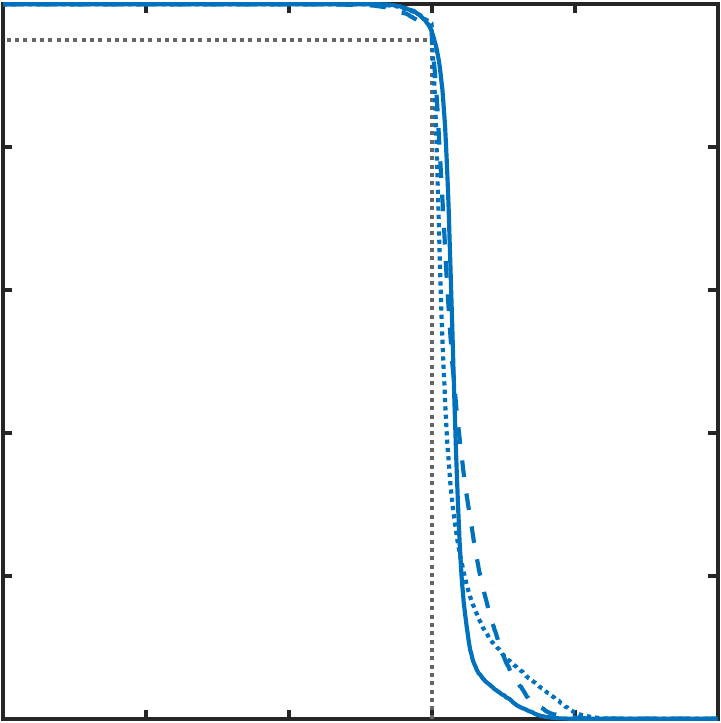}
        \put(6,85){\small Lymph Nodes}
        \put(111,90){\color{prostate}\bf-}
        \put(113.5,90){\color{prostate}\bf-}
        \put(116,90){\color{prostate}\bf-}
        \put(122,90){\scriptsize Solution w/ $(\mathcal{P}_2)^a$}
        \put(111,80){\color{prostate}\bf-}
        \put(116,80){\color{prostate}\bf-}
        \put(122,80){\scriptsize Solution w/ $(\mathcal{P}_2)^b$}
        \put(111,71){\color{prostate}\bf...}
        \put(122,70){\scriptsize Solution w/ \eqref{eq:convex}}
        \put(119,55){\scriptsize$^a$Re-weighted}
        \put(119.5,45){\scriptsize$^b$Polished}
    \end{overpic} \\ \vspace*{2.5ex} \hspace{0.1em}
    \begin{overpic}[scale=0.4]{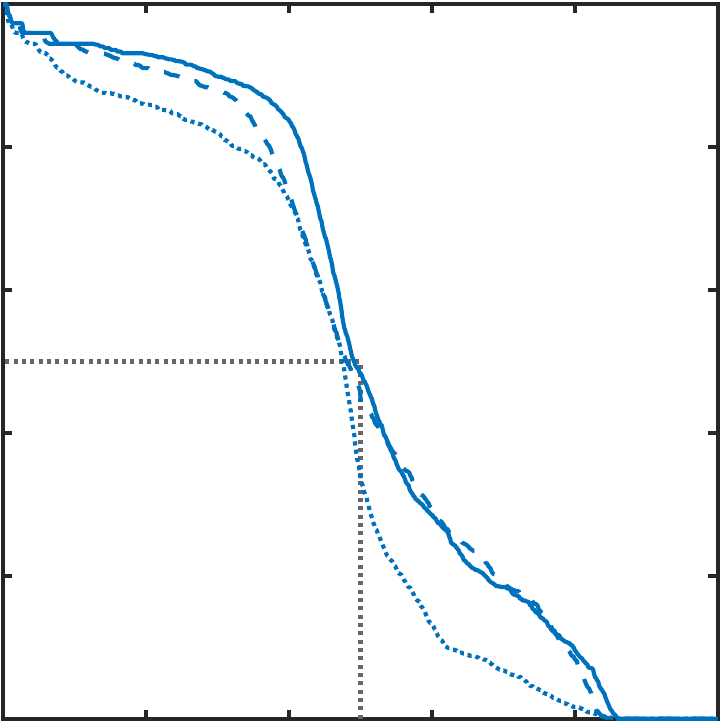}
        \put(-1,-5){\scriptsize0}
        \put(17,-5){\scriptsize20}
        \put(36.5,-5){\scriptsize40}
        \put(57,-5){\scriptsize60}
        \put(76,-5){\scriptsize80}
        \put(94,-5){\scriptsize100}
        \put(-4,-1){\scriptsize0}
        \put(-7,18){\scriptsize20}
        \put(-7,38){\scriptsize40}
        \put(-7,58){\scriptsize60}
        \put(-7,78){\scriptsize80}
        \put(-10,97){\scriptsize100}
        \put(35,-13){\small Dose (Gy)}
        \put(-20,18){\rotatebox{90}{\small Relative Volume (\%)}}
        \put(70,90){\small Rectum}
    \end{overpic} \hspace{0.5em}
    \begin{overpic}[scale=0.4]{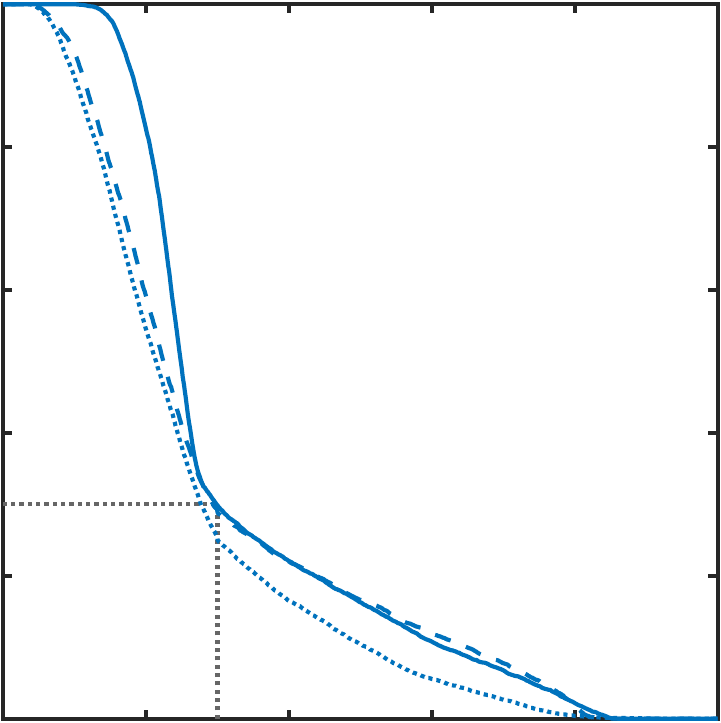}
        \put(-1,-5){\scriptsize0}
        \put(17,-5){\scriptsize20}
        \put(36.5,-5){\scriptsize40}
        \put(57,-5){\scriptsize60}
        \put(76,-5){\scriptsize80}
        \put(94,-5){\scriptsize100}
        \put(35,-13){\small Dose (Gy)}
        \put(70,90){\small Bladder}
    \end{overpic}
    \vspace{4ex}
    \caption{Dose--volume histograms for Section~\ref{sec:compareMultiple} with dose--volume constraints.
    Top row: All solutions meet the lower dose--volume constraints on the PTVs.
    Bottom row: All solutions meet the upper dose--volume constraints on the OARs.\label{fig:dvh5b}}
\end{figure}

\begin{figure}
    \centering
    \begin{overpic}[scale=0.4]{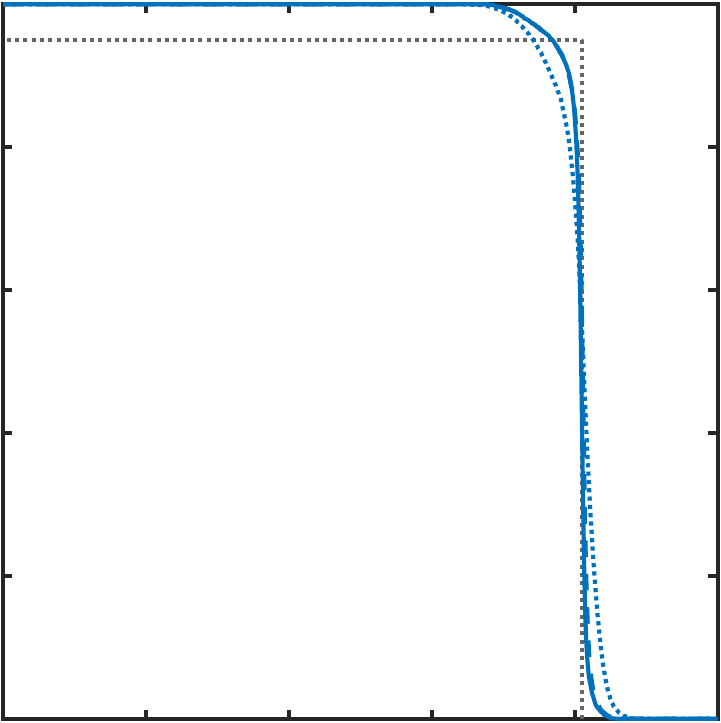}
        \put(-4,-1){\scriptsize0}
        \put(-7,18){\scriptsize20}
        \put(-7,38){\scriptsize40}
        \put(-7,58){\scriptsize60}
        \put(-7,78){\scriptsize80}
        \put(-10,97){\scriptsize100}
        \put(-20,18){\rotatebox{90}{\small Relative Volume (\%)}}
        \put(6,85){\small Prostate}
    \end{overpic} \hspace{0.5em}
    \begin{overpic}[scale=0.4]{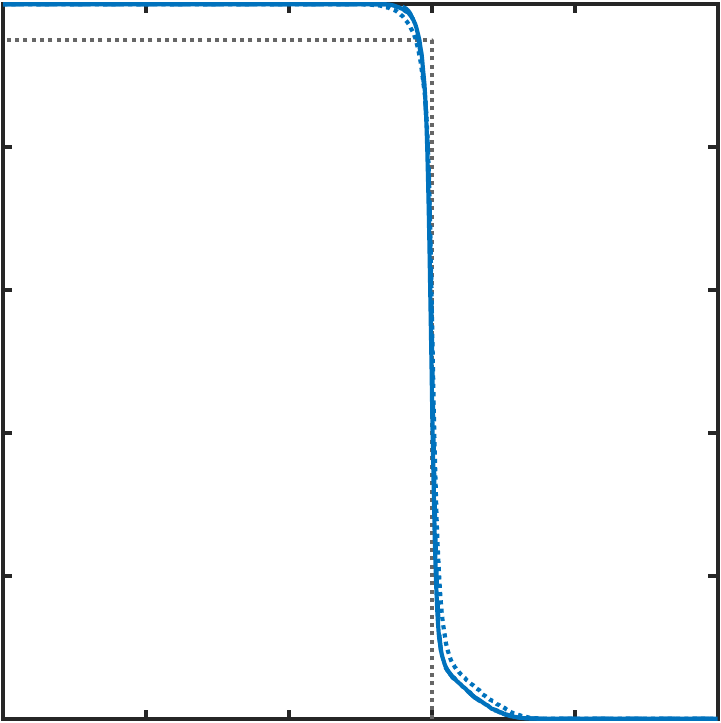}
        \put(6,85){\small Lymph Nodes}
        \put(111,90){\color{prostate}\bf-}
        \put(113.5,90){\color{prostate}\bf-}
        \put(116,90){\color{prostate}\bf-}
        \put(122,90){\scriptsize Solution w/ $(\mathcal{P}_2)$}
        \put(111,80){\color{prostate}\bf-}
        \put(116,80){\color{prostate}\bf-}
        \put(122,80){\scriptsize Solution w/ \eqref{eq:iter}}
        \put(111,71){\color{prostate}\bf...}
        \put(122,70){\scriptsize Solution w/ \eqref{eq:slack}}
    \end{overpic} \\ \vspace*{2.5ex}
    \begin{overpic}[scale=0.4]{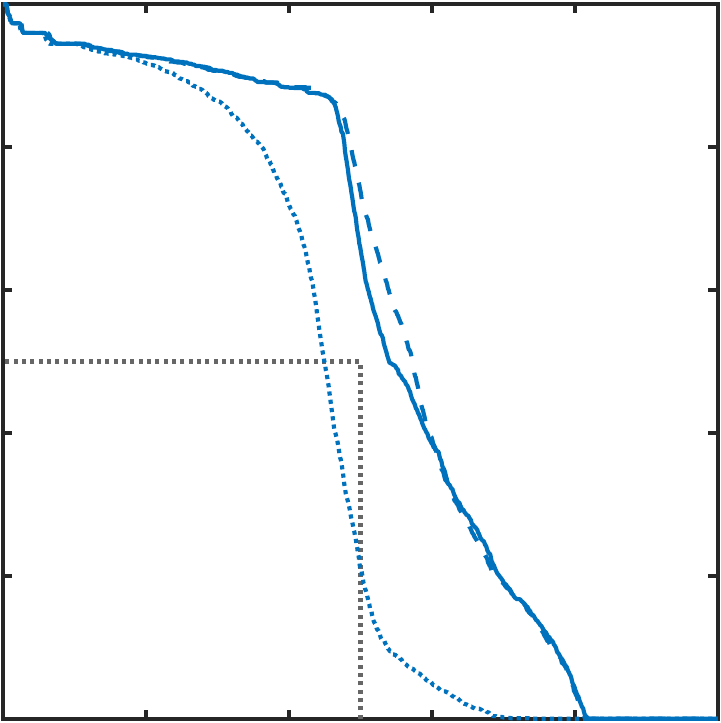}
        \put(-4,-1){\scriptsize0}
        \put(-7,18){\scriptsize20}
        \put(-7,38){\scriptsize40}
        \put(-7,58){\scriptsize60}
        \put(-7,78){\scriptsize80}
        \put(-10,97){\scriptsize100}
        \put(-20,18){\rotatebox{90}{\small Relative Volume (\%)}}
        \put(70,90){\small Rectum}
    \end{overpic} \hspace{0.5em}
    \begin{overpic}[scale=0.4]{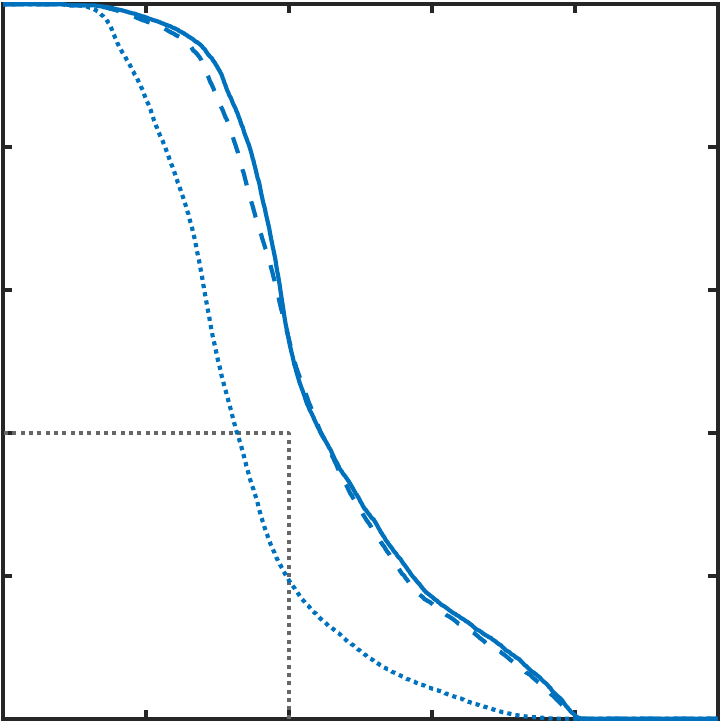}
        \put(70,90){\small Bladder}
    \end{overpic} \\ \vspace*{2.5ex} \hspace{0.1em}
    \begin{overpic}[scale=0.4]{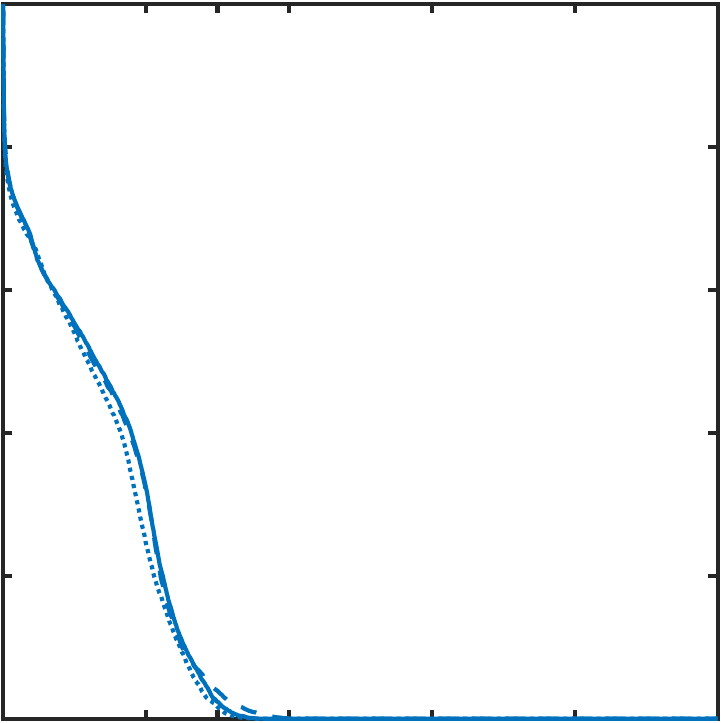}
        \put(-1,-5){\scriptsize0}
        \put(17,-5){\scriptsize20}
        \put(36.5,-5){\scriptsize40}
        \put(57,-5){\scriptsize60}
        \put(76,-5){\scriptsize80}
        \put(94,-5){\scriptsize100}
        \put(-4,-1){\scriptsize0}
        \put(-7,18){\scriptsize20}
        \put(-7,38){\scriptsize40}
        \put(-7,58){\scriptsize60}
        \put(-7,78){\scriptsize80}
        \put(-10,97){\scriptsize100}
        \put(35,-13){\small Dose (Gy)}
        \put(-20,18){\rotatebox{90}{\small Relative Volume (\%)}}
        \put(36,90){\small Left Femoral Head}
    \end{overpic} \hspace{0.5em}
    \begin{overpic}[scale=0.4]{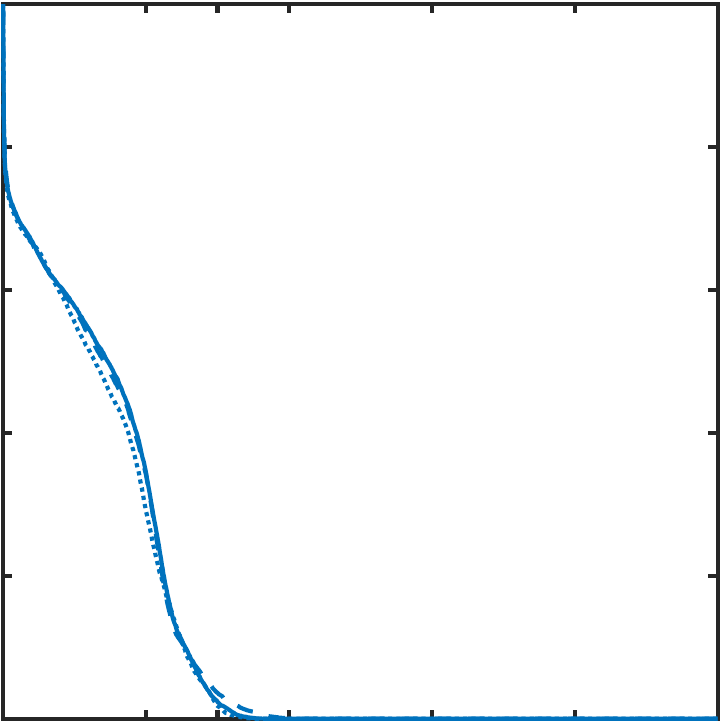}
        \put(-1,-5){\scriptsize0}
        \put(17,-5){\scriptsize20}
        \put(36.5,-5){\scriptsize40}
        \put(57,-5){\scriptsize60}
        \put(76,-5){\scriptsize80}
        \put(94,-5){\scriptsize100}
        \put(35,-13){\small Dose (Gy)}
        \put(31,90){\small Right Femoral Head}
    \end{overpic}
    \vspace{4ex}
    \caption{Dose--volume histograms for Section~\ref{sec:compareMultiple2} with dose--volume objectives.
    Top row: None of the solutions come close to satisfying the lower dose--volume constraints on the PTVs, though the solutions from $(\mathcal{P}_2)$ and \eqref{eq:iter} have higher PTV D95 values than the solution from \eqref{eq:slack}.
    Middle and bottom rows: The solution with \eqref{eq:slack} satisfies the upper dose--volume constraints on the OARs, with only 21.06\% of the rectum volume exceeding 50 Gy, 19.37\% of the bladder volume exceeding 40 Gy, and 1.65\% and 1.79\% of the left and right femoral head volumes exceeding 30 Gy, respectively.
    For the solution with $(\mathcal{P}_2)$, 65.61\% of the rectum volume exceeds 50 Gy, 52.89\% of the bladder volume exceeds 40 Gy, and 2.37\% and 2.38\% of the left and right femoral head volumes exceed 30 Gy, respectively.
    For the solution with \eqref{eq:iter}, 73.89\% of the rectum volume exceeds 50 Gy, 53.02\% of the bladder volume exceeds 40 Gy, and 3.88\% and 3.63\% of the left and right femoral head volumes exceed 30 Gy, respectively.\label{fig:dvh6a}}
\end{figure}

\begin{figure}
    \centering
    \begin{overpic}[scale=0.4]{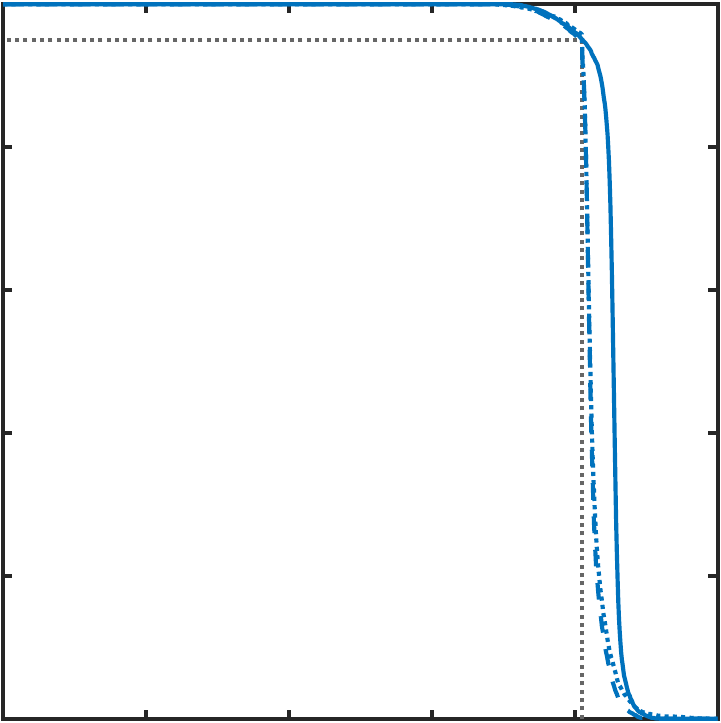}
        \put(-4,-1){\scriptsize0}
        \put(-7,18){\scriptsize20}
        \put(-7,38){\scriptsize40}
        \put(-7,58){\scriptsize60}
        \put(-7,78){\scriptsize80}
        \put(-10,97){\scriptsize100}
        \put(-20,18){\rotatebox{90}{\small Relative Volume (\%)}}
        \put(6,85){\small Prostate}
    \end{overpic} \hspace{0.5em}
    \begin{overpic}[scale=0.4]{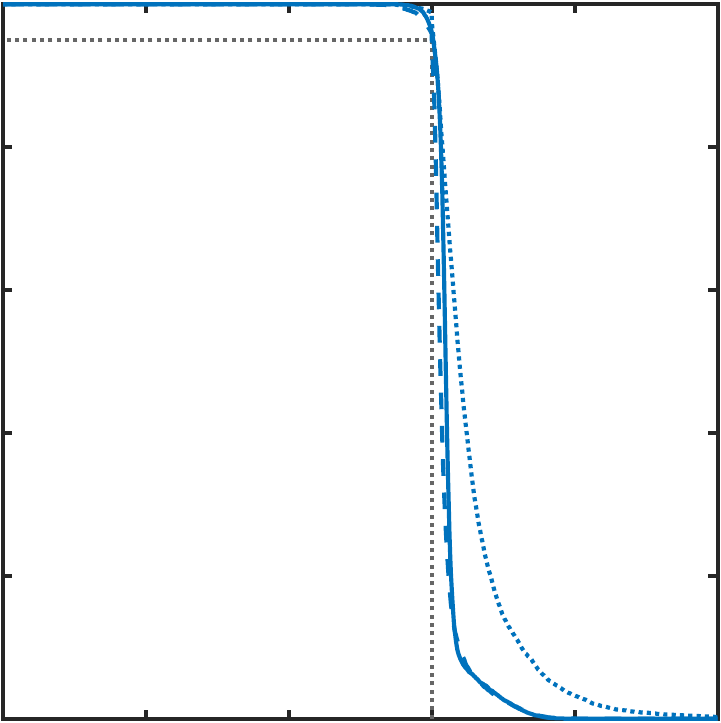}
        \put(6,85){\small Lymph Nodes}
        \put(111,90){\color{prostate}\bf-}
        \put(113.5,90){\color{prostate}\bf-}
        \put(116,90){\color{prostate}\bf-}
        \put(122,90){\scriptsize Solution w/ $(\mathcal{P}_2)^a$}
        \put(111,80){\color{prostate}\bf-}
        \put(116,80){\color{prostate}\bf-}
        \put(122,80){\scriptsize Solution w/ $(\mathcal{P}_2)^b$}
        \put(111,71){\color{prostate}\bf...}
        \put(122,70){\scriptsize Solution w/ \eqref{eq:init}}
        \put(119,55){\scriptsize$^a$Re-weighted}
        \put(119.5,45){\scriptsize$^b$Polished}
    \end{overpic} \\ \vspace*{2.5ex}
    \begin{overpic}[scale=0.4]{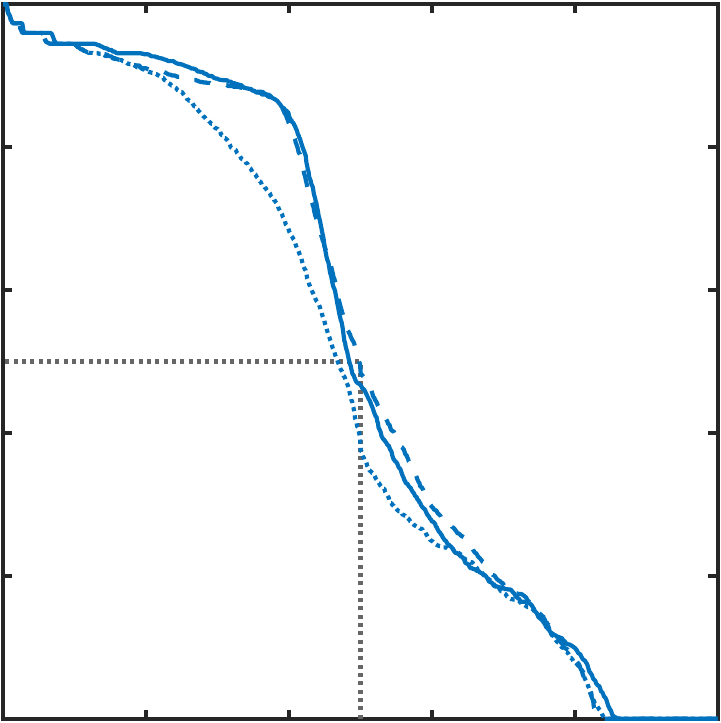}
        \put(-4,-1){\scriptsize0}
        \put(-7,18){\scriptsize20}
        \put(-7,38){\scriptsize40}
        \put(-7,58){\scriptsize60}
        \put(-7,78){\scriptsize80}
        \put(-10,97){\scriptsize100}
        \put(-20,18){\rotatebox{90}{\small Relative Volume (\%)}}
        \put(70,90){\small Rectum}
    \end{overpic} \hspace{0.5em}
    \begin{overpic}[scale=0.4]{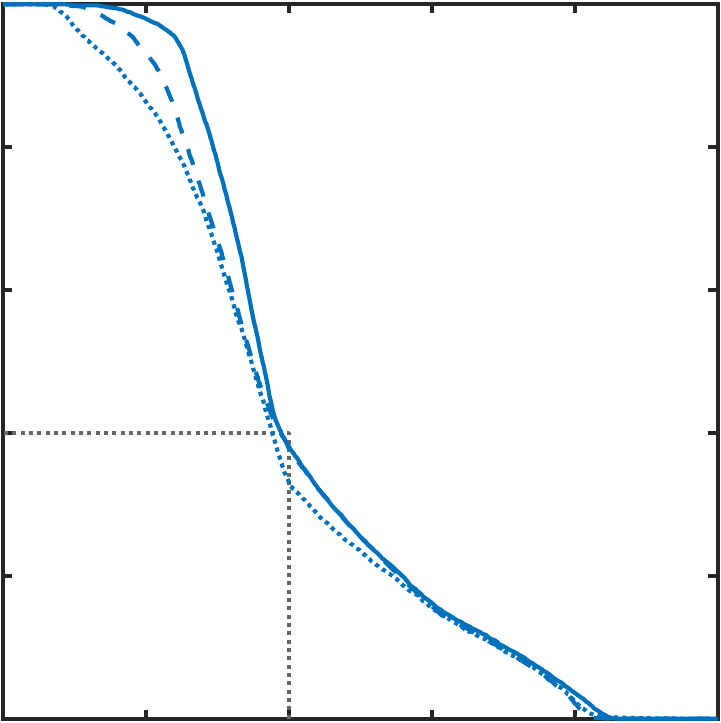}
        \put(70,90){\small Bladder}
    \end{overpic} \\ \vspace*{2.5ex} \hspace{0.1em}
    \begin{overpic}[scale=0.4]{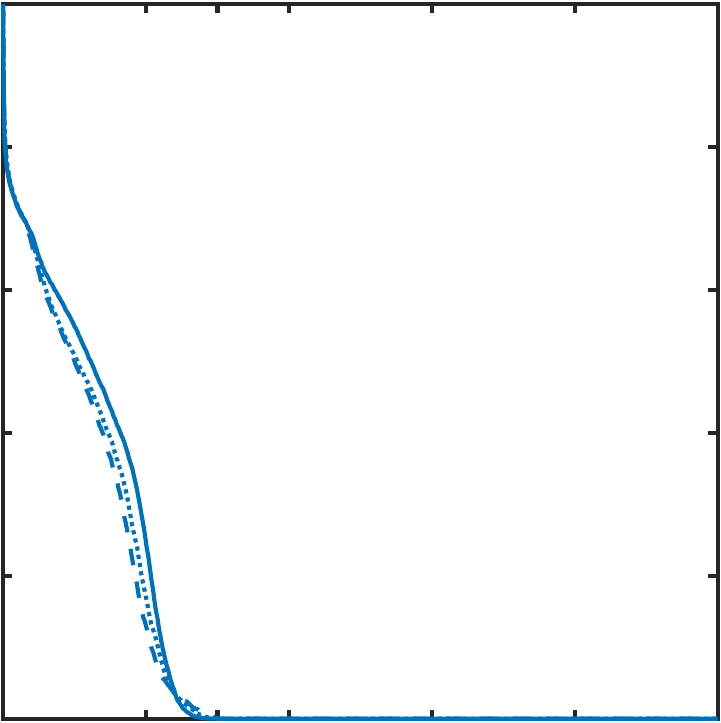}
        \put(-1,-5){\scriptsize0}
        \put(17,-5){\scriptsize20}
        \put(36.5,-5){\scriptsize40}
        \put(57,-5){\scriptsize60}
        \put(76,-5){\scriptsize80}
        \put(94,-5){\scriptsize100}
        \put(-4,-1){\scriptsize0}
        \put(-7,18){\scriptsize20}
        \put(-7,38){\scriptsize40}
        \put(-7,58){\scriptsize60}
        \put(-7,78){\scriptsize80}
        \put(-10,97){\scriptsize100}
        \put(35,-13){\small Dose (Gy)}
        \put(-20,18){\rotatebox{90}{\small Relative Volume (\%)}}
        \put(36,90){\small Left Femoral Head}
    \end{overpic} \hspace{0.5em}
    \begin{overpic}[scale=0.4]{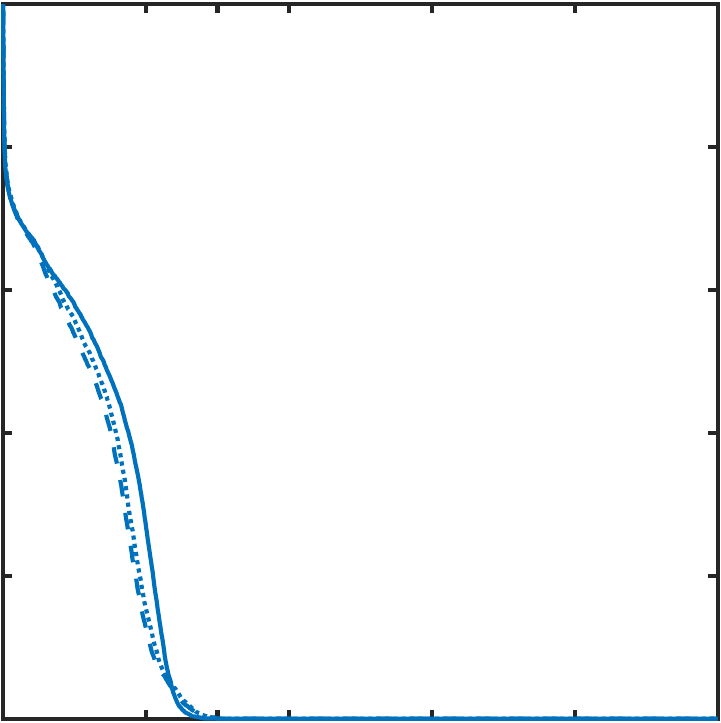}
        \put(-1,-5){\scriptsize0}
        \put(17,-5){\scriptsize20}
        \put(36.5,-5){\scriptsize40}
        \put(57,-5){\scriptsize60}
        \put(76,-5){\scriptsize80}
        \put(94,-5){\scriptsize100}
        \put(35,-13){\small Dose (Gy)}
        \put(31,90){\small Right Femoral Head}
    \end{overpic}
    \vspace{4ex}
    \caption{Dose--volume histograms for Section~\ref{sec:compareMultiple2} with dose--volume constraints.
    Top row: All solutions meet the lower dose--volume constraints on the PTVs.
    Middle and bottom rows: All solutions meet the upper dose--volume constraints on the OARs.\label{fig:dvh6b}}
\end{figure}

\end{document}